\documentclass[a4paper,12pt]{amsart}

\usepackage{latexsym}
\usepackage{amssymb}
\usepackage{graphicx}
\usepackage{wrapfig}
\usepackage{amsthm}
\usepackage{float}
\usepackage{epsfig}
\usepackage{multirow}
\usepackage[toc,page]{appendix}
\usepackage[footnotesize]{caption}
\RequirePackage{color}

\usepackage{amscd,enumerate}
\usepackage{amsfonts}

\input xy
\xyoption{all}

\usepackage{multicol}
\usepackage{hyperref}
\hypersetup{ colorlinks,
linkcolor=blue,
filecolor=green,
urlcolor=blue,
citecolor=blue }

\newtheorem{lemma}{Lemma}[section]

\newtheorem{theorem}[lemma]{Theorem}
\newtheorem{proposition}[lemma]{Proposition}
\newtheorem{remark}[lemma]{Remark}
\newtheorem{definition}[lemma]{Definition}
\newtheorem{definitions}[lemma]{Definitions}

\newcommand{\C}{{\mathcal{C}}}
\newcommand{\Ker}{{\rm{Ker}}}
\newcommand{\dsum}{\sum}
\newcommand{\dbigcup}{\bigcup}
\newcommand{\N}{{\mathbb{N}}}
\newcommand{\K}{{\mathbb{K}}}
\newcommand{\uloopr}[1]{\ar@'{@+{[0,0]+(-4,5)}@+{[0,0]+(0,10)}@+{[0,0] +(4,5)}}^{#1}}

\definecolor{turquoise2}{rgb}{0,0.898039,0.933333}
\definecolor{magenta}{rgb}{1,0,1}
\definecolor{olivedrab}{rgb}{0.419608,0.556863,0.137255}
\definecolor{purple2}{rgb}{0.568627,0.172549,0.933333}
\definecolor{amethyst}{rgb}{0.6, 0.4, 0.8}
\definecolor{ao(english)}{rgb}{0.0, 0.5, 0.0}
\definecolor{atomictangerine}{rgb}{1.0, 0.6, 0.4}
\definecolor{amber(sae/ece)}{rgb}{1.0, 0.49, 0.0}
\definecolor{alizarin}{rgb}{0.82, 0.1, 0.26}
\definecolor{auburn}{rgb}{0.43, 0.21, 0.1}
\definecolor{aqua}{rgb}{0.0, 1.0, 1.0}

\begin{document}


\title[Description of three-dimensional evolution algebras]{Description of three-dimensional evolution algebras}

\author[Y. Cabrera]{Yolanda Cabrera Casado}
\address{Y. Cabrera Casado: Departamento de \'Algebra Geometr\'{\i}a y Topolog\'{\i}a, Fa\-cultad de Ciencias, Universidad de M\'alaga, Campus de Teatinos s/n. 29071 M\'alaga.   Spain.}
\email{yolandacc@uma.es}

\author[M. Siles ]{Mercedes Siles Molina}
\address{M. Siles Molina: Departamento de \'Algebra Geometr\'{\i}a y Topolog\'{\i}a, Fa\-cultad de Ciencias, Universidad de M\'alaga, Campus de Teatinos s/n. 29071 M\'alaga.   Spain.}
\email{msilesm@uma.es}

\author[M. V. Velasco]{M. Victoria Velasco}
\address{M. V. Velasco:  Departamento de An\'{a}lisis Matem\'{a}tico, Universidad de Granada, 18071 Granada, Spain.}
\email{vvelasco@ugr.es}

\begin{abstract}
We classify  three dimensional evolution algebras over a field having characteristic different from 2 and in which there are roots of orders 2, 3 and 7.
\end{abstract}
\maketitle

\section{Introduction}
{The use of non-associative algebras to formulate Mendel's laws was started by Etherington in his papers \cite{Ethe1, Ethe2}. Other genetic algebras (those that model inheritance in genetics) called evolution algebras emerged to study non-Mendelian genetics.} Its theory in the finite-dimensional case was introduced by Tian in \cite{Tian}. The systematic study of evolution algebras of arbitrary dimension and of their algebraic properties was started in \cite{CSV1}, where the authors analyze evolution subalgebras, ideals, non-degeneracy, simple evolution algebras {and} irreducible evolution algebras. The aim of this paper is to obtain the classification of three-dimensional evolution algebras {having in mind to apply  this classification in a near future in a biological setting and to detect possible tools to implement in wider classifications}.

Two-dimensional evolution algebras over the complex numbers were {determined} in \cite{evo9}, although we have found that this classification is incomplete: {the algebra  $A$ with natural basis $\{e_1, e_2\}$ such that $e_1^2 =e_2$ and $e_2^2=e_1$} is a two-dimensional evolution algebra not isomorphic to any of the six types in \cite{evo9}. We realized of this fact when classifying the three-dimensional evolution algebras $A$ such that ${\rm dim}(A^2)=2$ and having annihilator\footnote{{The annihilator of $A$, ${\rm ann}(A)$, is defined as the set of those elements $x$ in $A$ such that $xA=0$.}} of dimension 1  {(see Tables 14-16)}.

The three dimensional case is much more complicated, as can be seen in this work, were it is proved that there are {116} types of three-dimensional evolution algebras. All of them  {are classified}  in Tables $1$-$24$.  {The} matrices  {appearing in different tables are not isomorphic (in the meaning that they do not generate the same evolution algebra). Matrices in different rows  of a same table neither are isomorphic. In} general, different values of the parameter for matrices in the same row give non-isomorphic evolution algebras, but in some case this is not true. These cases are displayed in Tables {$2'$}-$23'$.

Just after finishing this paper we found the article \cite{EL16}, where one of the aims of the authors is to classify indecomposable\footnote{Irreducible following \cite{CSV1}.} nilpotent evolution algebras up to dimension five over algebraically closed fields of characteristic not two. The three-dimensional {ones can be localized} in our classification {and for these}, it is not necessary to consider algebraically closed fields.

In this paper we deal with evolution algebras over a field $\K$ of characteristic different from 2 and in which every polynomial of the form $x^n-\alpha$, for $n=2, 3, 7$ and $\alpha\in \K$  has a root in the field. We denote by ${\phi}$ a seventh root of the unit and by $\zeta$ a third root of the unit.

\medskip
In Section \ref{ProductAndChangeOfBasis} we introduce the essential definitions. For every arbitrary finite dimensional algebra, fix a basis $B=\{e_i \ \vert \ i=1, \dots, n\}$. The product of this algebra, relative to the  basis  $B$ is determined by the matrices of the multiplication operators, $M_B(\lambda_{e_i})$ (see \eqref{productoEi}). The relationship under change of basis is also established. In the particular case of evolution algebras Theorem \ref{teor} shows this {connection}.

We start Section \ref{threedim} by analyzing the action of the group $S_3\rtimes (\K^\times)^3$ on $\mathcal M_3(\K)$. The orbits of this action will completely determine the non-isomorphic evolution algebras $A$ when ${\rm dim} (A^2)=3$ and in some cases when ${\rm dim} (A^2)=2$.

We have divided our study into four cases depending on the dimension of $A^2$, which can be 0, 1, 2 or 3. The first case is trivial. The study of the third and of the fourth ones is made by taking into account which are the possible matrices $P$ that appear as change of basis matrices. It happens that for dimension 3,  as we have said, the only matrices are those in $S_3\rtimes (\K^\times)^3$.

When the dimension of $A^2$ is 2, there exists three groups of cases (four in fact, but two of them are essentially the same). {Let $B=\{e_1, e_2, e_3\}$ be a natural basis of $A$ such that $\{e_1^2, e_2^2\}$ is a basis of $A^2$ and $e_3^2=c_1e_1^2+c_2e_2^2$ for some $c_1, c_2\in \K$}. The first   {case happens when $c_1c_2\neq 0$. Then,} $P\in S_3 \rtimes (\K^{\times})^3$. The second group of cases arises when {$c_1= 0$ and $c_2\neq 0$. Then,} the matrix $P$ is ${\rm id}_3$, $(2, 3)$,\footnote{{The matrix obtained from the identity matrix, ${\rm id}_3$, when exchanging the second and the third rows}} or the matrix $Q$ given in Case 2  {(when ${\rm dim}(A^2=2)$)}. The third one appears when {case happens when $c_1, c_2= 0$. In this case} the matrix $P$ is ${\rm id}_3$ or the matrices $Q'$ and $Q''$ given in Case 4  {(when ${\rm dim}(A^2=2)$)}.

For $P\in S_3 \rtimes (\K^{\times})^3$,  we classify taking into account: the dimension of the annihilator of $A$,  the number of non-zero entries  {in the structure matrix (which remains invariant, as it is proved in Proposition \ref{prop:numerodeceros})}, and if the algebra $A$ satisfies Property (2LI)\footnote{{For any basis $\{e_1, e_2, e_3\}$ the ideal $A^2$ has dimension two and it is generated by $\{e_i^2, e_j^2\}$, for every $i, j\in \{1, 2, 3\}$ with $i\neq j$.
}}.

For $P\in \{{\rm id}_3, (2, 3), Q\}$, we obtain a first classification, given in the different Figures. Then we compare which {matrices produce isomorphic algebras and eliminating redundancies we}  get the matrices given in  {the set $S$ that appears in Theorem \ref{thm:clasificacion}}. Again, some of these matrices give isomorphic evolution algebras. In order to classify them, we take into account that the number of non-zero entries of the matrices in {S} remains invariant under the action of the matrix $P$ (see Remark \ref{CerosInvariantesEnS}). Note that the resulting matrices correspond to evolution algebras with zero annihilator and do not satisfy Property (2LI).

For $P\in  \{{\rm id}_3, Q', Q''\}$ we classify taking into account that the third column of the structure matrix has three zero entries (the dimension of the annihilator is one and, consequently, they do not satisfy Property (2LI)) and the number of zeros in the first and the second row remains invariant under change of basis matrices (see Remark \ref{EntradasNoNulasEnFilas}).

For ${\rm dim}(A^2)=3$ we classify by the number of non-zero entries in the structure matrix.

In the case ${\rm dim} (A^2)=1$ it is not efficient to tackle the problem of the classification by obtaining the possible change of basis matrices, although {for completeness} we have determined them in Appendix \ref{apendice}. This is because we follow a different pattern. The key point for this study will be the extension property {\footnote{{There is a natural basis of $A^2$ that can be enlarged to a natural basis of $A$}} ((EP) for short).}
We have classified  taking into account the following properties: {whether or not} $A^2$ has the extension property, the dimension of the annihilator of $A$, and {whether or not} the evolution algebra $A$ has a principal\footnote{Principal means that it is generated as an ideal by one element.} two-dimensional evolution ideal which is degenerate\footnote{{An evolution algebra is non-degenerate if $e^2\neq0$ for any element $e$ in any basis (see \cite[Definition 2.16 and Corollary 2.19]{CSV1})}. Otherwise we say that it is degenerate.} as an evolution algebra (PD2EI for short).

The classification of three-dimensional evolution algebras is achieved in Theorem \ref{thm:clasificacion}. We summarize {the cases} in the tables that follow.

\medskip

\begin{center}
\scalebox{0.6}
{
\begin{tabular}{|c|c|c|c|}
\hline
&&&\\
 \rm{$A^2$ has EP}& $\rm{dim}({\rm ann}(A))$ & \rm{$A$ has  a PD2EI} & \rm{Number}\\
 &&&\\
\hline
 No & 0  & Yes & 1 \\
\hline
 No & 1  & Yes & 1 \\
\hline
 Yes & 2  & No & 1 \\
\hline
 Yes & 1  & No & 1 \\
\hline
 Yes & 0  & No & 1 \\
\hline
 Yes & 2  & Yes & 1 \\
\hline
 Yes & 1  & Yes & 1 \\
\hline
\end{tabular}
}

{\footnotesize {${\rm dim}(A^2)=1$}}
\end{center}

\medskip

\begin{center}

\scalebox{0.7}{

\begin{tabular}{|c|c|c|c|}

\hline
&&&\\
 & \rm{Non-zero entries}&  &\\
  \rm{dim(ann(A))} & * \rm{Non-zero entries in} $S$ & \rm{A has Property (2LI)} & \rm{Number}\\
&** \rm{Non-zero entries in rows 1 and 2} & & \\
&&&\\
\hline
 1& 1** & No  &  2 \\
\hline
 1& 2** & No  &  4 \\
\hline
 1& 3** & No  &  2 \\
\hline
 1& 4** & No  &  {2} \\
\hline
 0& 4* & No  &  3 \\
\hline
 0& 5* & No  &  6 \\
\hline
 0& 6* & No  &  3 \\
\hline
 0& 7* & No  &  6 \\
\hline
 0& 8* & No  &  3 \\
\hline
0 & 9* & No  & 3  \\
\hline
 0& 4 & Yes  & 4 \\
\hline
 0& 5 & Yes  & 3 \\
\hline
 0& 6 & Yes  & 7 \\
\hline
 0& 7 & Yes  & 6 \\
 \hline
 0& 8 & Yes  & 2 \\
 \hline
 0& 9 & Yes  & 1 \\
\hline
\end{tabular}}

{\footnotesize{{${\rm dim}(A^2)=2$}}}\label{tabla26}
\end{center}

\medskip

\begin{center}
\scalebox{0.7}{

\begin{tabular}{|c|c|}
\hline
&\\
  \rm{Non-zero entries} & \rm{Number}\\
   & \\
\hline
  3  & 3  \\
\hline
  4  & 6  \\
\hline
  5  & 16  \\
\hline
  6  & 15 \\
\hline
  7   & 8 \\
\hline
  8    & 2 \\
\hline
  9    & 1 \\
\hline
\end{tabular}}

{\footnotesize ${\rm dim}(A^2)=3$}\label{tabla27}

\end{center}

%
%
\section{Product and change of basis}\label{ProductAndChangeOfBasis}
%
%

In this section we study the product in an arbitrary algebra by considering the matrices associated to the product by any element in a fixed basis. We specialize to the case of evolution algebras and obtain the relationship for two structure matrices of the same evolution algebra relative to different basis.

\subsection{The product of an algebra}\label{productoAlgebra}

Let $A$ be a $\K$-algebra. Assume that  $B=\{e_{i} \ \vert \ i\in \Lambda \}$ is a basis of $A$, and let $\{\omega_{kij} \}_{i,j,k \in \Lambda} \subseteq \mathbb{K}$ be the \emph{structure constants}, i.e. $e_{i}e_{j}=\displaystyle\sum_{k\in \Lambda}\omega_{kij}e_{k}$ and $\omega_{kij}$ is zero for almost all $k$. Since in this paper we will deal only with finite dimensional evolution algebras, we will assume that $\Lambda$ is finite and has cardinal $n$.

 For any element $a\in A$ the following map defines the \emph{left multiplication operator} by $a$, denoted as ${\lambda}_a$:
$$
\begin{matrix}
{\lambda}_a: & A & \to A\\
& x & \mapsto ax
\end{matrix}
$$
Then, for every $i\in \Lambda$ we have
$$ M_{B}({\lambda}_{e_{i}})=
\left(
  \begin{array}{cccc}
    \omega_{1i1} &  \cdots & \omega_{1in} \\
    \vdots &  \ddots & \vdots \\
    \omega_{ni1} &  \cdots & \omega_{nin} \\
     \end{array}
\right),$$
\noindent
where for any linear map $T: A \to A$ we write $M_{B}( T)$ to denote the matrix in $\mathcal M_\Lambda(\K)$ associated to  $T$ relative to the basis $B$.
\medskip

 Let $A$ be an algebra and let $B=\{e_{i} \ \vert \ i\in \Lambda \}$ be a basis of $A$. For arbitrary elements $x=\sum_{i \in \Lambda}{\alpha}_{i}e_{i}$ and $y=\sum_{i \in \Lambda}{\beta}_{i}e_{i}$ in $A$ the product $xy$ is as follows:

$xy=\left(\displaystyle\sum_{i \in \Lambda}\alpha _{i}e_{i}\right)\left(\displaystyle\sum_{j \in \Lambda}\beta _{j}e_{j}\right)$ $=$ $\displaystyle\sum_{i,j \in \Lambda}$ $\alpha _{i}\beta _{j}$ $e_{i}e_{j}$ $=$ $\displaystyle\sum_{i,j \in \Lambda}$ $\left(\alpha _{i}\beta
_{j}\displaystyle\sum_{k \in \Lambda}\omega_{kij}e_{k}\right)$ $=$
$\displaystyle\sum_{k,i,j \in \Lambda}\alpha _{i}{\beta}_{j}\omega_{kij}e_{k}.$  \\ \\

Denote by  $\xi_{B}(x)$ the coordinates of an element $x$ in $A$ relative to the basis $B$, written by columns. Then:
\begin{eqnarray*}
\xi_{B}(xy)={\xi_{B}\left(\displaystyle\sum_{k,i,j \in \Lambda}\alpha _{i}{\beta}_{j}\omega_{kij}e_k\right) }& = &
 \begin{pmatrix}
     \omega_{111} &  \cdots & \omega_{11n} \\
     \vdots &  \ddots & \vdots \\
    \omega_{n11} & \cdots & \omega_{n1n} \\
         \end{pmatrix}
  \begin{pmatrix}
    {\alpha}_{1}{\beta}_{1} \\
     \vdots \\
     {\alpha}_{1}{\beta}_{n} \\
     \end{pmatrix} \\
         &  & \\
     & + & \cdots \\
           & + & \begin{pmatrix}
    \omega_{1n1}  & \cdots & \omega_{1nn} \\
    \vdots & \ddots & \vdots \\
    \omega_{nn1}  & \cdots & \omega_{nnn} \\
         \end{pmatrix}
  \begin{pmatrix}
     {\alpha}_{n}{\beta}_{1} \\
      \vdots \\
       {\alpha}_{n}{\beta}_{n} \\
     \end{pmatrix}.
     \end{eqnarray*}
That is,
\begin{equation}\label{productoEi}
\xi_{B}(xy)=\displaystyle\sum_{i \in \Lambda} M_{B}({\lambda}_{e_{i}})
\begin{pmatrix}
{\alpha}_{i}{\beta}_{1} \\
 \vdots \\
 {\alpha}_{i}{\beta}_{n}\\
 \end{pmatrix}.
 \end{equation}

 \medskip

An \emph{evolution algebra} over a field $\mathbb K$ is a $\mathbb K$-algebra $A$ provided with
a basis $B=\{e_{i} \ \vert \ i\in \Lambda \}$ such that $e_{i}e_{j}=0$
whenever $i\neq j$.
 Such a basis $B$ is called a \emph{natural basis}. Now, the structure constants of $A$ relative to $B$ are the scalars $\omega _{ki}\in \mathbb K$ such that
 $e_{i}^{2}:=e_ie_i=\sum\limits_{k\in \Lambda} \omega_{ki}e_{k}$. The matrix $M_B:= \left(\omega_{ki}\right)$ is said to be the \emph{structure matrix of} $A$ \emph{relative to} $B$.
\medskip

For any finite dimensional evolution algebra $A$ with a natural basis $B$  we have
 $$M_{B}= \sum\limits_{i \in \Lambda}M_{B}({\lambda}_{e_
{i}}).$$

In case of $A$ being an evolution algebra and $B=\{e_{i} \ \vert \ i\in \Lambda \}$ a natural basis of $A$,  the structure constants satisfy that  $\omega_{kij}=0$ for every $i,j,k \in \Lambda $ with $ i \neq j$. If we denote $\omega_{kii}=\omega_{ki}$ we obtain that:
\begin{eqnarray*}\label{noseusa}
\xi_{B}(xy) & = &
\begin{pmatrix}
\omega_{11} & 0 & \cdots & 0 \\
\vdots & \vdots & \ddots & \vdots\\
\omega_{n1} & 0 & \cdots & 0 \\
\end{pmatrix}
     \begin{pmatrix}
{\alpha}_{1}{\beta}_{1} \\
 \vdots \\
 {\alpha}_{1}{\beta}_{n}\\
 \end{pmatrix}+ \ldots +
\begin{pmatrix}
0 & \cdots & 0 & \omega_{1n}  \\
\vdots & \ddots & \vdots & \vdots \\
0 & \cdots & 0 & \omega_{nn}  \\
\end{pmatrix}
     \begin{pmatrix}
 {\alpha}_{n}{\beta}_{1} \\
 \vdots \\
 {\alpha}_{n}{\beta}_{n}\\
  \end{pmatrix} \\
  & = &
   \begin{pmatrix}
\omega_{11} &\cdots & \omega_{1n} \\
\vdots & \ddots &\vdots \\
\omega_{n1} &\cdots & \omega_{nn} \\
 \end{pmatrix}
     \begin{pmatrix}
{\alpha}_{1}{\beta}_{1} \\
\vdots \\
 {\alpha}_{n}{\beta}_{n}\\
 \end{pmatrix},
  \end{eqnarray*}
because for every $i \in \Lambda$ the matrix $M_{B}({\lambda}_{e_{i}})$ has zero entries except at most in its $i$th column.

Summarizing,
\begin{equation}\label{preproducto}
\xi_{B}(xy)= M_{B}
\begin{pmatrix}{\alpha}_{1}{\beta}_{1} \\ \vdots \\ {\alpha}_{n}{\beta}_{n} \\ \end{pmatrix}.
\end{equation}
\begin{definition}\label{PuntoGordo}
\rm
 Let $A$ be an algebra and $B=\{e_i\ \vert \ i\in \Lambda\}$ a basis of $A$. For arbitrary elements $x=\displaystyle \sum_{i \in \Lambda}  {\alpha}_{i} e_{i}$ and $y=\displaystyle \sum_{i \in \Lambda}  {\beta}_{i} e_{i} $  in $A$, we define
$$ x \bullet_{B} y = \left(   \sum_{i \in \Lambda} {\alpha}_{i} e_{i} \right) \bullet_{B}\left(   \sum_{i \in \Lambda}  {\beta}_{i} e_{i} \right)  :=  \sum_{i \in \Lambda}{\alpha}_i {\beta}_i e_i.$$
\end{definition}

Now, in the case of an evolution algebra we may write \eqref{preproducto} as follows.
\begin{equation}\label{producto}
\xi_{B}(xy)= M_{B}\left(\xi_{B}(x)\bullet_B \xi_{B}(y)\right),
\end{equation}
where, by abuse of notation, we write $\bullet_B$ to multiply two matrices, by identifying the matrices with the corresponding vectors and multiplying them as in Definition \ref{PuntoGordo}.
\subsection{Change of basis.}
First, we study the matrix of the product of a finite dimensional arbitrary algebra under change of basis. {Then we fix our attention in evolution algebras.}

 Let $B=\{e_{i}\; \vert \; i \in\Lambda\}$ and $B'=\{f_{j}\; \vert \; j\in\Lambda\}$ be two bases of an algebra \emph{A}. Suppose that the relation between these bases is given by
 \[e_{i}= \sum_{k \in \Lambda}q_{ki}f_{k} \quad \text{and} \quad f_{i}= \sum_{k \in \Lambda}p_{ki}e_{k},\]
 where $\{p_{ki}\}_{k,i \in \Lambda}$ and $\{q_{ki}\}_{k,i \in \Lambda}$ are subsets of $\mathbb{K}$ such that $P_{BB'}:= (q_{ki})$ and $P_{B'B}:=(p_{ki})$ are the change of basis matrices. Assume that the structure constants of $A$ relative to $B$ and to $B'$ are, respectively, $\{\varpi_{kij}\}_{i,j,k \in \Lambda}$ and $\{\omega_{kij}\}_{i,j,k \in \Lambda}$. Then, for every $i,j \in \Lambda$:
\begin{align*}
f_{i}f_{j} & =\left(\sum_{k \in \Lambda}p_{ki}e_{k}\right)\left(\sum_{t \in \Lambda}p_{tj}e_{t}\right)=\sum_{k,t \in \Lambda}p_{ki}p_{tj}e_{k}e_{t}= \sum_{k,t,m \in \Lambda}p_{ki}p_{tj}\varpi_{mkt}e_{m}\\
&=
\sum_{k,t,m,l \in \Lambda}p_{ki}p_{tj}\varpi_{mkt}q_{lm}f_{l}=
\sum_{l\in \Lambda}\left(\sum_{k, t, m \in \Lambda}\left(p_{ki}p_{tj}\varpi_{mkt}q_{lm}\right) \right) f_{l}=
\sum_{l\in \Lambda}\omega_{lij}f_{l}.
\end{align*}
Therefore,
$\sum_{k, t, m\in \Lambda}\left(p_{ki}p_{tj}\varpi_{mkt}q_{lm}\right)=\omega_{lij}.$

Our next aim is to express every $\omega_{lij}$ in terms of certain matrices. To find such matrices, write:

\vspace*{-13mm}

\begin{align*}
\omega_{lij} & =   p_{1i}p_{1j}\varpi_{111}q_{l1}+ \ldots + p_{1i}p_{1j}\varpi_{n11}q_{ln} \\
&   \vdots\\
& +  p_{1i}p_{nj}\varpi_{11n}q_{l1}+ \ldots + p_{1i}p_{nj}\varpi_{n1n}q_{ln} \\
&           \vdots  \\
& +          p_{ni}p_{1j}\varpi_{1n1}q_{l1}+ \ldots + p_{ni}p_{1j}\varpi_{nn1}q_{ln}  \\
&             \vdots  \\
& +         p_{ni}p_{nj}\varpi_{1nn}q_{l1}+ \ldots + p_{ni}p_{nj}\varpi_{nnn}q_{ln}.
\end{align*}


In terms of matrices,


\begin{eqnarray*}
\omega_{lij} & = &
   \begin{pmatrix}
     q_{l1} & \cdots & q_{ln} \\
         \end{pmatrix}
  \begin{pmatrix}
    \varpi_{111} & \cdots & \varpi_{11n} \\
    \vdots  & \ddots & \vdots \\
    \varpi_{n11}  & \cdots & \varpi_{n1n} \\
     \end{pmatrix}
  \begin{pmatrix}
    p_{1i}p_{1j}  \\
    \vdots \\
    p_{1i}p_{nj} \\
     \end{pmatrix}
     \\
&  + &
     \cdots  \\
& + &
     \begin{pmatrix}     q_{l1} & \cdots & q_{ln} \\
         \end{pmatrix}
  \begin{pmatrix}
    \varpi_{1n1} &  \cdots & \varpi_{1nn} \\
    \vdots & \ddots & \vdots \\
    \varpi_{nn1} &  \cdots & \varpi_{nnn} \\
     \end{pmatrix}
  \begin{pmatrix}
    p_{ni}p_{1j}  \\
    \vdots \\
    p_{ni}p_{nj} \\
     \end{pmatrix}.
\end{eqnarray*}

This is equivalent to:

\begin{eqnarray*}
 M_{B'}({\lambda}_{f_{i}}) & = &
     \begin{pmatrix}
     q_{11} & \cdots & q_{1n} \\
     \vdots & \ddots & \vdots \\
     q_{n1} &  \cdots & q_{nn} \\
         \end{pmatrix}
  \begin{pmatrix}
    \varpi_{111} & \cdots & \varpi_{11n} \\
    \vdots &  \ddots & \vdots \\
    \varpi_{n11} & \cdots & \varpi_{n1n} \\
     \end{pmatrix}
  \begin{pmatrix}
    p_{1i}p_{11} & \cdots & p_{1i}p_{1n}\\
    \vdots & \ddots & \vdots \\
    p_{1i}p_{n1} & \cdots & p_{1i}p_{nn}\\
     \end{pmatrix} \\
     &   &  \\
& +  &  \cdots  \\
     &  + &
     \begin{pmatrix}
     q_{11} & \cdots & q_{1n} \\
     \vdots & \ddots & \vdots \\
     q_{n1} & \cdots & q_{nn} \\
         \end{pmatrix}
  \begin{pmatrix}
    \varpi_{1n1} & \cdots & \varpi_{1nn} \\
    \vdots &  \ddots & \vdots \\
    \varpi_{nn1} &  \cdots & \varpi_{nnn} \\
     \end{pmatrix}
  \begin{pmatrix}
    p_{ni}p_{11} &  \cdots & p_{ni}p_{1n}\\
    \vdots &  \ddots & \vdots \\
    p_{ni}p_{n1} &  \cdots & p_{ni}p_{nn}\\
     \end{pmatrix} \\
     & =  & P_{B'B}^{-1}\left(\displaystyle\sum_{k}M_{B}({\lambda}_{e_{k}})p_{ki}\right)P_{B'B}.
\end{eqnarray*}

We finish the section by asserting the relationship among two structure matrices associated to the same evolution algebra relative to different bases. We include the proof of Theorem \ref{teor} for completeness. The ideas we have used can be found in \cite[Section 3.2.2.]{Tian}.

\begin{theorem}\label{teor}
Let $A$ be an  evolution algebra and let $B=\{e_{1},\ldots,e_{n} \}$ be a natural basis of $A$ with structure matrix $M_{B}=(\omega_{ij})$. Then:
\begin{enumerate}[\rm (i)]
\item If $B'=\{f_{1},\ldots,f_{n}\}$ is a natural basis of $A$  and $P=(p_{ij})$ is the change of basis matrix $P_{B'B}$, i.e., $f_{i}=\displaystyle\sum_{j}p_{ji}e_{j}$, for every $i$, then
 $\vert P \vert \neq 0$ and
\begin{eqnarray}\label{ecuac1}
\begin{pmatrix}
\omega_{11} & \cdots & \omega_{1n} \\
\vdots & \ddots & \vdots \\
\omega_{n1} & \cdots & \omega_{nn} \\
\end{pmatrix}
\left[
\begin{pmatrix}
p_{1i} \\
\vdots  \\
p_{ni} \\
\end{pmatrix}
\bullet_B
\begin{pmatrix}
p_{1j} \\
\vdots  \\
p_{nj} \\
\end{pmatrix}
\right]=
\begin{pmatrix}
0 \\
\vdots  \\
0 \\
\end{pmatrix} \quad \text{for every}\  i \neq j.
\end{eqnarray}
Moreover,
\begin{equation}\label{ecuac2}
M_{B'}=\begin{pmatrix}
p_{11} & \cdots & p_{1n} \\
\vdots & \ddots & \vdots \\
p_{n1} & \cdots & p_{nn} \\
\end{pmatrix}^{-1}
\begin{pmatrix}
\omega_{11} & \cdots & \omega_{1n} \\
\vdots & \ddots & \vdots \\
\omega_{n1} & \cdots & \omega_{nn} \\
\end{pmatrix}
\begin{pmatrix}
p_{11}^{2} & \cdots & p_{1n}^{2} \\
\vdots & \ddots & \vdots \\
p_{n1}^{2} & \cdots & p_{nn}^{2} \\
\end{pmatrix}=
P^{-1}M_{B}P^{(2)},
\end{equation}
\noindent
where $P^{(2)}=(p_{ij}^2)$.
\item {Assume that} $P=(p_{ij})\in \mathcal M_n(\K)$ has non-zero determinant and satisfies  the relations in {\rm{(\ref{ecuac1})}}. Define $B'=\{f_1, \dots, f_n\}$, where $f_{i}=\sum_{j}p_{ji}e_{j}$, for every $i$.  Then, $B'$ is a natural basis and {\rm{(\ref{ecuac2})}}  is satisfied.
\end{enumerate}
\end{theorem}

\begin{proof}
(i). Clearly, since $B$ and $B'$ are two bases of $A$ then $\vert P\vert \neq 0$.  Besides, since $B$ and $B'$ are natural bases, by  \eqref{preproducto}  we have:
\begin{center}
\begin{equation*}
\xi_B(f_{i}f_{j})=
\begin{pmatrix}
\omega_{11} & \cdots & \omega_{1n} \\
\vdots & \ddots & \vdots \\
\omega_{n1} & \cdots & \omega_{nn} \\
\end{pmatrix}
\left[
\begin{pmatrix}
p_{1i} \\
\vdots  \\
p_{ni} \\
\end{pmatrix}
\bullet_B
\begin{pmatrix}
p_{1j} \\
\vdots  \\
p_{nj} \\
\end{pmatrix}
\right] { =\begin{pmatrix}
0 \\
\vdots  \\
0 \\
\end{pmatrix}}
\end{equation*}
\end{center}
and
\begin{center}
\begin{equation*}
\xi_B(f_{i}^2)=
\begin{pmatrix}
\omega_{11} & \cdots & \omega_{1n} \\
\vdots & \ddots & \vdots \\
\omega_{n1} & \cdots & \omega_{nn} \\
\end{pmatrix}
\begin{pmatrix}
p_{1i}^{2} \\
\vdots  \\
p_{ni}^{2} \\
\end{pmatrix}
\end{equation*}
\end{center}
for every $i, j$, being $i\neq j$.

On the other hand,  if $M_{B'}=(\varpi_{ij})$, for every $i \neq j$ we obtain:

$$
\xi_{B'}({f_i^2})=\begin{pmatrix}
p_{11} & \cdots & p_{1n} \\
\vdots & \ddots & \vdots \\
p_{n1} & \cdots & p_{nn} \\
\end{pmatrix}^{-1}
\begin{pmatrix}
\omega_{11} & \cdots & \omega_{1n} \\
\vdots & \ddots & \vdots \\
\omega_{n1} & \cdots & \omega_{nn} \\
\end{pmatrix}
\begin{pmatrix}
p_{1i}^{2} \\
\vdots  \\
p_{ni}^{2}  \\
\end{pmatrix}=
\begin{pmatrix}
\varpi_{1i} \\
\vdots \\
\varpi_{ni} \\
\end{pmatrix}
$$

\noindent
and consequently

$$
M_{B'}=\begin{pmatrix}
p_{11} & \cdots & p_{1n} \\
\vdots & \ddots & \vdots \\
p_{n1} & \cdots & p_{nn} \\
\end{pmatrix}^{-1}
\begin{pmatrix}
\omega_{11} & \cdots & \omega_{1n} \\
\vdots & \ddots & \vdots \\
\omega_{n1} & \cdots & \omega_{nn} \\
\end{pmatrix}
\begin{pmatrix}
p_{11}^{2} & \cdots & p_{1n}^{2} \\
\vdots & \ddots & \vdots \\
p_{n1}^{2} & \cdots & p_{nn}^{2} \\
\end{pmatrix} = P^{-1}M_{B}P^{(2)}.
$$

(ii). Assume that $P=(p_{ij})$ has non zero determinant. Then $B'$, defined as in the statement, is a basis of $A$. Moreover,  if \eqref{ecuac1} is satisfied, then $B'$ is a natural basis as follows by \eqref{preproducto}.
\end{proof}

The formula \eqref{ecuac1} can be rewritten in a more condensed way. Concretely {(see \cite{Tian})},
\begin{center}
\begin{equation}\label{ecuac3}
M_{B}(P \ast P)=0,
\end{equation}
\end{center}
where  $P \ast P= (c_{k (i,j)})\in \mathcal M_{n \times \frac{n(n-1)}{2}}(\K)$, being
$c_{k (i,j)}=p_{ki}p_{kj}$ for every pair $(i,j)$ with $i < j$ and $i, j \in \{1, \dots, n\}$.
\medskip

%
%
\section{Three-dimensional evolution algebras}\label{threedim}
%
%

 The aim of this section is to determine  the three-dimensional evolution algebras over a field $\K$ having characteristic different from two and such that for any $\alpha\in \K$  and $n= 2, 3, 7$, the equation $x^n=\alpha$ has a solution. For our purposes, we divide our study in different cases, depending on the dimension of $A^2$.

\subsection{Action of $S_3 \rtimes (\K^\times)^3$ on $\mathcal M_3(\K)$}\label{Action}

\noindent

\noindent
Let $\K$ be a field. By $\K^\times$ we denote $\K\setminus \{0\}$.
For every $\alpha, \beta, \gamma \in \K^\times$, we define the matrices:

\begin{equation*}
 \rm{\Pi}_{1}(\alpha):=
 \begin{pmatrix}
\alpha & 0 & 0 \\
 0 & 1 & 0 \\
0  & 0 & 1
  \end{pmatrix}, \, \,
\rm{\Pi}_{2}(\beta):=
 \begin{pmatrix}
1 & 0 & 0 \\
 0 & \beta & 0 \\
0  & 0 & 1
   \end{pmatrix}, \, \,
  \rm{\Pi}_{3}(\gamma):=
 \begin{pmatrix}
1 & 0 & 0 \\
0 & 1 & 0 \\
0  & 0 & \gamma
   \end{pmatrix}.
\end{equation*}

It is easy to prove that they commute each other. This implies that

$$G=\left\{\rm{\Pi}_1(\alpha)\rm{\Pi}_2(\beta)\rm{\Pi}_3(\gamma)\ \vert \, \alpha,\beta,\gamma \in \K^\times \right\}=\left\{
\begin{pmatrix}
\alpha & 0 & 0 \\
 0 & \beta & 0 \\
0  & 0 & \gamma
\end{pmatrix}
\ \vert \ \alpha, \beta, \gamma \in \K^\times  \right\} $$

\noindent
is an abelian subgroup of $\rm{GL}_3(\K )$. We will denote the diagonal matrix $\begin{pmatrix}
\alpha & 0 & 0 \\
 0 & \beta & 0 \\
0  & 0 & \gamma
\end{pmatrix}$ by $(\alpha,\beta,\gamma)$. With this notation in mind, it is immediate to see that $G \cong \K^\times \times \K^\times \times \K^\times $ with product given by $(\alpha,\beta,\gamma)(\alpha',\beta',\gamma'):= (\alpha\alpha',\beta\beta',\gamma\gamma')$.

%

\medskip

Now, consider the symmetric group $S_3$ of all permutations of the set $\{1, 2, 3\}$. The standard notation for $S_3$ is:
$$ S_3=\{id,(1,2),(1,3),(2,3),(1,2,3),(1,3,2)\},$$

\noindent
where $id$ is the identity map, $(i, j)$ is the permutation that sends the element $i$ into the element $j$  and $(i, j, k)$ is the permutation sending $i$ to $j$, $j$ to $k$ and $k$ to $i$, for $\{i, j, k\} =\{1, 2, 3\}$.

We may identify  $S_3$ with the set

\begin{equation}\label{permumat}
\left\{\rm{id}_3,
\begin{pmatrix}
0 & 1 & 0 \\
1 & 0 & 0 \\
0  & 0 & 1
\end{pmatrix},
\begin{pmatrix}
0 & 0 & 1 \\
0 & 1 & 0 \\
1  & 0 & 0
\end{pmatrix},
\begin{pmatrix}
1 & 0 & 0 \\
0 & 0 & 1 \\
0  & 1 & 0
\end{pmatrix},
\begin{pmatrix}
0 & 1 & 0 \\
0 & 0 & 1 \\
1  & 0 & 0
\end{pmatrix},
\begin{pmatrix}
0 & 0 & 1 \\
1 & 0 & 0 \\
0  & 1 & 0
\end{pmatrix}
 \right\}
 \end{equation}

\noindent
in the following way: $id$ is identified with the identity matrix $\rm{id}_3$, $(1, 2)$ with the matrix
$$
\begin{pmatrix}
0 & 1 & 0 \\
1 & 0 & 0 \\
0  & 0 & 1
\end{pmatrix}$$
\noindent
because this matrix appears {when permuting the first and the second columns of} $\rm{id}_3$, etc.
The matrices in \eqref{permumat} are called $3\times3$ {\emph{permutation matrices}}.
%
%

From now on, we will consider that $S_3$ consists of the permutation matrices.
This allows to see $S_3$ as a subgroup of $\rm{GL}_3(\K)$. Denote by $H$ the subgroup of  $\rm{GL}_3(\K)$ generated by  $S_3$ and  $(\K^\times)^3$.


\medskip

It is not difficult to verify that for every $\sigma \in S_3$ and every $({\lambda}_1, {\lambda}_2, {\lambda}_3) \in (\K^\times)^3$ its product is as follows:
$$({\lambda}_1,{\lambda}_2,{\lambda}_3)\sigma=\sigma ({\lambda}_{\sigma(1)},{\lambda}_{\sigma(2)},{\lambda}_{\sigma(3)}).$$

\noindent
Therefore, we may write
$$H=\{\sigma(\alpha,\beta,\gamma)\ \vert \ \sigma \in S_3, \; (\alpha,\beta,\gamma)\in (\K^\times)^3\}. $$

\noindent
The multiplication in $H$ is given by
\begin{eqnarray}\label{ProdSemidir}
\sigma({\alpha}_1,{\alpha}_2,{\alpha}_3)\tau ({\beta}_1,{\beta}_2,{\beta}_3) & = &\sigma \tau ({\alpha}_{\tau(1)},{\alpha}_{\tau(2)},{\alpha}_{\tau(3)})({\beta}_1,{\beta}_2,{\beta}_3)\\ \nonumber
& = &\sigma \tau ({\alpha}_{\tau(1)}{\beta}_1,{\alpha}_{\tau(2)}{\beta}_2,{\alpha}_{\tau(3)}{\beta}_3).
\end{eqnarray}

{A} \emph{semidirect product} of $S_3$ and $(\K^{\times})^3$ is defined as $S_3 \times (\K^{\times})^3$ with product as in \eqref{ProdSemidir}. It is denoted by $$S_3 \rtimes (\K^{\times})^3.$$
Notice that  $S_3 \rtimes (\K^{\times})^3$ coincides with

\begin{equation}
\label{permutationmatrix}
\footnotesize{
\left\{
\begin{pmatrix}
\alpha & 0 &  0 \\
0 & \beta & 0  \\
0 & 0 & \gamma
\end{pmatrix},
\begin{pmatrix}
0 & \alpha &  0 \\
\beta & 0 & 0  \\
0 & 0 & \gamma
\end{pmatrix},
\begin{pmatrix}
0 & 0 &  \alpha\\
0 & \beta & 0  \\
\gamma & 0 & 0
\end{pmatrix},
\begin{pmatrix}
\alpha & 0 &  0 \\
0 & 0 & \beta  \\
0 & \gamma & 0
\end{pmatrix},
\begin{pmatrix}
0 & \alpha &  0 \\
0 & 0 & \beta  \\
\gamma & 0 & 0
\end{pmatrix},
\begin{pmatrix}
0 & 0 &  \alpha \\
\beta & 0 & 0  \\
0 & \gamma & 0
\end{pmatrix}\  \vert\ \alpha, \beta, \gamma \in \K^\times \right\}}
\end{equation}

\medskip
\noindent
Thus, $S_3 \rtimes (\K^{\times})^3=\left\{ (\alpha,\beta,\gamma)\sigma \; \;\vert \; \alpha, \beta, \gamma \in \K^\times, \; \;\sigma \in S_3 \right\}$.\\

\medskip

We define the  action of $S_3 \rtimes (\K^{\times})^3$ on the set ${\mathcal M}_3(\K)$ given by:

\begin{equation}\label{producto1}
\sigma \cdot
\begin{pmatrix}
\omega_{11} & \omega_{12} & \omega_{13}\\
\omega_{21} & \omega_{22} & \omega_{23} \\
\omega_{31} & \omega_{32} & \omega_{33} \\
\end{pmatrix}:=
\begin{pmatrix}
\omega_{\sigma(1)\sigma(1)} & \omega_{\sigma(1)\sigma(2)} & \omega_{\sigma(1)\sigma(3)}\\
\omega_{\sigma(2)\sigma(1)} & \omega_{\sigma(2)\sigma(2)} & \omega_{\sigma(2)\sigma(3)} \\
\omega_{\sigma(3)\sigma(1)} & \omega_{\sigma(3)\sigma(2)} & \omega_{\sigma(3)\sigma(3)} \\
\end{pmatrix}.
\end{equation}

\begin{equation}\label{producto2}
(\alpha,\beta,\gamma)\cdot
\begin{pmatrix}
\omega_{11} & \omega_{12} & \omega_{13}\\
\omega_{21} & \omega_{22} & \omega_{23} \\
\omega_{31} & \omega_{32} & \omega_{33} \\
\end{pmatrix}:=
\begin{pmatrix}
\alpha \omega_{11} & \frac{\beta^2}{\alpha} \omega_{12} & \frac{\gamma^2}{\alpha} \omega_{13}\\ \\
\frac{\alpha^2}{\beta} \omega_{21} & \beta \omega_{22} & \frac{\gamma^2}{\beta} \omega_{23}\\ \\
\frac{\alpha^2}{\gamma} \omega_{31} & \frac{\beta^2}{\gamma} \omega_{32} & \gamma \omega_{33}
\end{pmatrix}
\end{equation}

\noindent
for every $\sigma \in S_3$ and  every $(\alpha,\beta,\gamma) \in (\K^{\times})^3$.
\medskip

For arbitrary $P\in S_3 \rtimes (\K^{\times})^3$ and $ {M}\in {\mathcal M}_3(\K)$, the action of $P$ on ${M}$ can be formulated as follows:

\begin{equation}\label{accion}
 P \cdot {M}: = P^{-1}{M}P^{(2)}.
 \end{equation}
\medskip

\begin{remark}\label{explicacion}
\rm
The action given in \eqref{accion} has been inspired by Condition \eqref{ecuac2} in Theorem \ref{teor}. Notice that any matrix $P$ in $S_3 \rtimes (\K^{\times})^3$ is a change of basis matrix from a natural basis $B$ into another natural basis $B'$ and the relationship among the structure matrices $M_B$ and $M_{B'}$ and the matrix $P$ is as given in Condition \eqref{ecuac2}, that is, $P^{-1}M_BP^{(2)}=M_B'$. This is the reason because we define the action of $P$ on $M_B$ by:
$$P \cdot M_B = P^{-1}M_BP^{(2)}.$$
\end{remark}

The result that follows will be very useful in Theorem \ref{thm:clasificacion}.

\begin{proposition}\label{prop:numerodeceros}
For any $P\in S_3\rtimes (\K^\times)^3$ and any $M \in \mathcal{M}_3(\K)$ we have:
\begin{enumerate}[\rm (i)]
\item\label{numerodeceros1} The number of zero entries in $M$ coincides with the number of zero entries in $P \cdot M$.
\item\label{numerodeceros2} The number of zero entries in the main diagonal of $M$ coincides with the number of zero entries in the main diagonal of $P \cdot M$.
\item\label{numerodeceros3} The rank of $M$ and the rank of $P\cdot M$ coincide.
\item\label{numerodeceros4} Assume that $M$ is the structure matrix of an evolution algebra $A$ relative to a natural basis $B$. Assume that $A^2=A$. If $N$ is the structure matrix of $A$ relative to a natural basis $B'$ then there exists $Q\in S_3\rtimes (\K^\times)^3$ such that $N= Q \cdot M$.
\end{enumerate}
\end{proposition}
\begin{proof}
Fix an element $P$ in $S_3\rtimes (\K^\times)^3$. Then there exist $\sigma \in S_3$ and $(\alpha, \beta, \gamma)\in (\K^\times)^3$ such that $P=\sigma(\alpha, \beta, \gamma)$. Therefore $P\cdot M = (\sigma(\alpha, \beta, \gamma))\cdot M =  \sigma \cdot ((\alpha, \beta, \gamma)\cdot M)$. Item  \eqref{numerodeceros1} and \eqref{numerodeceros2} follows by \eqref{producto1} and \eqref{producto2}. Item \eqref{numerodeceros3}  is easy to show because $P\cdot M = P^{-1}MP^{(2)}$ and $P$ is an invertible matrix.
Finally, \eqref{numerodeceros4} follows from the definition of the action and \cite[Theorem 4.4]{EL}.
\end{proof}

\subsection{Main theorem}
Here we prove the main result of the paper: the classification of three-dimensional evolution algebras over a field of characteristic different from two in which there are roots of orders two, three and seven.

\begin{definitions}
\rm (See \cite[Definitions 2.4]{CSV1}).
An \emph{evolution subalgebra} of an evolution algebra $A$
is a subalgebra $A' \subseteq A$ such that $A'$
is an evolution algebra, i.e. $A'$ has a natural basis.

We say that $A'$ has the \emph{extension property} if there exists a natural basis $B'$ of $A'$ which can be extended to a natural basis of $A$.
\end{definitions}

\medskip

An evolution algebra is \emph{non-degenerate} if $e^2\neq0$ for any element $e$ in any basis (see \cite[Definition 2.16 and Corollary 2.19]{CSV1}). Otherwise we say that it is \emph{degenerate}. Note that this definition does not depend on the basis.

\medskip

\begin{definition}\label{}
\rm
A three dimensional evolution algebra $A$ is said to have \emph{Property} (2LI) if for any basis $\{e_1, e_2, e_3\}$ of $A$,  the ideal $A^2$ has dimension two and it is generated by $\{e_i^2, e_j^2\}$, for  $i, j\in \{1, 2, 3\}$ with $i\neq j$.
\end{definition}

 \begin{theorem}\label{thm:clasificacion}
 Let $A$ be a three-dimensional evolution $\K$-algebra.
 \begin{enumerate}[\rm (i)]
\item\label{cero} If ${\rm dim}(A^2)=0$ then $M_B=0$ for any natural basis $B$ of $A$.
\item\label{uno}  If ${\rm dim}(A^2)=1$ then there exists a natural basis $B$ such that $M_B$ is one of the  seven matrices given in Table 1. All of them produce mutually non-isomorphic evolution algebras. The algebras in this case are completely classified by the following properties: having $A^2$ the extension property, ${\rm dim}({\rm ann}(A))$, and {whether or not} $A$ has a principal ideal of dimension two which is degenerate.
\item\label{dos}  If ${\rm dim}(A^2)=2$, then then there exists a natural basis $B$ such that $M_B$ is one of the  matrices given in the first column of Tables 2 to 17. They are all mutually non-isomorphic except in the cases shown in Tables 2' to 17'. There are {57} possible cases.
Let $B=\{e_1, e_2, e_3\}$ {be such that $\{e_1^2, e_2^2\}$ is a basis of $A^2$ and $e_3^2=c_1e_1^2+c_2e_2^2$}, for $c_1, c_2\in \K$.

 \begin{enumerate}[\rm (a)]
 \item If $c_1c_2\neq 0$, then the evolution algebras have ${\rm dim}({\rm ann}(A))=0$;   the algebra $A$ has Property {\rm (2LI)} and the number of non-zero entries in $M_B$ can be 4 to 9.
 \item If $c_1 = 0 $ and $c_2 \neq 0$ (the case $c_2 = 0 $ and $c_1 \neq 0$ is analogous), then  the evolution algebras appearing have ${\rm dim}({\rm ann}(A))=0$;  the algebra $A$ has not Property {\rm (2LI)} and the number of non-zero entries in the set that follows can be from 4 to 9.
 \smallskip

\scalebox{0.68}
{ $S=\Bigg\{$
$
\begin{pmatrix}
1 & 0 &  0 \\
0 & 1 & 1  \\
0 & \alpha & \alpha
\end{pmatrix},
$
$
\begin{pmatrix}
0 & 1 & 1 \\
1 & 0 & 0 \\
\alpha & 0 & 0
\end{pmatrix},
$
$
\begin{pmatrix}
\alpha & 1 &  1 \\
0 & 1 & 1  \\
0 & \beta & \beta
\end{pmatrix},
$
$
\begin{pmatrix}
0 & 0 & 0 \\
1 & 1 & 1  \\
\alpha & \beta & \beta
\end{pmatrix},
$
$
\begin{pmatrix}
1 & 1 & 1 \\
\alpha & 0 & 0 \\
\beta & 0 & 0
\end{pmatrix},
$
$
\begin{pmatrix}
1 & 0 & 0 \\
\alpha & 1 & 1 \\
\beta & \gamma & \gamma
\end{pmatrix},
$
$\begin{pmatrix}
0 & 1 & 1 \\
\alpha & 1 & 1 \\
\beta & \gamma & \gamma
\end{pmatrix},
$
$
\begin{pmatrix}
\alpha & 1 & 1 \\
\beta & 1 & 1 \\
\gamma & \lambda & \lambda
\end{pmatrix}
$
$\Bigg\}.$
}
\smallskip

 \item If $c_1, c_2= 0 $, then  the evolution algebras appearing have ${\rm dim}({\rm ann}(A))=1$;  the algebra $A$ has not Property {\rm (2LI)} and the number of non-zero entries in rows one and two can be from 1 to 4.
 \end{enumerate}

 \item\label{tres}  If ${\rm dim}(A^2)=3$ then there exists a natural basis $B$ such that $M_B$ is one of the   matrices given in the first column of Tables 18 to 24. They are all mutually non-isomorphic except in the cases shown in Tables 19' to 23'. They are completely determined by the  number of non-zero entries in $M_B$. There are 51 possible cases.
 \end{enumerate}
\end{theorem}

\noindent
\emph{Proof.}
\medskip

 Fix a three-dimensional evolution algebra $A$ and a natural basis $B=\{e_{1},e_{2},e_{3}\}$. Let $M_B$ be the structure matrix of $A$ relative to $B$:
$$M_B=\begin{pmatrix}
\omega_{11} & \omega_{12} & \omega_{13}\\
\omega_{21} & \omega_{22} & \omega_{23} \\
\omega_{31} & \omega_{32} & \omega_{33} \\
\end{pmatrix}.
$$

 In order to classify all the three dimensional evolution algebras we try to find a basis of $A$ for which  its structure matrix has an expression as easy as possible, where by `easy' we mean with the maximum number of 0, 1 and -1 in the entries.

\medskip
\noindent
{\bf Case ${\rm dim}(A^{2})=0$. }

\noindent
Then $M_B=0$
and there is a unique evolution algebra.
\medskip

\noindent
{\bf Case ${\rm dim}(A^{2})=1$.}

\noindent
Without loss in generality we may assume $e_{1}^{2} \neq 0$. Write $e_{1}^{2}=\omega_{1}e_{1}+\omega_{2}e_{2}+\omega_{3}e_{3}$, where $\omega_i\in \K$ and $\omega_i\neq 0$ for some $i$.
Note that $\{e_{1}^{2}\}$ is a basis of $A^{2}$.

Since $e_{2}^{2}, e_{3}^{2}\in A^2$, there exist $c_1, c_2 \in \K$ such that
\begin{align*}
e_{2}^{2}=c_{1}e_{1}^{2}=c_{1}(\omega_{1}e_{1}+\omega_{2}e_{2}+\omega_{3}e_{3}),\\
e_{3}^{2}=c_{2}e_{1}^{2}=c_{2}(\omega_{1}e_{1}+\omega_{2}e_{2}+\omega_{3}e_{3}). \\
\end{align*}

Then
$$ M_{B}=
 \begin{pmatrix}
 \omega_{1} & c_{1}\omega_{1} & c_{2}\omega_{1} \\
 \omega_{2} & c_{1}\omega_{2} & c_{2}\omega_{2} \\
 \omega_{3} & c_{1}\omega_{3} & c_{2}\omega_{3}
  \end{pmatrix}.$$

We start the study of this case by paying attention to the algebraic properties of the evolution algebras that we consider. To see which are the matrices that appear as change of basis matrices, we refer the reader to Appendix \ref{apendice}.

We analyze when $A^2$ has the extension property. That is, if there exists a natural basis $B'=\{e_{1}', e_{2}', e_{3}' \}$ of $A$ with
\begin{eqnarray} \label{amplia}
           e_{1}'&=& e_1^2 = \omega_1e_1+\omega_2e_2+\omega_3e_3 \\
\nonumber  e_{2}'&=& \alpha e_{1}+ \beta e_{2}+\gamma e_{3} \\
\nonumber  e_{3}'&=& \delta e_{1}+ \nu e_{2}+\eta e_{3},
 \end{eqnarray}

\noindent
for some $\alpha, \beta, \gamma, \delta, \nu, \eta \in \K$ that we may choose satisfying $\nu(\beta-\gamma)\neq 0$.
Being $B'$ a basis implies
\begin{eqnarray}\label{determ}
\vert\ P_{B'B}\ \vert =\left\vert\begin{matrix}
 \omega_{1} & \alpha & \delta\\
 \omega_{2} & \beta & \nu \\
 \omega_{3} & \gamma & \eta
 \end{matrix}\right\vert
 \neq 0.
\end{eqnarray}
By  Theorem \ref{teor}, $B'$ is a natural basis if and only if the following conditions are satisfied.
\begin{eqnarray}
\alpha \omega_{1}+\beta \omega_{2}c_{1}+\gamma \omega_{3}c_{2}  =& 0 \label{condic1}\\
\delta \omega_{1}+\nu \omega_{2}c_{1}+\eta \omega_{3}c_{2}  =& 0  \label{condic2}\\
\nonumber \alpha \delta + \beta \nu c_{1} + \gamma \eta c_{2}   =& 0
\end{eqnarray}

\medskip
In these conditions, the structure matrix of $A$ relative to $B'$ is:
$$ M_{B'}=
 \begin{pmatrix}
 \omega_{1}^{2}+\omega_{2}^{2}c_{1}+\omega_{3}^{2}c_{2} & \alpha^{2}+\beta^{2}c_{1}+\gamma^{2}c_{2} & \delta^{2}+\nu^{2}c_{1}+\eta^{2}c_{2} \\
 0 & 0 & 0 \\
 0 & 0 & 0
  \end{pmatrix}. $$

\medskip

 For the computations, we will take into account  \eqref{producto2}. On the other hand, to find the different mutually non-isomorphic evolution algebras it will be very useful to study if they have a two-dimensional evolution ideal generated by one element which is degenerate as an evolution algebra.

 Now, we start with the analysis of the different cases.
\medskip

 \noindent
  {\bf Case 1.} Suppose that $\omega_{1} \neq 0$.

  \noindent
 By changing the basis, we may consider $e_{1}^{2}=e_{1}+\omega_{2}e_{2}+\omega_{3}e_{3}$.
Using  \eqref{condic1} we get
$\alpha=-(\beta \omega_2 c_1 +\gamma \omega_3 c_2)$ and by   \eqref{condic2},  $\delta=-(\nu \omega_2 c_1 +\eta \omega_3 c_2)$. If we replace $\alpha$ and $\delta$ in \eqref{determ} we obtain that:

$$ \vert\ P_{B'B}\ \vert = (1+\omega_2^2c_1+\omega_3^2c_2)\nu(\beta-\gamma). $$

Now we distinguish  if $\vert\ P_{B'B}\ \vert $ is zero or not. This happens depending on $1+\omega_2^2c_1+\omega_3^2c_2$ being zero or not.\\

 \noindent
 {\bf Case 1.1} Assume $1+\omega_2^2c_1+\omega_3^2c_2=0$.

 \noindent
In this case $A^2$ has not the extension property since $\vert P_{B'B} \vert=0 $. We will analyze what happens when $1+\omega_3^2c_2 \neq 0$ and when $1+\omega_3^2c_2=0$.
\medskip

 \noindent
{\bf Case 1.1.1}   If $1+\omega_3^2c_2 \neq 0$.

\noindent
Note that $\omega_{2}^2c_1 \neq 0$ since otherwise we get a contradiction.  Then $c_{1}=\dfrac{-1-\omega_{3}^{2}c_{2}}{\omega_{2}^{2}}$. In this case, the structure matrix is:
$$ \footnotesize{M_{B}=
 \begin{pmatrix}
 1 & \dfrac{-1-\omega_{3}^{2}c_{2}}{\omega_{2}^{2}} & c_{2} \\
 \\
  \omega_{2} & \dfrac{-1-\omega_{3}^{2}c_{2}}{\omega_{2}} & c_{2}\omega_{2} \\
 \\
 \omega_{3} & \dfrac{(-1-\omega_{3}^{2}c_{2})\omega_{3}}{\omega_{2}^{2}} & c_{2}\omega_{3}
  \end{pmatrix}. }$$

\medskip

 \noindent
  {\bf Case 1.1.1.1} Suppose that $\omega_3 \neq 0$.

\noindent

If we take the natural basis $B''=\{e_1, \omega_2e_2, \omega_3e_3\}$, then

\begin{eqnarray}\label{caso 1.1.1.1}
\footnotesize{M_{B''}=
 \begin{pmatrix}
1 & -1-\omega_{3}^{2}c_{2} & \omega_{3}^{2}c_{2} \\ \\
 1 & -1-\omega_{3}^{2}c_{2} & \omega_{3}^{2}c_{2} \\ \\
 1 & -1-\omega_{3}^{2}c_{2} & \omega_{3}^{2}c_{2}
  \end{pmatrix}.}
  \end{eqnarray}

We are going to distinguish two cases: $c_2=0$ and $c_2\neq 0$.\\

Assume first $c_2=0$. Then $ M_{B''}=
 \begin{pmatrix}
1 & -1 & 0  \\
 1 & -1 & 0 \\
 1 & -1 & 0
  \end{pmatrix}.$ By considering another change of basis we find a structure matrix with more zeros. Concretely, let $B'''=\{e_2, e_1+e_3, e_3\}$. Then

 $$ M_{B'''}=
 \begin{pmatrix}
 1& -1 & 0 \\
 1& -1 & 0 \\
 0& 0 & 0 \\
  \end{pmatrix}. $$

In what follows we will assume that $c_2 \neq 0$. We recall that we are considering the structure matrix given in \eqref{caso 1.1.1.1}. Take $I:=<(1+\omega_3^2c_2)e_1+e_2>$. Then $I$ is a two-dimensional evolution ideal which is degenerate as an evolution algebra.

Now, for $B'''$ the natural basis given by

$$ \footnotesize{P_{B'''B''}=
\begin{pmatrix}
 \dfrac{1+(\omega_3^2c_2)^3+2(\omega_3^2c_2)^2+(\omega_3^2c_2)}{2(1+\omega_3^2c_2)} & \dfrac{-1+(\omega_3^2c_2)^3+2(\omega_3^2c_2)^2+(\omega_3^2c_2)}{2(1+\omega_3^2c_2)} & -(\omega_3^2c_2) \\ \\
 \dfrac{-1+(\omega_3^2c_2)^3+2(\omega_3^2c_2)^2+(\omega_3^2c_2)}{2(1+\omega_3^2c_2)} & \dfrac{1+(\omega_3^2c_2)^3+2(\omega_3^2c_2)^2+(\omega_3^2c_2)}{2(1+\omega_3^2c_2)} & 0 \\ \\
 \dfrac{1+(\omega_3^2c_2)^3+2(\omega_3^2c_2)^2+(\omega_3^2c_2)}{2(1+\omega_3^2c_2)} & \dfrac{-1+(\omega_3^2c_2)^3+2(\omega_3^2c_2)^2+(\omega_3^2c_2)}{2(1+\omega_3^2c_2)} & 1 \\ \\
  \end{pmatrix}  }$$

\noindent
we obtain:
$$ M_{B'''}=\begin{pmatrix}
 1 & -1 & 1 \\
 1 & -1 & 1 \\
 0 & 0 & 0
  \end{pmatrix}. $$

\noindent
Note that $\vert P_{B'''B''} \vert = -2(\omega_3^2c_2) (1+\omega_3^2c_2)^2\neq 0$ because $\omega_3^2c_2 \neq 0$ and $\omega_{3}^2c_{2} \neq -1$.

\medskip
  \noindent
  {\bf Case 1.1.1.2} Suppose that $\omega_3=0$. \\

\noindent
Then $1+\omega_{2}^2c_1=0$ and necessarily $\omega_2^2c_1 \neq 0$. In this case,

\begin{eqnarray}\label{1.1.1.2}
\footnotesize{M_{B}=
 \begin{pmatrix}
1 &   \dfrac{-1}{\omega_2^2} & c_2 \\ \\
\omega_{2} & \dfrac{-1}{\omega_2} & c_2\omega_2 \\ \\
0 &  0 & 0
  \end{pmatrix}.}
  \end{eqnarray}

Again we will distinguish two cases depending on $c_2$.\\

Assume $c_2\neq 0$.
Take $B''=\{e_1,\omega_2e_2, \frac{1}{\sqrt{c_2}}e_3\}$. Then
 $ M_{B''}=
 \begin{pmatrix}
 1& -1 & 1 \\
 1& -1 & 1 \\
 0& 0 & 0 \\
  \end{pmatrix},$ which has already appeared.

\medskip

Suppose $c_2=0$. Then, for $B''=\{e_1, \omega_2e_2, e_3\}$ we have $ M_{B''}=
 \begin{pmatrix}
 1& -1 & 0 \\
 1& -1 & 0 \\
 0& 0 & 0 \\
  \end{pmatrix}$, matrix that has already appeared.\\

\noindent
{\bf Case 1.1.2} Suppose that $1+\omega_{3}^{2}c_2=0$.

\noindent
This implies that $\omega_3^2c_2 \neq 0$ and $\omega_2^2c_1 =0$.
\medskip

\noindent
{\bf Case 1.1.2.1} Assume $c_1 \neq 0$.

\noindent
This implies that $\omega_2=0$. Moreover, as $\omega_{3} \neq 0$, necessarily $c_2=\frac{-1}{\omega_3^2}$. If we take the natural basis $B''= \{e_1, e_3,e_2\}$, then
$ M_{B''}=
\tiny{\begin{pmatrix}
 1 & \dfrac{-1}{\omega_3^2}  &  c_1\\ \\
 \omega_{3} &  \dfrac{-1}{\omega_3}  &\omega_3c_1\\ \\
 0 & 0 & 0
   \end{pmatrix}} $ and we are as in  Case 1.1.1.2.

\medskip

\noindent
{\bf Case 1.1.2.2} Suppose $c_1=0$ and $\omega_2=0$. \\
\noindent
Take $B''=\{e_1,  \omega_3e_3,e_2\}$. Then
 $
  M_{B''}=
 \begin{pmatrix}
1 & -1 & 0 \\
 1 & -1 & 0 \\
0 & 0 & 0
  \end{pmatrix}$ again.
\medskip

 \noindent
 {\bf Case 1.1.2.3} Assume $c_1=0$ and $\omega_2 \neq 0$.

\noindent
Taking $B''=\{e_1,e_3,e_2\}$, we are in the same conditions as in Case 1.1.1.1 with $c_2=0$.
\medskip

\noindent
{\bf Case 1.2} Assume $1+\omega_2^2c_1+\omega_3^2c_2 \neq 0$.

\noindent
We will prove that $A^2$ has the extension property. In any subcase we will provide with a natural basis for $A$ one of which elements constitutes a natural basis of $A^2$.

\medskip

\noindent
{\bf Case 1.2.1} Suppose that $c_1=c_2=0$.\\
\noindent
Consider the natural basis $B'=\{e_1^2, e_2+e_3, 2e_2+e_3\}$. Then
 $$ M_{B'}=
 \begin{pmatrix}
 1& 0 & 0 \\
 0 & 0 & 0 \\
 0 & 0 & 0
  \end{pmatrix}. $$
We claim that this evolution algebra does not have a two-dimensional evolution ideal generated by one element. To prove this, consider $f=m e_1 + ne_2+pe_3$. Then the ideal $I$ that it generates is the linear span of  $\{f\}\cup \{m^ie_1\}_{i\in \N}$. In order for $I$ to have a natural basis with two elements, necessarily $m=0$, implying that the dimension of $I$ is one, a contradiction.
\medskip

\noindent
{\bf Case 1.2.2} Assume that $c_1=0$ and $c_2 \neq 0$. \\
\noindent
Then $1+c_2\omega_3^2 \neq 0$. For $B'=\{e_1+\omega_2e_2+\omega_3e_3, e_2, -\omega_3c_2e_1+e_2+e_3\}$ the structure matrix is

$$ M_{B'}=
 {\begin{pmatrix}
 1+c_2\omega_3^2& 0 & c_2(1+c_2\omega_3^2) \\
 0 & 0 & 0 \\
 0 & 0 & 0
  \end{pmatrix}}.$$

Note that $A^2$ has the extension property because the first element in $B'$ is $e_1^2$, which is a natural basis of $A^2$.

Consider $B''=\left\{ \frac{1}{1+c_{2}\omega_{3}^{2}}\ e_1, e_2,  \frac{1}{\sqrt{c_2}(1+c_{2}\omega_{3}^{2})}\ e_3\right\}$. Then
  $$ M_{B''}=
 \begin{pmatrix}
1 & 0 & 1 \\
 0 & 0 & 0 \\
 0 & 0 & 0
  \end{pmatrix}. $$

We claim that this evolution algebra does not have a degenerate two-dimensional evolution ideal generated by one element. Let $f={m} e_1 + {n} e_2+{p} e_3$. Then the ideal generated by $f$, say $I$,  is the linear span of
$\{f, {p} e_1, {m} e_1\} \cup \{({m}^2+{p}^2){m}^i e_1\}_{i\in \N\cup\{0\}} \cup \{({m}^2+{p}^2)^2{m}^i e_1\}_{i\in \N\cup\{0\}}$. After some computations, in order for $I$ to have dimension 2 and to be degenerated implies ${m}=0$ or ${p}=0$, a contradiction.
\medskip

\noindent
{\bf Case 1.2.3} If $c_1\neq 0$ and $c_2 \neq 0$.
\medskip

\noindent
{\bf Case 1.2.3.1} Assume $1+\omega_2^2 c_1 \neq 0$.

\noindent
For $B'$ the natural basis such that
$
P_{B'B}=\tiny{\left(
\begin{array}{ccc}
 1 & -\omega_2c_1 & \dfrac{-\omega_3c_2}{1+c_1\omega_2^2} \\ \\
 \omega_2 & 1 &  \dfrac{-\omega_3\omega_2c_2}{1+c_1\omega_2^2} \\ \\
\omega_3 & 0 & 1
\end{array}
\right)},
$
\noindent
we obtain that
$
M_{B'}=\tiny{\begin{pmatrix}
 1+\omega_{2}^2 c_{1}+\omega_{3}^2 c_{2} & c_{1} \left(1+c_{1}\omega_{2}^2\right) & \dfrac{c_{2} \left(1+\omega_{2}^2 c_{1}+\omega_{3}^2 c_{2}\right)}{\left(1+c_{1}\omega_{2}^2\right)} \\
 0 & 0 & 0 \\
 0 & 0 & 0
\end{pmatrix}}.
$
\medskip

Now, consider the natural basis $B''=\{f_1,f_2,f_3\}$ such that
$$P_{B''B'}=
\tiny{\begin{pmatrix}
 \dfrac{1}{1+\omega_{2}^2 c_{1}+\omega_{3}^2 c_{2}} & 0 & 0 \\
 0 & \dfrac{1}{\sqrt{c_{1} \left(1+c_{1}\omega_{2}^2\right) \left(1+\omega_{2}^2 c_{1}+\omega_{3}^2 c_{2}\right)}} & 0 \\
 0 & 0 & \dfrac{\sqrt{1+c_1\omega_2^2}}{\sqrt{c_2}(1+c_1\omega_2^2+c_2\omega_3^2)}
\end{pmatrix}}
$$

\noindent
and the structure matrix is:

$$ M_{B''}=
 \begin{pmatrix}
1 & 1 & 1 \\
 0 & 0 & 0 \\
 0 & 0 & 0
  \end{pmatrix}. $$

\noindent
It is not difficult to show that this evolution algebra does not have a degenerate two-dimensional evolution ideal generated by one element.\\

\noindent
{\bf Case 1.2.3.2} Assume $1+\omega_2^2 c_1 =0$.

\noindent
Then $\omega_2\omega_3c_1c_2 \neq 0$ and so $c_1=-1/\omega_2^2$. For $B'$ such that
 $  \tiny{P_{B'B}=
 \begin{pmatrix}
 1& \dfrac{-c_2}{2} & \dfrac{2-\omega_3^2c_2}{2} \\ \\
 \omega_2 & \dfrac{\omega_2c_2}{2} & \omega_2(1+\frac{1}{2}c_2\omega_3^2) \\ \\
 \omega_3 & \dfrac{1}{\omega_3} & \omega_3
  \end{pmatrix}}$
we have
 $ \footnotesize{M_{B'}=
 \begin{pmatrix}
\omega_3^2c_2 & \dfrac{c_2}{\omega_3^2} & -\omega_3^2c_2 \\
 0 & 0 & 0 \\
 0 & 0 & 0
  \end{pmatrix}}.$

 Now, we consider the natural basis $B''$ for which
 $ \tiny{P_{B''B'}=
 \begin{pmatrix}
 \dfrac{1}{\omega_3^2c_2} & 0 & 0 \\
 0 & \dfrac{1}{c_2} & 0 \\
 0 & 0 & \dfrac{\sqrt{-1}}{\omega_3^2c_2}
  \end{pmatrix}}. $
Then, the structure matrix is
$ M_{B''}=
 \begin{pmatrix}
1 & 1 & 1 \\
0& 0 & 0\\
0 & 0 & 0
\end{pmatrix},$ which has already appeared.
\medskip

\noindent
{\bf Case 1.2.4} Suppose that $c_1\neq0$ and $c_2=0$.

\noindent
Considering the natural basis $B''=\{e_{1},e_{3},e_{2} \}$ we obtain
$ \footnotesize{M_{B''}=
 \begin{pmatrix}
 1  & 0 & c_1\\
 \omega_3 & 0 & \omega_3c_1 \\
 \omega_2 & 0 & \omega_2c_1
  \end{pmatrix}}, $
and we are in the same conditions as in Case 1.2.1.2.
\medskip

\noindent
 {\bf Case 2} Suppose that $\omega_1=0.$\\

\noindent
The structure matrix of the evolution algebra is

$$ M_{B}=
 \begin{pmatrix}
0 & 0 & 0 \\
 \omega_{2} & \omega_{2}c_{1} & \omega_{2}c_2 \\
 \omega_{3} & \omega_{3}c_{1} & \omega_{3}c_2
  \end{pmatrix}. $$

\noindent
Necessarily there exists $i \in \{2,3\}$ such that $\omega_i\neq0$. Without loss in generality we assume $\omega_2 \neq 0$.
\medskip

\noindent
  {\bf Case 2.1} Assume $c_1 \neq 0.$

\noindent
Consider the natural basis $B''=\{e_2, e_3, e_1\}$. Then
$ \footnotesize{M_{B''}=
 \begin{pmatrix}
 \omega_2c_1 & \omega_2c_2 & 1 \\
 \omega_3c_1 & \omega_3c_2 & \omega_3 \\
0       & 0      & 0
  \end{pmatrix}}$ and we are in the same conditions as in Case 1.
\medskip

\noindent
{\bf Case 2.2} If $c_1=0.$

\noindent
{\bf Case 2.2.1} Assume $c_2\omega_3 \neq0.$

\noindent
Taking the natural basis $B''=\{e_3, e_2, e_1\}$, then
$\footnotesize{ M_{B''}=
 \begin{pmatrix}
 \omega_3c_2 & 0 & \omega_3 \\
 \omega_2c_2 & 0 & \omega_2\\
0       & 0      & 0
  \end{pmatrix}} $ and we are in the same conditions as in Case 1.
\medskip

\noindent
{\bf Case 2.2.2} Suppose that $c_2\omega_3=0.$
\medskip

\noindent
{\bf Case 2.2.2.1} Assume $c_2=0.$

\noindent
Take the natural basis $B'=\{\omega_2e_2+\omega_3e_3, \frac{1}{\omega_2}e_3, e_1 \}$. Then
$$ M_{B'}=
 \begin{pmatrix}
0 & 0 & 1 \\
0 & 0 & 0 \\
 0 & 0 & 0
 \end{pmatrix}. $$
 Note that $A^2$ has the extension property.
\medskip

\noindent
{\bf Case 2.2.2.2} Assume $c_2\neq0.$

\noindent
Then $\omega_3=0$. For $B'=\{\omega_2e_2, e_1,  \frac{1}{\sqrt{c_2}}e_3\}$ we have
$$ M_{B'}=
 \begin{pmatrix}
0 & 1 & 1 \\
 0 & 0 & 0  \\
 0 & 0 & 0
  \end{pmatrix}. $$
 In this case, $A^2$ has also the extension property.

We have completed the study of all the cases and will list them in a table. All of them produces evolution algebras $A$ such that ${\rm dim}(A)=3$ and ${\rm dim}(A^2)=1$. They are mutually non-isomorphic, as will be clear from the table. We specify the following properties that are invariant under isomorphisms of evolution algebras: {Whether or not} $A^2$ has the extension property, the dimension of the annihilator of $A$, and {whether or not} $A$ has a principal degenerate two-dimensional evolution ideal.

 In every of the cases listed below we have analyzed when $A^2$ has the extension property. To compute the dimension of the annihilator we have used \cite[Proposition 2.18]{CSV1}. We also specify in the table if the evolution algebra has or not a  two-dimensional evolution ideal, which is degenerate as an evolution algebra, and which is generated by one element.

Recall that for a commutative algebra $A$ the \emph{annihilator} of $A$, denoted by
${\rm ann}(A)$ is defined to be ${\rm ann}(A)=\{x \in A \ \vert \ xA=0\}.$
\newpage

{\footnotesize{
\begin{center}
\begin{tabular}{|c||c||c||c|}
\hline
&&&\\
Type & $A^2$ \rm{has the extension} &\rm{dimension of }&\rm{A has a principal degenerate}\\
 &  \rm{property} & \rm{ann(A)} & \rm{two-dimensional evolution ideal}\\
&&&\\
\hline
\hline
&&&\\[-0.2cm]

$
\begin{pmatrix}
1 & -1 & 1 \\
1 & -1 & 1 \\
0  & 0 & 0

\end{pmatrix}
$  &
No
&
0
 &

 $I=<e_3>$

 \\
&&&\\[-0.2cm]
\hline
\hline
&&&\\[-0.2cm]

$
\begin{pmatrix}
1 & -1 & 0 \\
1 & -1 & 0 \\
0  & 0 & 0

\end{pmatrix}
$  &
No
&
1
&

$I=<e_1+e_2+e_3>$

 \\
&&&\\[-0.2cm]
\hline
\hline
&&&\\[-0.2cm]

$
\begin{pmatrix}
1 & 1 & 1 \\
0 & 0 & 0 \\
0 & 0 & 0 \\

\end{pmatrix}
$  &
Yes
&
0
&

No

\\
&&&\\[-0.2cm]
\hline
\hline
&&&\\[-0.2cm]

$
\begin{pmatrix}
1 & 0 & 1 \\
0 & 0 & 0 \\
0  & 0 & 0

\end{pmatrix}
$  &
Yes
&
1
&

No

 \\
&&&\\[-0.2cm]
\hline
\hline
&&&\\[-0.2cm]

$
\begin{pmatrix}
1 & 0 & 0 \\
0 & 0 & 0 \\
0  & 0 & 0

\end{pmatrix}
$  &
Yes
&
2
&

No

\\
&&&\\[-0.2cm]
\hline
\hline
&&&\\[-0.2cm]

$
\begin{pmatrix}
0 & 1 & 1 \\
0 & 0 & 0 \\
0  & 0 & 0

\end{pmatrix}
$  &
Yes
&
1
&

$I=<e_3>$

\\
&&&\\[-0.2cm]
\hline
\hline
&&&\\[-0.2cm]

$
\begin{pmatrix}
0 & 0 & 1 \\
0 & 0 & 0 \\
0  & 0 & 0

\end{pmatrix}
$  &
Yes
&
2
&

$I=<e_3>$

 \\

&&&\\
\hline
\end{tabular}

{\footnotesize{TABLE 1. ${\rm dim}(A^2)=1$.}}\label{tabla1}
\end{center}}}

\noindent
{\bf Case ${\rm dim}(A^{2})=2$.}

\noindent
The first step is to compute the possible matrices $P_{B'B}$ for natural basis $B$ and $B'$. Without loss in generality, we may assume that there exists a natural basis $B=\{e_1,e_2,e_3\}$ such that

\begin{equation}\label{matrizrangodos}
M_B=\begin{pmatrix}
\omega_{11} & \omega_{12} & c_1\omega_{11}+c_2\omega_{12} \\
\omega_{21} & \omega_{22} & c_1\omega_{21}+c_2\omega_{22} \\
\omega_{31}  & \omega_{32} & c_1\omega_{31}+c_2\omega_{32}
\end{pmatrix}
\end{equation}

\noindent
for some $c_1, c_2\in \K$ with $\omega_{11}\omega_{22}-\omega_{12}\omega_{21}\neq 0$.\\

Let $B'$ be another natural basis and let $P_{B'B}$ be the change of basis matrix. Write
$$
P_{B'B}=\begin{pmatrix}
p_{11} & p_{12} & p_{13} \\
p_{21} & p_{22} & p_{23} \\
p_{31}  & p_{32} & p_{33}
\end{pmatrix}.$$

Since $B'$ is a natural basis, by \eqref{ecuac3} it verifies:

\begin{eqnarray}\label{ecuacio1}
\left\lbrace
\begin{array}{lll}
\omega_{11} p_{11} p_{12}+\omega_{12} p_{21} p_{22}+(\omega_{11} c_1+\omega_{12} c_2) p_{31} p_{32}  =& 0 \\
\omega_{21} p_{11} p_{12}+\omega_{22} p_{21} p_{22}+(\omega_{21} c_1+\omega_{22} c_2) p_{31} p_{32}=&0 \\
\omega_{31} p_{11} p_{12}+\omega_{32} p_{21} p_{22}+(\omega_{31} c_1+\omega_{32} c_2) p_{31} p_{32}= &0 \\
\end{array}
\right.
\end{eqnarray}

\begin{eqnarray}\label{ecuacio2}
\left\lbrace
\begin{array}{lll}
\omega_{11} p_{11} p_{13}+\omega_{12} p_{21} p_{23}+(\omega_{11} c_1+\omega_{12} c_2) p_{31} p_{33}= &0 \\
\omega_{21} p_{11} p_{13}+\omega_{22} p_{21} p_{23}+(\omega_{21} c_1+\omega_{22} c_2) p_{31} p_{33}= &0 \\
\omega_{31} p_{11} p_{13}+\omega_{32} p_{21} p_{23}+(\omega_{31} c_1+\omega_{32} c_2) p_{31} p_{33}= &0 \\
\end{array}
\right.
\end{eqnarray}

\begin{eqnarray}\label{ecuacio3}
\left\lbrace
\begin{array}{lll}
\omega_{11} p_{12} p_{13}+\omega_{12} p_{22} p_{23}+(\omega_{11} c_1+\omega_{12} c_2) p_{32} p_{33}=& 0 \\
\omega_{21} p_{12} p_{13}+\omega_{22} p_{22} p_{23}+(\omega_{21} c_1+\omega_{22} c_2) p_{32} p_{33}=& 0 \\
\omega_{31} p_{12} p_{13}+\omega_{32} p_{22} p_{23}+(\omega_{31} c_1+\omega_{32} c_2) p_{32} p_{33}= &0
\end{array}
\right.
\end{eqnarray}

\medskip

We consider the homogeneous system \eqref{ecuacio1} in the three variables $p_{11}p_{12}$, $p_{21}p_{22}$ and $p_{31}p_{32}$. Taking into account that the rank of this system is 2, we may compute its solutions as follows:

\begin{equation}\label{cuartas}
p_{11}p_{12}=\dfrac{
\left\vert
\begin{matrix}
-(\omega_{11}c_1+\omega_{12}c_2)p_{31}p_{32} & \omega_{12} \\ \\
-(\omega_{21}c_1+\omega_{22}c_2)p_{31}p_{32} & \omega_{22} \\ \\
\end{matrix}
\right\vert}{\omega_{11}\omega_{22}-\omega_{21}\omega_{12}}= -c_1p_{31}p_{32}
\end{equation}

\begin{equation}\label{quintas}
p_{21}p_{22}=\dfrac{
\left\vert
\begin{matrix}
\omega_{11} & -(\omega_{11}c_1+\omega_{12}c_2)p_{31}p_{32} \\ \\
\omega_{21}  & -(\omega_{21}c_1+\omega_{22}c_2)p_{31}p_{32} \\ \\
\end{matrix}
\right\vert}{\omega_{11}\omega_{22}-\omega_{21}\omega_{12}}= -c_2p_{31}p_{32}
\end{equation}

In an analogous way, we may consider the systems given in \eqref{ecuacio2} and \eqref{ecuacio3}. Their solutions can be computed as follows:

\begin{equation}\label{segundas}
p_{11}p_{13}=-c_1 p_{31}p_{33};\quad
p_{21}p_{23}=-c_2 p_{31}p_{33};
\end{equation}
\noindent
and
\begin{equation}\label{terceras}
p_{12}p_{13}=-c_1 p_{32}p_{33};\quad
p_{22}p_{23}=-c_2 p_{32}p_{33}.
\end{equation}

\medskip

 \noindent
 {\bf Case 1} $c_1c_2 \neq 0.$
\medskip

In this case the annihilator is zero because there cannot be a column of zeros (apply \cite[Proposition 2.18]{CSV1}). All the evolution algebras appearing in this case will have Property (2LI), that we define.

\begin{definition}
\rm
A three-dimensional evolution algebra $A$ is said to have \emph{Property} (2LI) if for any basis $\{e_1, e_2, e_3\}$ the set $\{e_i^2, e_j^2\}$ is linearly independent, for every $i, j\in \{1, 2, 3\}$ with $i\neq j$.
\end{definition}
 \medskip

In what follows we  prove, by way of contradiction, that $P_{B'B} \in S_3\rtimes (\K^\times)^3$. Note that (see \eqref{permutationmatrix}) elements in  $P_{B'B} \in S_3\rtimes (\K^\times)^3$ are those invertible matrices in $\mathcal M_3(\K)$ having two zeros in every row and every column.

Then, let $P_{B'B} \notin S_3\rtimes (\K^\times)^3$. Assume, for example,  that $p_{31}p_{32}p_{33}\neq 0$.
By \eqref{segundas} and \eqref{terceras} we have
 that $p_{11}p_{12}p_{13}\neq 0$ and $p_{21}p_{22}p_{23}\neq 0$. If we replace $p_{11}=\dfrac{-c_1p_{31}p_{32}}{p_{12}}$ and $p_{21}=\dfrac{-c_2p_{31}p_{32}}{p_{22}}$ in \eqref{segundas} we obtain $p_{13}=\dfrac{p_{33}p_{12}}{p_{32}}$ and $p_{23}=\dfrac{p_{33}p_{22}}{p_{32}}$. Finally, if we replace these two values in \eqref{terceras}, we get $p_{12}^2=-c_1p_{32}^2$ and $p_{22}^2=-c_2p_{32}^2$. Therefore, \\
$$p_{12}=\pm \sqrt{-c_1} \ p_{32} $$
$$p_{22}=\pm \sqrt{-c_2}\ p_{32} $$

\noindent
and
$$p_{13}=\pm \sqrt{-c_1}\ p_{33} $$
$$p_{23}=\pm \sqrt{-c_2}\ p_{33} $$
$$p_{21}=\pm \sqrt{-c_2}\ p_{31} $$
$$p_{11}=\pm \sqrt{-c_1}\ p_{31} $$

Now,
$$\vert P_{B'B} \vert =p_{11}p_{22}p_{33}+ p_{12}p_{23}p_{31}+ p_{13}p_{21}p_{32}-p_{13}p_{22}p_{31}-p_{21}p_{12}p_{33}-p_{11}p_{32}p_{23}=0.$$

This is a contradiction. Therefore $p_{31}p_{32}p_{33}=0$, hence,
there exists at least one $i \in \{1,2,3\}$ such that $p_{3i}=0$. We may suppose without loss in generality that $p_{31}=0$. This means that $p_{11}p_{12}=0$, $p_{21}p_{22}=0$, $p_{11}p_{13}=0$, $p_{21}p_{23}=0$ and, obviously, $p_{31}p_{32}=0$ and $p_{31}p_{33}=0$.

We claim that $P_{B'B}$ has two zero entries in every row and column. In other words, that $P_{B'B}\in S_3 \rtimes (\K^\times)^3$.

Assume  $p_{11}=0$. Since $\vert P_{B'B} \vert\neq 0$, necessarily $p_{21}\neq 0$. So, $p_{23}=p_{22}=0$ and consequently, using \eqref{terceras}, $p_{32}p_{33}=0$ and $p_{12}p_{13}=0$. We have $p_{22}=0$ and $p_{23}=0$.

In $p_{32}p_{33}=0$ we distinguish two cases. First, assume $p_{32}=0$. Then $p_{12}\neq 0$ (because $\vert P_{B'B} \vert\neq 0$). Since $p_{12}p_{13}=0$ we get $p_{13}=0$. Use again $\vert P_{B'B} \vert\neq 0$ to obtain $p_{33}\neq 0$ and we have proved that, in this case, $P_{B'B}\in S_3 \rtimes (\K^\times)^3$.
Second, assume $p_{32}\neq 0$. Then $p_{33}=0$ and $p_{13}\neq 0$ because $\vert P_{B'B} \vert\neq 0$. Use $p_{12}p_{13}=0$ to get, reasoning as before, that $p_{12}=0$ and $p_{32}\neq 0$. This proves again $P_{B'B}\in S_3 \rtimes (\K^\times)^3$.

Now, assume $p_{11}\neq0$. Then, by \eqref{cuartas} and \eqref{segundas}, we get $p_{12}=p_{13}=0$. So, $p_{12}p_{13}=0$, $p_{32}p_{33}=0$ and $p_{22}p_{23}=0$. Now we use \eqref{cuartas}, \eqref{segundas} and \eqref{terceras} to obtain $p_{31}p_{32}=0,$ $p_{31}p_{33}=0$ and $p_{32}p_{33}=0$. Taking into account this identities and $0\neq \vert P_{B'B}\vert =p_{11}p-{22}p_{33}- p_{11}p_{32}p_{23}$ we prove $P_{B'B}\in S_3 \rtimes (\K^\times)^3$ as claimed.

Now that we know the possible matrices for $P_{B'B}$, we may look for all the possible $M_B$. By Proposition \ref{prop:numerodeceros} \eqref{numerodeceros1} all the structure matrices representing the same evolution algebra have the same number of zero entries. This is the reason for studying the classification depending on the number of non-zero entries in $M_B$ (recall that $M_B$ is the matrix given in \eqref{matrizrangodos}).

We claim that the first case to be considered is the one for which $M_B$ has four non-zero entries. Indeed, fix our attention in the first and second columns in $M_B$ as given in \eqref{matrizrangodos}. The maximum number of zero entries in that columns is four. Now, the third column can have only one zero because $c_1$ and $c_2$ are non-zero and we have two non-zero entries in the first and second columns, which are neither in the same row nor in the second column as the rank of $M_B$ is two.
Taking into account  \eqref{producto2},  we may assume that two of non-zero entries are 1. In some cases, we will be able to place one or two more 1 in a third and fourth entries. The remaining non-zero entries will be parameters.

We explain now two types of  tables we include, called ``Table $m$" and ``Table $m^\prime$".
For ``Table $m$" (for $m\neq 8, 9, 10, 11, 12, 13$), we list in the first row (starting by the second column) the five permutation matrices different from the identity. As for the second row we start with an arbitrary structure matrix under the case we are considering. Then we apply the action of an element in $S_3$ (listed at the beginning of each column) and write in the corresponding row the obtaining matrix. We start the third row with a matrix under the case we are considering and not included in the second row, and continue in this way until we reach all the possibilities for this case. In order to make easier the understanding of the reader, we distinguish in color the different possibilities that we have.
As for the second type of tables we include, the reason is the following: For given parameters (those appearing in the listed matrices), matrices in the same row of a concrete table produce isomorphic evolution algebras. Matrices appearing in different rows correspond to non-isomorphic evolution algebras. Now the question is: For matrices in the same row having different parameters, are the corresponding evolution algebras isomorphic? To answer this question we include the second type of tables, ``Table $m^\prime$".
\medskip

\noindent
 {\bf Case 1.1} $M_B$ has four non-zero entries.
\medskip

\noindent
Note that there is, necessarily, a row with all its entries equal zero, because there is no a column with all its entries equal zero (as $c_1c_2 \neq 0$). For each possible row with three zeros, there are $\binom{6}{2} = 15$ possible places where to put two zeros in the remaining rows. Because  $A$ has Property {(2LI)} a row can not have two zeros. This happens 6 times. We have to eliminate the cases in which there is a zero column (three cases). Then we have $15-6-3=6$ cases for each possible  row with three zeros. Therefore we obtain 18 cases.

\noindent
There are only four families of mutually non-isomorphic evolution algebras: those whose structure matrix is in the first column of the table below.
\bigskip

\begin{center}

\scalebox{0.65}{

\begin{tabular}{|c||c||c||c||c||c|}

\hline
&&&&&\\
&(1,2)&(1,3)&(2,3)&(1,2,3)&(1,3,2)\\
&&&&&\\
\hline
\hline
&&&&&\\[-0.2cm]

$\color{ao(english)}{
\begin{pmatrix}
0 & 0 & 0 \\
0 & 1 & 1 \\
1  & 0 & c_1

\end{pmatrix}}
$  &

$\color{ao(english)}{
 \begin{pmatrix}
1 & 0 & 1 \\
0 & 0 & 0 \\
0  & 1 & c_1
\end{pmatrix}}$
&
$\color{ao(english)}{
\begin{pmatrix}
c_1 & 0 & 1 \\
1 & 1 & 0 \\
0  & 0 & 0
\end{pmatrix}} $
 &

 $\color{ao(english)}{
 \begin{pmatrix}
0 & 0 & 0 \\
1 & c_1 & 0 \\
0  & 1 & 1
\end{pmatrix}}$
&
$\color{ao(english)}{
\begin{pmatrix}
c_1 & 1 & 0 \\
0 & 0 & 0 \\
1  & 0 & 1
\end{pmatrix}}$
&
$\color{ao(english)}{
\begin{pmatrix}
1 & 1 & 0 \\
0 & c_1 & 1 \\
0  & 0 & 0
\end{pmatrix}}
$ \\
&&&&&\\[-0.2cm]
\hline
\hline
&&&&&\\[-0.2cm]

$\color{ao(english)}{
\begin{pmatrix}
0 & 0 & 0 \\
0 & 1 & c_2 \\
1  & 1 & 0

\end{pmatrix}}
$  &

$\color{ao(english)}{
 \begin{pmatrix}
1 & 0 & c_2 \\
0 & 0 & 0 \\
1  & 1 & 0

\end{pmatrix}}$
&
$\color{ao(english)}{
\begin{pmatrix}
0 & 1 & 1 \\
c_2 & 1 & 0 \\
0  & 0 & 0
\end{pmatrix}} $
 &

 $\color{ao(english)}{
 \begin{pmatrix}
0 & 0 & 0 \\
1 & 0 & 1 \\
0  & c_2 & 1
\end{pmatrix}}$
&
$\color{ao(english)}{
\begin{pmatrix}
0 & 1 & 1 \\
0 & 0 & 0 \\
c_2  & 0 & 1
\end{pmatrix}}
$
&
$\color{ao(english)}{
\begin{pmatrix}
1 & c_2 & 0 \\
1 & 0 & 1 \\
0  & 0 & 0
\end{pmatrix}}
$ \\
&&&&&\\[-0.2cm]
\hline
\hline
&&&&&\\[-0.2cm]

$\color{ao(english)}{
\begin{pmatrix}
0 & 0 & 0 \\
1 & 0 & c_1 \\
1 & 1 & 0

\end{pmatrix}}
$  &

$\color{ao(english)}{
 \begin{pmatrix}
0 & 1 & c_1 \\
0 & 0 & 0 \\
1  & 1 & 0

\end{pmatrix}}$
&
$\color{ao(english)}{
\begin{pmatrix}
0 & 1 & 1 \\
c_1 & 0 & 1 \\
0  & 0 & 0
\end{pmatrix}} $
 &
${
\begin{pmatrix}
0 & 0 & 0 \\
1 & 0 & 1 \\
1  & c_1 & 0
\end{pmatrix}}$
&
${
\begin{pmatrix}
0 & 1 & 1 \\
0 & 0 & 0 \\
c_1  & 1 & 0
\end{pmatrix}}
$
&
${
\begin{pmatrix}
0 & c_1 & 1 \\
1 & 0 & 1 \\
0  & 0 & 0
\end{pmatrix}}
$ \\
&&&&&\\[-0.2cm]
\hline
\hline
&&&&&\\[-0.2cm]

$\color{ao(english)}{
\begin{pmatrix}
0 & 0 & 0 \\
1 & 1 & 0 \\
1  & 0 & c_1

\end{pmatrix}}
$  &

$\color{ao(english)}{
 \begin{pmatrix}
1 & 1 & 0 \\
0 & 0 &  0 \\
0  & 1 & c_1

\end{pmatrix}}$
&
$\color{ao(english)}{
\begin{pmatrix}
c_1 & 0 & 1 \\
0 & 1 & 1 \\
0  & 0 & 0

\end{pmatrix} }$
 &
 ${
 \begin{pmatrix}
0 & 0 & 0 \\
1 & c_1 & 0 \\
1  & 0 & 1
\end{pmatrix}}$
&
${
\begin{pmatrix}
c_1 & 1 & 0 \\
0 & 0 & 0 \\
0  & 1 & 1
\end{pmatrix}}
$
&
${
\begin{pmatrix}
1 & 0 & 1 \\
0 & c_1 & 1 \\
0  & 0 & 0
\end{pmatrix}}
$ \\

&&&&&\\
\hline
\end{tabular}}

\footnotesize{TABLE 2. ${\rm dim}(A^2)=2$; $c_1, c_2\neq 0$; ${\rm dim}({\rm ann}(A))=0$; $A$ has Property (2LI);}\label{tabla2}

\footnotesize{four non-zero entries.}
\end{center}
\medskip

Now, we study if every resulting family of evolution algebras contains isomorphic evolution algebras. The procedure we have used is the following: we start with one $M_B$ and study if there are matrices $P_{B'B}$ such that $M_{B'}$ is in the same family. For the computations we have used Mathematica. The program can be found in Appendix \ref{apendiceb}.
%
This explanation  serves for all the cases.

\begin{center}

\scalebox{0.7}{

\begin{tabular}{|c||c||c|}

\hline
&&\\
$M_{B}$& $P_{B'B}$ & $M_{B'}$\\
&&\\
\hline
\hline
&&\\

${
\begin{pmatrix}
0 & 0 & 0 \\
0 & 1 & 1 \\
1  & 0 & c_1

\end{pmatrix}}
$  &

${
 \begin{pmatrix}
\sqrt{-1} & 0 & 0 \\
0 & 1 & 0 \\
0  & 0 & -1
\end{pmatrix}}$
&
${
\begin{pmatrix}
0 & 0 & 0 \\
0 & 1 & 1 \\
1  & 0 & -c_1
\end{pmatrix}} $
 \\
&&\\
\hline
\hline
&&\\

${
\begin{pmatrix}
0 & 0 & 0 \\
1 & 0 & c_1 \\
1  & 1 & 0

\end{pmatrix}}
$  &

${
 \begin{pmatrix}
\dfrac{1}{\sqrt{c_1}} & 0 & 0 \\
0 &  0  & \dfrac{1}{c_1}\\
0  & \dfrac{1}{c_1} &  0
\end{pmatrix}}$
&
${
\begin{pmatrix}
0 & 0 & 0 \\
1 & 0 & \dfrac{1}{c_1}  \\
1  & 1 & 0
\end{pmatrix}} $
\\
&&\\
\hline
\hline
&&\\

${
\begin{pmatrix}
0 & 0 & 0 \\
1 & 1 & 0 \\
1  & 0 & c_1

\end{pmatrix}}
$  &

${
 \begin{pmatrix}
\dfrac{1}{\sqrt{c_1}} & 0 & 0 \\
0 &  0  & \dfrac{1}{c_1}\\
0  & \dfrac{1}{c_1} &  0
\end{pmatrix}}$
&
${
\begin{pmatrix}
0 & 0 & 0 \\
1 & 1 & 0 \\
1  & 0 & \dfrac{1}{c_1}
\end{pmatrix}} $
\\
&&\\
\hline
\end{tabular}}

\footnotesize{TABLE $2^\prime$.}
\end{center}

\noindent
{\bf Case 1.2} $M_B$ has five non-zero entries.
\medskip

\noindent
The structure matrix must have a zero column. So, for each possible zero row, there exist $\binom{6}{3} = 6$ possible places where to write the remaining zero. Therefore, there are $6 \cdot 3=18$ cases. They appear in the list that follows. The parametric families of mutually non-isomorphic evolution algebras are three and appear in the first column of the table below.

\begin{center}

\scalebox{0.7}{

\begin{tabular}{|c||c||c||c||c||c|}
\hline
&&&&&\\
&(1,2)&(1,3)&(2,3)&(1,2,3)&(1,3,2)\\
&&&&&\\
\hline
\hline
&&&&&\\[-0.2cm]
$\color{ao(english)}{
\begin{pmatrix}
0 & 0 & 0 \\
0 & 1 & 1 \\
1  & \alpha & c_1+\alpha

\end{pmatrix}}
$  &

$\color{ao(english)}{
 \begin{pmatrix}
1 & 0 & 1 \\
0 & 0 & 0 \\
\alpha  & 1 & c_1 +\alpha

\end{pmatrix}}$
&
$\color{ao(english)}{
\begin{pmatrix}
c_1 + \alpha & \alpha & 1 \\
1 & 1 & 0 \\
0  & 0 & 0

\end{pmatrix} }$
 &

 $\color{ao(english)}{
 \begin{pmatrix}
0 & 0 & 0 \\
1 & c_1 + \alpha & \alpha \\
0  & 1 & 1
\end{pmatrix}}$
&
$\color{ao(english)}{
\begin{pmatrix}
c_1 + \alpha  & 1 & \alpha \\
0 & 0 & 0 \\
1  & 0 & 1
\end{pmatrix}}$
&
$\color{ao(english)}{
\begin{pmatrix}
1 & 1 & 0 \\
\alpha & c_1 + \alpha & 1 \\
0  & 0 & 0
\end{pmatrix}}
$ \\
&&&&&\\[-0.2cm]
\hline
\hline
&&&&&\\[-0.2cm]
$\color{ao(english)}{
\begin{pmatrix}
0 & 0 &  0 \\
1 & 0 & c_1 \\
1  & 1 & c_1 + c_2
\end{pmatrix}}
$  &

$\color{ao(english)}{
 \begin{pmatrix}
0 & 1 & c_1 \\
0 & 0 & 0 \\
1  & 1 & c_1 + c_2

\end{pmatrix}}$
&
$\color{ao(english)}{
\begin{pmatrix}
c_1 + c_2 & 1 & 1 \\
c_1 & 0 & 1 \\
0  & 0 & 0
\end{pmatrix} }$
 &

 $\color{ao(english)}{
 \begin{pmatrix}
0 & 0 & 0 \\
1 & c_1 + c_2 & 1 \\
1  & c_1 & 0
\end{pmatrix}}$
&
$\color{ao(english)}{
\begin{pmatrix}
c_1 + c_2 & 1 & 1 \\
0 & 0 & 0 \\
c_1  & 1 & 0
\end{pmatrix}}$
&
$\color{ao(english)}{
\begin{pmatrix}
0 & c_1 & 1 \\
1 & c_1 + c_2 & 1 \\
0  & 0 & 0
\end{pmatrix}}
$ \\
&&&&&\\[-0.2cm]
\hline
\hline
&&&&&\\[-0.2cm]
$\color{ao(english)}{
\begin{pmatrix}
0 & 0 & 0 \\
1 & 1 & c_1 + c_2 \\
1  & 0 & c_1

\end{pmatrix}}
$  &

$\color{ao(english)}{
 \begin{pmatrix}
1 & 1 & c_1 + c_2 \\
0 & 0 & 0 \\
0 & 1 & c_1

\end{pmatrix}}$
&
$\color{ao(english)}{
\begin{pmatrix}
c_1 & 0 & 1 \\
c_1 + c_2 & 1 & 1 \\
0  & 0 & 0

\end{pmatrix} }$
 &

 $\color{ao(english)}{
 \begin{pmatrix}
0 & 0 & 0 \\
1 & c_1 & 0 \\
1  & c_1 + c_2 & 1
\end{pmatrix}}$
&
$\color{ao(english)}{
\begin{pmatrix}
c_1 & 1 & 0 \\
0 & 0 & 0 \\
c_1 + c_2  & 1 & 1
\end{pmatrix}}
$
&
$\color{ao(english)}{
\begin{pmatrix}
1 & c_1 + c_2 & 1 \\
0 & c_1 & 1 \\
0  & 0 & 0
\end{pmatrix}}
$ \\
&&&&&\\
\hline
\end{tabular}}

\footnotesize{TABLE 3. ${\rm dim}(A^2)=2$; $\alpha, c_1, c_1+\alpha, c_1+c_2 \neq 0$; ${\rm dim}({\rm ann}(A))=0$;  }

\footnotesize{$A$ has Property (2LI). Five non-zero entries.}
\label{tabla3}
\end{center}

\newpage

\begin{center}

\scalebox{0.7}{

\begin{tabular}{|c||c||c|}

\hline
&&\\
$M_{B}$& $P_{B'B}$ & $M_{B'}$\\
&&\\
\hline
\hline
&&\\

${
\begin{pmatrix}
0 & 0 & 0 \\
0 & 1 & 1 \\
1  & \alpha & c_1 + \alpha

\end{pmatrix}}
$  &

${
 \begin{pmatrix}
\sqrt{-1} & 0 & 0 \\
0 & 1 & 0 \\
0  & 0 & -1
\end{pmatrix}}$
&
${
\begin{pmatrix}
0 & 0 & 0 \\
0 & 1 & 1 \\
1  & -\alpha & -c_1 -\alpha
\end{pmatrix}} $
 \\
 &&\\
\hline

\end{tabular}}

\footnotesize{TABLE $3^\prime$.}
\end{center}

\bigskip

\noindent
{\bf Case 1.3} $M_B$ has six non-zero entries.
\medskip

\noindent
There exists $\binom{9}{3} = 84$ possibilities to place three zeros. As the structure matrix has Property (2LI), it can have neither two zeros in a row nor a zero column. So, for each place, we eliminate the six cases where to write two zeros in a row. Then, we remove $9 \cdot 6+3 =57$ cases. Therefore, we have 27 cases. The mutually non-isomorphic parametric families of evolution algebras are seven and are listed in the first column of the table below.

\noindent
In some cases, the parameters $\alpha$, $\beta$ and $c_1$ must satisfy certain conditions:

In the fourth row: $c_1 \alpha \beta +1 \neq 0$. In the fifth row: $c_1 \alpha + \beta \neq 0$ and in the seventh row: $c_2 \alpha + c_1 \neq 0$.

\begin{center}
\scalebox{0.6}{
\begin{tabular}{|c||c||c||c||c||c|}

\hline
&&&&&\\
&(1,2)&(1,3)&(2,3)&(1,2,3)&(1,3,2)\\
&&&&&\\
\hline
\hline
&&&&&\\[-0.2cm]

$\color{ao(english)}{
\begin{pmatrix}
0 & 0 & 0 \\
1 & 1 & c_1+c_2 \\
1  & \alpha & c_1 + \alpha c_2

\end{pmatrix}}
$  &

$\color{ao(english)}{
 \begin{pmatrix}
1 & 1 & c_1+c_2 \\
0 & 0 & 0 \\
\alpha  & 1 & c_1 + \alpha c_2

\end{pmatrix}}$
&
$\color{ao(english)}{
\begin{pmatrix}
c_1 + \alpha  c_2 & \alpha & 1 \\
c_1 + c_2 & 1 & 1 \\
0  & 0 & 0

\end{pmatrix} }$
 &

 ${
 \begin{pmatrix}
0 & 0 & 0 \\
1 & c_1 + \alpha c_2 & \alpha \\
1  & c_1+c_2 & 1
\end{pmatrix}}$
&
${
\begin{pmatrix}
c_1 + \alpha c_2 & 1 & \alpha \\
0 & 0 & 0 \\
c_1 + c_2  & 1 & 1
\end{pmatrix}}$
&
${
\begin{pmatrix}
1 & c_1 + c_2 & 1 \\
\alpha & c_1 + \alpha  c_2 & 1 \\
0  & 0 & 0
\end{pmatrix}}
$ \\
&&&&&\\[-0.2cm]
\hline
\hline
&&&&&\\[-0.2cm]
$\color{ao(english)}{
\begin{pmatrix}
0 & 1 &  1 \\
0 & 1 & 1 \\
\alpha  & 0 & \alpha  c_2
\end{pmatrix}}
$  &

$\color{ao(english)}{
 \begin{pmatrix}
1 & 0 & 1 \\
1 & 0 & 1 \\
0  & \alpha & \alpha  c_2

\end{pmatrix}}$
&
$\color{ao(english)}{
\begin{pmatrix}
\alpha  c_2 & 0 & \alpha \\
1 & 1 & 0 \\
1  & 1 & 0
\end{pmatrix} }$
 &

 $\color{ao(english)}{
 \begin{pmatrix}
0 & 1 & 1 \\
\alpha & \alpha c_2 & 0 \\
0  & 1 & 1
\end{pmatrix}}$
&
$\color{ao(english)}{
\begin{pmatrix}
\alpha  c_2 & \alpha & 0 \\
1 & 0 & 1 \\
1  & 0 & 1
\end{pmatrix}}$
&
$\color{ao(english)}{
\begin{pmatrix}
1 & 1 & 0 \\
0 & \alpha c_2 & \alpha \\
1  & 1 & 0
\end{pmatrix}}
$ \\
\hline
\end{tabular}}

\footnotesize{TABLE 4. ${\rm dim}(A^2)=2$; $\alpha,\, \beta,\, c_1+\alpha c_2, \, c_1, \, c_2, c_1+c_2 \neq 0$; ${\rm dim}({\rm ann}(A))=0$; }\label{tabla4}

\footnotesize{$A$ has Property (2LI); six non-zero entries.}
\end{center}
\bigskip

\begin{center}

\scalebox{0.7}{

\begin{tabular}{|c||c||c|}

\hline
&&\\
$M_{B}$& $P_{B'B}$ & $M_{B'}$\\
&&\\
\hline
\hline
&&\\[-0.2cm]
${
\begin{pmatrix}
0 & 0 & 0 \\
1 & 1 & c_1+c_2 \\
1 & \alpha  & c_1 + \alpha c_2

\end{pmatrix}}
$  &

${
 \begin{pmatrix}
 \dfrac{1}{\sqrt{c_1+ c_2 \alpha }}& 0 & 0 \\
0 &  0 & \dfrac{1}{c_1 + \alpha c_2}  \\
0 & \dfrac{1}{c_1 + c_2 \alpha}  &   0
\end{pmatrix}}$
&
${
\begin{pmatrix}
0 & 0 & 0 \\
1 & 1 & \dfrac{\alpha}{c_1+c_2 \alpha}  \\
1 & \dfrac{c_1 + c_2}{c_1 + c_2 \alpha}  & \dfrac{1}{c_1 + c_2 \alpha}
\end{pmatrix}} $
\\

&&\\
\hline
\hline
&&\\
${
\begin{pmatrix}
0 & 1 & 1 \\
0 & 1 & 1 \\
\alpha & 0  & \alpha c_2

\end{pmatrix}}
$  &

${
 \begin{pmatrix}
1 & 0 & 0 \\
0 &  1 & 0  \\
0 & 0  & -1
\end{pmatrix}}$
&
${
\begin{pmatrix}
0 & 1 & 1 \\
0 & 1 & 1 \\
-\alpha & 0  & -c_2 \alpha
\end{pmatrix}} $
\\
\hline
\end{tabular}
}

\end{center}

\begin{center}

\scalebox{0.7}{

\begin{tabular}{|c||c||c|}

\hline
&&\\
$M_{B}$& $P_{B'B}$ & $M_{B'}$\\
&&\\
\hline
\hline
&&\\[-0.2cm]

${
\begin{pmatrix}
0 & \alpha &  1 \\
1 & 0 & c_1 \\
\beta & 1 & 0
\end{pmatrix}}
$  &
${\footnotesize{
 \begin{pmatrix}
0 & -{\phi} & 0  \\
 0 & 0 & {\phi}^2   \\
 {\phi}^4 & 0 & 0
\end{pmatrix}}}$
&
${\footnotesize{
\begin{pmatrix}
 0 & -{\phi}^5 \beta   & 1   \\
1 & 0 &  -{\phi}^3 \alpha \\
{\phi}^6 c_1 & 1 & 0
\end{pmatrix}} }$
\\
&&\\
\hline
\hline
&&\\[-0.2cm]
${
\begin{pmatrix}
0 & \alpha &  1 \\
1 & 0 & c_1 \\
\beta & 1 & 0
\end{pmatrix}}
$  &
${\footnotesize{
 \begin{pmatrix}
0 & {\phi}^2 & 0  \\
0 & 0 & {\phi}^4  \\
  -{\phi} & 0 & 0
\end{pmatrix}}}$
&
${\footnotesize{
\begin{pmatrix}
 0 & -{\phi}^3 \beta   & 1   \\
1 & 0 &  {\phi}^6 \alpha \\
-{\phi}^5 c_1  & 1 & 0
\end{pmatrix}} }$
\\
\hline
\end{tabular}
}

\end{center}

\begin{center}

\scalebox{0.7}{

\begin{tabular}{|c||c||c|}

\hline
&&\\
$M_{B}$& $P_{B'B}$ & $M_{B'}$\\
&&\\
\hline
\hline
&&\\[-0.2cm]

${
\begin{pmatrix}
0 & \alpha &  1 \\
1 & 0 & c_1 \\
\beta & 1 & 0
\end{pmatrix}}
$  &
${\footnotesize{
 \begin{pmatrix}
0 & 0 & \dfrac{1}{\sqrt[7]{c_1^2 \alpha  \beta^4}} \\
0 &  \dfrac{1}{\sqrt[7]{c_1 \alpha^4  \beta^2}} & 0 \\
\dfrac{1}{\sqrt[7]{c_1^4 \alpha^2  \beta}} & 0 & 0
\end{pmatrix}}}$
&
${\footnotesize{
\begin{pmatrix}
 0 & \sqrt[7]{\dfrac{c_1^2}{\alpha^6 \beta^3}}   & 1   \\
1 & 0 &  \sqrt[7]{\dfrac{\alpha^2}{c_1^3 \beta^6}} \\
\sqrt[7]{\dfrac{\beta^2}{c_1^6 \alpha^3 }}  & 1 & 0
\end{pmatrix}} }$
\\

&&\\
\hline
\hline
&&\\[-0.2cm]
${
\begin{pmatrix}
0 & \alpha &  1 \\
1 & 0 & c_1 \\
\beta & 1 & 0
\end{pmatrix}}
$  &
${\footnotesize{
 \begin{pmatrix}
0 & 0 & \dfrac{{\phi}^2}{\sqrt[7]{c_1^2 \alpha \beta^4}} \\
0 &  \dfrac{-{\phi}}{\sqrt[7]{c_1 \alpha^4  \beta^2}} & 0 \\
\dfrac{{\phi}^4}{\sqrt[7]{c_1^4 \alpha^2  \beta}} & 0 & 0
\end{pmatrix}}}$
&
${\footnotesize{
\begin{pmatrix}
 0 & -{\phi}^5 \sqrt[7]{\dfrac{ c_1^2}{\alpha^6 \beta^3}}   & 1   \\
1 & 0 &  -{\phi}^3 \sqrt[7]{\dfrac{\alpha^2}{c_1^3 \beta^6}} \\
{\phi}^6\sqrt[7]{\dfrac{\beta^2}{c_1^6 \alpha^3 }}  & 1 & 0
\end{pmatrix}} }$
\\
\hline
\end{tabular}
}
\end{center}

\begin{center}

\scalebox{0.7}{

\begin{tabular}{|c||c||c|}

\hline
&&\\
$M_{B}$& $P_{B'B}$ & $M_{B'}$\\
&&\\
\hline
\hline
&&\\[-0.2cm]

${
\begin{pmatrix}
0 & \alpha &  1 \\
1 & 0 & c_1 \\
\beta & 1 & 0
\end{pmatrix}}
$  &
${\footnotesize{
 \begin{pmatrix}
 0 & \dfrac{1}{\sqrt[7]{c_1^2 \alpha  \beta^4}} & 0 \\
  \dfrac{1}{\sqrt[7]{c_1 \alpha^4  \beta^2}} & 0 & 0 \\
 0 & 0 & \dfrac{1}{\sqrt[7]{c_1^4 \alpha^2  \beta}}
\end{pmatrix}}}$
&
${\footnotesize{
\begin{pmatrix}
 0 & \sqrt[7]{\dfrac{\alpha^2}{c_1^3 \beta^6}}   & 1   \\
1 & 0 &  \sqrt[7]{\dfrac{\beta^2}{c_1^6 \alpha^3 }} \\
\sqrt[7]{\dfrac{c_1^2}{\alpha^6 \beta^3}}  & 1 & 0
\end{pmatrix}} }$
\\

&&\\
\hline
\hline
&&\\[-0.2cm]
${
\begin{pmatrix}
0 & \alpha &  1 \\
1 & 0 & c_1 \\
\beta & 1 & 0
\end{pmatrix}}
$  &
${\footnotesize{
 \begin{pmatrix}
0 & \dfrac{{\phi}^2}{\sqrt[7]{c_1^2 \alpha  \beta^4}} & 0  \\
\dfrac{-{\phi}}{\sqrt[7]{c_1 \alpha^4  \beta^2}} & 0 & 0 \\
0 & 0 & \dfrac{{\phi}^4}{\sqrt[7]{c_1^4 \alpha^2  \beta}}
\end{pmatrix}}}$
&
${\footnotesize{
\begin{pmatrix}
 0 & -{\phi}^3\sqrt[7]{\dfrac{ \alpha^2}{c_1^3 \beta^6}}   & 1   \\
1 & 0 &  {\phi} ^6\sqrt[7]{\dfrac{\beta^2}{c_1^6 \alpha^3 }} \\
-{\phi}^5\sqrt[7]{\dfrac{ c_1^2}{\alpha^6 \beta^3}}  & 1 & 0
\end{pmatrix}} }$
\\
\hline
\end{tabular}
}

\end{center}

\begin{center}

\scalebox{0.7}{

\begin{tabular}{|c||c||c|}

\hline
&&\\
$M_{B}$& $P_{B'B}$ & $M_{B'}$\\
&&\\
\hline
\hline
&&\\[-0.2cm]

${
\begin{pmatrix}
0 & \alpha &  1 \\
1 & 0 & c_1 \\
\beta & 1 & 0
\end{pmatrix}}
$  &
${\footnotesize{
 \begin{pmatrix}
 \dfrac{1}{\sqrt[7]{c_1^2 \alpha  \beta^4}} & 0 & 0 \\
0 & 0 &  \dfrac{1}{\sqrt[7]{c_1 \alpha^4  \beta^2}}  \\
 0 & \dfrac{1}{\sqrt[7]{c_1^4 \alpha^2  \beta}} & 0
\end{pmatrix}}}$
&
${\footnotesize{
\begin{pmatrix}
 0 & \sqrt[7]{\dfrac{\beta^2}{c_1^6 \alpha^3 }}   & 1   \\
1 & 0 &  \sqrt[7]{\dfrac{c_1^2}{\alpha^6 \beta^3}} \\
\sqrt[7]{\dfrac{\alpha^2}{c_1^3 \beta^6}}  & 1 & 0
\end{pmatrix}} }$
\\

&&\\
\hline
\hline
&&\\[-0.2cm]
${
\begin{pmatrix}
0 & \alpha &  1 \\
1 & 0 & c_1 \\
\beta & 1 & 0
\end{pmatrix}}
$  &
${\footnotesize{
 \begin{pmatrix}
\dfrac{{\phi}^2}{\sqrt[7]{c_1^2 \alpha  \beta^4}} & 0  & 0 \\
0 & 0 & \dfrac{-{\phi}}{\sqrt[7]{c_1 \alpha^4  \beta^2}} \\
0 &  \dfrac{{\phi}^4}{\sqrt[7]{c_1^4 \alpha^2  \beta}}& 0
\end{pmatrix}}}$
&
${\footnotesize{
\begin{pmatrix}
 0 & {\phi}^6 \sqrt[7]{\dfrac{\beta^2}{c_1^6 \alpha^3 }}   & 1   \\
1 & 0 &  -{\phi}^5\sqrt[7]{\dfrac{ c_1^2}{\alpha^6 \beta^3}} \\
-{\phi}^3\sqrt[7]{\dfrac{  \alpha^2}{c_1^3 \beta^6}}  & 1 & 0
\end{pmatrix}} }$
\\
\hline
\end{tabular}
}

\end{center}

\begin{center}

\scalebox{0.7}{

\begin{tabular}{|c||c||c|}

\hline
&&\\
$M_{B}$& $P_{B'B}$ & $M_{B'}$\\
&&\\
\hline
\hline
&&\\

${
\begin{pmatrix}
0 & 1 & c_2 \\
\alpha & 1 & 0 \\
1  & 0 & c_1
\end{pmatrix}}
$  &

${
 \begin{pmatrix}
 \dfrac{c_2}{{c_1} ^2 } & 0 & 0 \\
0 & 0 & \dfrac{{c_2}^2 \alpha  }{{c_1}^4 }    \\
0  &  \dfrac{1}{c_1} & 0
\end{pmatrix}}$
&
${
\begin{pmatrix}
0 & 1 & \dfrac{c_2 ^3 \alpha ^2 }{ c_1 ^6} \\
\dfrac{c_2^2}{c_1^3} & 1 & 0  \\
1  & 0 & \dfrac{{c_2}^2 \alpha}{{c_1}^4}
\end{pmatrix}} $
\\
&&\\
\hline
\hline
&&\\

${
\begin{pmatrix}
1 & 0 & c_1 \\
1 & \alpha & 0 \\
0  & 1 & c_2
\end{pmatrix}}
$  &

${
 \begin{pmatrix}
0 & 0 & \dfrac{c_1}{\alpha ^4} \\
\dfrac{1}{\alpha} & 0 & 0     \\
0 & \dfrac{1}{\alpha ^2}   & 0
\end{pmatrix}}$
&
${
\begin{pmatrix}
1 & 0 & \dfrac{c_1^2} {\alpha ^7} \\
1 & \dfrac{c_2}{\alpha ^2} & 0  \\
0 & 1 & \dfrac{c_1}{\alpha^4}
\end{pmatrix}} $
\\

&&\\
\hline
\hline
&&\\

${
\begin{pmatrix}
1 & 0 & c_1 \\
1 & \alpha & 0 \\
0  & 1 & c_2
\end{pmatrix}}
$  &

${
 \begin{pmatrix}
0 &  \dfrac{c_1}{c_2 ^2} & 0\\
0 & 0 & \dfrac{c_1^2}{c_2^4}\\
0 & \dfrac{1}{\alpha ^2}   & 0
\end{pmatrix}}$
&
${
\begin{pmatrix}
1 & 0 & \dfrac{c_1^2} {\alpha ^7} \\
1 & \dfrac{c_2}{\alpha ^2} & 0  \\
0 & 1 & \dfrac{c_1}{\alpha^4}
\end{pmatrix}} $
\\
&&\\
\hline

\end{tabular}}

\footnotesize{TABLE $4^\prime$.}
\end{center}

\noindent
{\bf Case 1.4} $M_B$ has seven non-zero entries.
\medskip

\noindent
There are $\binom{9}{2}=36$ possibilities to place two zeros. But we have to eliminate the cases in which there are two zeros in the same row. So, there are $36-9=27$ cases. The parameters {$c_1, c_2$ and} $\beta$ must satisfy certain conditions: in the third {and in the fifth rows $c_2 \beta + c_1 = 0$}. There are six mutually non-isomorphic evolution algebras, which are listed below.
\newpage

\begin{center}
\scalebox{0.65}{
\begin{tabular}{|c||c||c||c||c||c|}

\hline
&&&&&\\
&(1,2)&(1,3)&(2,3)&(1,2,3)&(1,3,2)\\
&&&&&\\
\hline
\hline
&&&&&\\[-0.2cm]

$\color{ao(english)}{
\begin{pmatrix}
0 & 1 & 1 \\
0 & 1 & 1 \\
\alpha  & \beta & c_1 \alpha + \beta
\end{pmatrix}}
$  &

$\color{ao(english)}{
\begin{pmatrix}
1 & 0 & 1 \\
1 & 0 & 1 \\
\beta   & \alpha & c_1 \alpha  + \beta
\end{pmatrix} }$

&

$\color{ao(english)}{
\begin{pmatrix}
c_1 \alpha  + \beta & \beta  & \alpha \\
1 & 1 & 0 \\
1  & 1 & 0
\end{pmatrix}}$
&
$\color{ao(english)}{
 \begin{pmatrix}
0 & 1 & 1 \\
\alpha & c_1 \alpha + \beta & \beta  \\
0  & 1 & 1
\end{pmatrix}}$
&
$\color{ao(english)}{
\begin{pmatrix}
c_1 \alpha  + \beta & \alpha & \beta \\
1 & 0 & 1 \\
1  & 0 & 1
\end{pmatrix}}
$
&
$\color{ao(english)}{
\begin{pmatrix}
1 & 1 & 0 \\
\beta & c_1 \alpha + \beta & \alpha \\
1  & 1 & 0
\end{pmatrix}}
$ \\
&&&&&\\
\hline
\hline
&&&&&\\[-0.2cm]

$\color{ao(english)}{
\begin{pmatrix}
0 & 1 & 1 \\
1 & 0 & c_1 \\
\alpha  & \beta & c_1\alpha+\beta
\end{pmatrix}}
$  &

${
 \begin{pmatrix}
0 & 1 & c_1 \\
1 & 0 & 1 \\
\beta  & \alpha & c_1\alpha+\beta
\end{pmatrix}}$
&
$\color{ao(english)}{
\begin{pmatrix}
c_1\alpha+\beta & \beta & \alpha \\
c_1 & 0 & 1 \\
1  & 1 & 0
\end{pmatrix} }$
 &

$\color{ao(english)}{
\begin{pmatrix}
0 & 1 & 1 \\
\alpha & c_1\alpha+\beta & \beta \\
1  & c_1 & 0
\end{pmatrix}}$
&
$
\begin{pmatrix}
c_1\alpha+\beta & \alpha & \beta \\
1 & 0 & 1 \\
c_1  & 1 & 0
\end{pmatrix}
$
&
$
\begin{pmatrix}
0 & c_1 & 1 \\
\beta & c_1\alpha+\beta & \alpha \\
1  & 1 & 0
\end{pmatrix}
$ \\
\hline
\end{tabular}}

\footnotesize{TABLE 5. ${\rm dim}(A^2)=2$; $\alpha, \,\beta, c_1 \alpha+\beta, \, \alpha + c_2, c_1, c_2, c_1 + c_2 \neq 0$; ${\rm dim}({\rm ann}(A))=0$; }\label{tabla5}

\footnotesize{$A$ has Property (2LI); seven non-zero entries.}
\end{center}
\bigskip

\begin{center}

\scalebox{0.7}{

\begin{tabular}{|c||c||c|}

\hline
&&\\
$M_{B}$& $P_{B'B}$ & $M_{B'}$\\
&&\\[-0.2cm]
\hline
\hline
&&\\[-0.2cm]
${
\begin{pmatrix}
0 & 1 & 1 \\
0 & 1 & 1 \\
\alpha & \beta & c_1 \alpha  +\beta \\
\end{pmatrix}}
$  &

${
 \begin{pmatrix}
 1 & 0 & 0 \\
0 &  1 & 0  \\
0 & 0  & -1
\end{pmatrix}}$
&
${
\begin{pmatrix}
0 & 1 & 1 \\
0 & 1 & 1 \\
-\alpha & -\beta & -c_1 \alpha  -\beta \\
\end{pmatrix}} $
\\

&&\\[-0.2cm]
\hline
\hline
&&\\[-0.2cm]

${
\begin{pmatrix}
0 & 1 & 1 \\
1 & 0 & c_1 \\
\alpha  & \beta & c_1 \alpha  + \beta
\end{pmatrix}}
$  &

${
 \begin{pmatrix}
1 & 0 & 0 \\
0 & 1 & 0   \\
0 & 0  &  -1
\end{pmatrix}}$
&
${
\begin{pmatrix}
0 & 1 & 1 \\
1 & 0 & c_1  \\
- \alpha  & -\beta  & -c_1 \alpha - \beta
\end{pmatrix}} $
\\
\hline

\end{tabular}}

\end{center}

\begin{center}

\scalebox{0.7}{

\begin{tabular}{|c||c||c|}

\hline
&&\\
$M_{B}$& $P_{B'B}$ & $M_{B'}$\\
&&\\[-0.2cm]
\hline
\hline
&&\\[-0.2cm]

${
\begin{pmatrix}
0 & 1 & 1 \\
1 & 0 & c_1 \\
\alpha  & \beta & c_1 \alpha  + \beta
\end{pmatrix}}
$  &

${
 \begin{pmatrix}
0 & -\zeta & 0 \\
\zeta^2 & 0 & 0   \\
0 & 0  &  \dfrac{\zeta}{\sqrt{c_1}}
\end{pmatrix}}$
&
${
\begin{pmatrix}
0 & 1 & 1 \\
1 & 0 & \dfrac{-\zeta}{c_1}  \\
-\sqrt{c_1} \beta  &  \zeta \sqrt{c_1} \alpha  & \dfrac{\zeta (c_1 \alpha + \beta)}{\sqrt{c_1}}
\end{pmatrix}} $
\\

&&\\[-0.2cm]
\hline
\hline
&&\\[-0.2cm]

${
\begin{pmatrix}
0 & 1 & 1 \\
1 & 0 & c_1 \\
\alpha  & \beta & c_1 \alpha  + \beta
\end{pmatrix}}
$  &

${
 \begin{pmatrix}
0 & \zeta^2 & 0 \\
-\zeta & 0 & 0   \\
0 & 0  &  \dfrac{-\zeta^2}{\sqrt{c_1}}
\end{pmatrix}}$
&
${
\begin{pmatrix}
0 & 1 & 1 \\
1 & 0 & \dfrac{\zeta^2}{c_1}  \\
-\sqrt{c_1} \beta  & - \zeta^2 \sqrt{c_1} \alpha  & \dfrac{-\zeta^2 (c_1 \alpha + \beta)}{\sqrt{c_1}}
\end{pmatrix}} $
\\
\hline

\end{tabular}}

\footnotesize{TABLE $5^\prime$.}
\end{center}

\newpage

\noindent
{\bf Case 1.5} $M_B$ has eight non-zero entries.

\noindent
There are nine possibilities to place a zero in the structure matrix. Therefore, there are two parametric families of evolution algebras listed in the following table.\\
\bigskip

\begin{center}

\scalebox{0.65}{

\begin{tabular}{|c||c||c||c||c||c|}

\hline
&&&&&\\
&(1,2)&(1,3)&(2,3)&(1,2,3)&(1,3,2)\\
&&&&&\\[-0.2cm]
\hline
\hline
&&&&&\\[-0.2cm]

$\color{ao(english)}{
\begin{pmatrix}
0 & 1 & 1 \\
1 & \alpha & c_1+\alpha \\
\beta  & \gamma & c_1 \beta  + \gamma
\end{pmatrix}}
$  &

$\color{ao(english)}{
 \begin{pmatrix}
\alpha & 1 & c_1+\alpha \\
1 & 0 & 1 \\
\gamma  & \beta & c_1 \beta + \gamma
\end{pmatrix}}$
&
$\color{ao(english)}{
\begin{pmatrix}
c_1 \beta  + \gamma & \gamma & \beta \\
c_1+\alpha & \alpha & 1 \\
1  & 1 & 0
\end{pmatrix} }$
 &

 $\begin{pmatrix}
0 & 1 & 1 \\
\beta & c_1 \beta + \gamma & \gamma \\
1  & c_1+\alpha & \alpha
\end{pmatrix}$
&
$
\begin{pmatrix}
c_1 \beta  + \gamma  & \beta & \gamma \\
1 & 0 & 1 \\
c_1+\alpha  & 1 & \alpha
\end{pmatrix}$
&
$
\begin{pmatrix}
\alpha & c_1+\alpha & 1 \\
\gamma & c_1 \beta  + \gamma & \beta \\
1  & 1 & 0
\end{pmatrix}
$ \\
&&&&&\\[-0.2cm]
\hline
\hline
&&&&&\\[-0.2cm]

$\color{ao(english)}{
\begin{pmatrix}
1 & 0 & 1 \\
\alpha & 1 & \alpha+c_2 \\
\beta  & \gamma & \beta+ c_2 \gamma
\end{pmatrix}}
$  &

$\color{ao(english)}{
 \begin{pmatrix}
1 & \alpha & \alpha+c_2 \\
0 & 1 & 1 \\
\gamma  & \beta & \beta+ c_2 \gamma
\end{pmatrix}}$
&
$\color{ao(english)}{
\begin{pmatrix}
\beta+ c_2 \gamma  & \gamma & \beta \\
\alpha+c_2 & 1 & \alpha \\
1  & 0 & 1
\end{pmatrix} }$
 &

 $\color{ao(english)}{
 \begin{pmatrix}
1 & 1 & 0 \\
\beta & \beta+ c_2 \gamma & \gamma \\
\alpha  & \alpha+c_2 & 1
\end{pmatrix}}$
&
$\color{ao(english)}{
\begin{pmatrix}
\beta+c_2 \gamma  & \beta & \gamma \\
1 & 1 & 0 \\
\alpha+c_2 & \alpha & 1
\end{pmatrix}}$
&
$\color{ao(english)}{
\begin{pmatrix}
1 & \alpha+c_2 & \alpha \\
\gamma & \beta+ c_2 \gamma   & \beta \\
0  & 1 & 1
\end{pmatrix}}
$ \\
&&&&&\\
\hline
\end{tabular}}

\footnotesize{TABLE 6. ${\rm dim}(A^2)=2$; $\alpha, \, \beta,\, \gamma,\, c_1 +\alpha, \,  \alpha + c_2, \,  c_1 \beta + \gamma, \, \beta + c_2\gamma, \,c_1,\, c_2 \neq 0$; } \\
\footnotesize{ ${\rm dim}({\rm ann}(A))=0$; $A$ has Property (2LI); eight non-zero entries.}\label{tabla5}

\end{center}
\bigskip

\begin{center}

\scalebox{0.7}{

\begin{tabular}{|c||c||c|}

\hline
&&\\
$M_{B}$& $P_{B'B}$ & $M_{B'}$\\
&&\\[-0.2cm]
\hline
\hline
&&\\[-0.2cm]
${
\begin{pmatrix}
0 & 1 & 1 \\
1 & \alpha & c_1 + \alpha \\
\beta & \gamma & c_1 \beta + \gamma
\end{pmatrix}}
$  &

${
 \begin{pmatrix}
 1 & 0 & 0 \\
0 & 1  & 0  \\
0 & 0  & -1
\end{pmatrix}}$
&
${
\begin{pmatrix}
0 & 1 & 1 \\
1 & \alpha & c_1 + \alpha \\
-\beta & -\gamma & -(c_1 \beta  + \gamma)
\end{pmatrix}}$
\\
&&\\[-0.2cm]
\hline
\hline
&&\\[-0.2cm]
${
\begin{pmatrix}
0 & 1 & 1 \\
1 & \alpha & c_1 + \alpha \\
\beta & \gamma & c_1 \beta + \gamma
\end{pmatrix}}
$  &

${
 \begin{pmatrix}
 \zeta^2 & 0 & 0 \\
0 & -\zeta & 0  \\
0 & 0  & -\zeta
\end{pmatrix}}$
&
${
\begin{pmatrix}
0 & 1 & 1 \\
1 & -\zeta \alpha & -\zeta(c_1 + \alpha) \\
\beta & -\zeta\gamma & -\zeta(c_1 \beta  + \gamma)
\end{pmatrix}}$
\\
\hline

\end{tabular}}

\end{center}

\newpage

\begin{center}

\scalebox{0.7}{

\begin{tabular}{|c||c||c|}

\hline
&&\\
$M_{B}$& $P_{B'B}$ & $M_{B'}$\\
&&\\[-0.2cm]
\hline
\hline
&&\\[-0.2cm]
${
\begin{pmatrix}
0 & 1 & 1 \\
1 & \alpha & c_1 + \alpha \\
\beta & \gamma  & c_1 \beta  + \gamma
\end{pmatrix}}
$  &
${
 \begin{pmatrix}
 \dfrac{1}{\sqrt[3]{\beta ^2}} & 0 & 0 \\
0 & 0  &  \dfrac{-1}{\sqrt[3]{\beta }} \\
0 & \dfrac{1}{\sqrt[3]{\beta}}  & 0
\end{pmatrix}}$
&
${
\begin{pmatrix}
0 & 1 & 1 \\
1 & \dfrac{c_1 \beta  + \gamma}{\sqrt[3]{\beta}} & \dfrac{\gamma}{\sqrt[3]{\beta}}  \\
\dfrac{-1}{\beta} & \dfrac{-(c_1 + \alpha)}{\sqrt[3]{\beta}}  & \dfrac{-\alpha}{\sqrt[3]{\beta}}
\end{pmatrix}} $
\\
&&\\[-0.2cm]
\hline
\hline
&&\\[-0.2cm]

${
\begin{pmatrix}
0 & 1 & 1 \\
1 & \alpha & c_1 + \alpha \\
\beta & \gamma  & c_1 \beta  + \gamma
\end{pmatrix}}
$  &
${
 \begin{pmatrix}
 \dfrac{\zeta^2}{\sqrt[3]{\beta ^2}} & 0 & 0 \\
0 & 0  &  \dfrac{-\zeta}{\sqrt[3]{\beta }} \\
0 & \dfrac{-\zeta}{\sqrt[3]{\beta}}  & 0
\end{pmatrix}}$
&
${
\begin{pmatrix}
0 & 1 & 1 \\
1 & \dfrac{-\zeta(c_1 \beta  + \gamma)}{\sqrt[3]{\beta}} & \dfrac{-\zeta \gamma}{\sqrt[3]{\beta}}  \\
\dfrac{1}{\beta} & \dfrac{-\zeta(c_1 + \alpha)}{\sqrt[3]{\beta}}  & \dfrac{-\zeta\alpha}{\sqrt[3]{\beta}}
\end{pmatrix}} $
\\
\hline

\end{tabular}}

\footnotesize{TABLE $6^\prime$.}
\end{center}

\noindent
{\bf Case 1.6} $M_B$ has nine non-zero entries.

\noindent
The parameters $\alpha$, $\beta$ and $\gamma$ have to verify that the three of them cannot be equal in order for $M_{B}$ to have rank two. This produces only one parametric family of evolution algebras.\\
\bigskip

\begin{center}

\scalebox{0.65}{

\begin{tabular}{|c||c||c||c||c||c|}

\hline
&&&&&\\
&(1,2)&(1,3)&(2,3)&(1,2,3)&(1,3,2)\\
&&&&&\\[-0.2cm]
\hline
\hline
&&&&&\\[-0.2cm]

$\color{ao(english)}{
\begin{pmatrix}
1 & \alpha & c_1+c_2 \alpha  \\
1 & \beta & c_1+c_2 \beta  \\
1  & \gamma & c_1+ c_2 \gamma
\end{pmatrix}}
$  &

${
 \begin{pmatrix}
\beta & 1 & c_1+ c_2 \beta  \\
\alpha & 1 & c_1+ c_2 \alpha  \\
\gamma  & 1 &  c_1+ c_2 \gamma
\end{pmatrix}}$
&
$\begin{pmatrix}
 c_1+c_2 \gamma  & \gamma & 1 \\
 c_1+ c_2 \beta  & \beta & 1 \\
c_1+c_2 \alpha  & \alpha & 1
\end{pmatrix} $
 &

 $\begin{pmatrix}
1 & c_1+ c_2 \alpha  & \alpha \\
1 & c_1+ c_2 \gamma  & \gamma \\
1  &  c_1+ c_2 \beta  & \beta
\end{pmatrix}$
&
$
\begin{pmatrix}
c_1+ c_2 \gamma  & 1 & \gamma \\
c_1+ c_2 \alpha  & 1 & \alpha \\
 c_1+ c_2 \beta   & 1 & \beta
\end{pmatrix}$
&
$
\begin{pmatrix}
\beta &  c_1+c_2 \beta & 1 \\
\gamma & c_1 + c_2 \gamma  & 1 \\
\alpha  & c_1+c_2 \alpha & 1
\end{pmatrix}
$ \\
&&&&&\\
\hline
\end{tabular}}

{\footnotesize {TABLE 7.  ${\rm dim}(A^2)=2$; $\alpha, \beta,  \gamma,  c_1 + c_2\alpha, c_1+ c_2\beta, c_1 +c_2 \gamma   \neq 0$;
${\rm dim}({\rm ann}(A))=0$;  }}

{\footnotesize {$A$ has Property (2LI); nine non-zero entries. }}\label{tabla6}
\end{center}
\medskip

\begin{center}

\scalebox{0.7}{

\begin{tabular}{|c||c||c|}

\hline
&&\\
$M_{B}$& $P_{B'B}$ & $M_{B'}$\\
&&\\[-0.2cm]
\hline
\hline
&&\\[-0.2cm]
${
\begin{pmatrix}
1 & \alpha & c_1+c_2 \alpha  \\
1 & \beta &  c_1 + c_2 \beta  \\
1 & \gamma  & c_1 + c_2 \gamma
\end{pmatrix}}
$  &
${
\begin{pmatrix}
1 & 0 & 0 \\
0 & 0 & 1 \\
0 & 1 & 0

\end{pmatrix}}
$  &

${
\begin{pmatrix}
1 &  c_1+ c_2 \alpha & \alpha \\
1 & c_1+ c_2 \gamma & \gamma \\
1 & c_1 + c_2 \beta & \beta
\end{pmatrix}}$
\\
&&\\[-0.2cm]
\hline
\hline
&&\\[-0.2cm]
${
\begin{pmatrix}
1 & \alpha & c_1+c_2 \alpha \\
1 & \beta &  c_1 + c_2 \beta  \\
1 & \gamma  & c_1 + c_2\gamma
\end{pmatrix}}
$  &

${
 \begin{pmatrix}
 0 & 0 & \dfrac{\alpha}{\beta^2} \\
\dfrac{1}{\beta} & 0  & 0  \\
0 &  \dfrac{\gamma}{\beta ^2} & 0
\end{pmatrix}}$
&
${
\begin{pmatrix}
1  & \dfrac{\gamma^2( c_1 + c_2 \beta )}{\beta ^3} & \dfrac{\alpha^2}{\beta ^3} \\
1  & \dfrac{\gamma( c_1 + c_2 \gamma )}{\beta ^2} & \dfrac{\alpha^2}{\beta ^2 \gamma} \\
1  & \dfrac{\gamma^2( c_1 + c_2 \alpha)}{\alpha \beta ^2} & \dfrac{\alpha}{\beta ^2} \\
\end{pmatrix}}$
\\
\hline

\end{tabular}}

{\footnotesize {TABLE $7^\prime$.}}
\end{center}
\medskip

\noindent
{\bf Case 2} $c_1=0$ and $c_2\neq0$.

\medskip

\noindent
Note that if we consider the new natural basis $\{e'_1, e'_2,e'_3\}$ with $e'_1=e_1$, $e'_2=\sqrt{c_2}e_2-e_3$ and $e'_3=\sqrt{c_2}e_2+e_3$ we obtain that $(e'_{2})^2=(e'_{3})^2$. By abusing of notation, we may assume that the natural basis concerned is $\{e_1,e_2,e_3\}$ with $e_2^2=e_{3}^{2}$.

\noindent
Note that in this case the dimension of the annihilator of the evolution algebra is zero. We will see that the possible change of basis matrices are precisely two elements in $S_3\rtimes (\K^\times)^3$ (we consider only those for which $e_3^2=e_2^2$) and one more not in $S_3\rtimes (\K^\times)^3$ that we will specify.

In this case, the equations \eqref{cuartas}, \eqref{quintas}, \eqref{segundas} and \eqref{terceras} are as follows:

\begin{align*}
p_{11}p_{12}=0;\quad p_{21}p_{22}=- p_{31}p_{32};\\
p_{11}p_{13}=0; \quad p_{21}p_{23}=- p_{31}p_{33};\\
p_{12}p_{13}=0; \quad p_{22}p_{23}=- p_{32}p_{33}.\\
\end{align*}

\noindent
We may suppose that $p_{11}=p_{12}=0$. \\

\noindent
Assume that $p_{31}p_{32}p_{33}\neq 0$. This implies that $p_{21}p_{22}p_{23}\neq 0$. As $p_{21}=\frac{- p_{31}p_{32}}{p_{22}}$, $p_{23}=\frac{p_{33}p_{22}}{p_{32}}$. Then $p_{22}=\pm \sqrt{-1}p_{32}$, $p_{23}=\pm \sqrt{-1}p_{33}$ and $p_{21}=\pm \sqrt{-1}p_{31}$. But, in these conditions $\vert P_{B'B} \vert =0$. Therefore there exists at least one $i \in\{1,2,3\}$ such that $p_{3i}=0$. \\

\noindent
If $p_{31}=0$, then $p_{21}p_{22}=0$ and $p_{21}p_{23}=0$. Since $p_{21}\neq 0$, necessarily $p_{22}=p_{23}=0$, implying $p_{32}p_{33}=0$. Consequently, $P_{B'B} \in S_3\rtimes (\K^\times)^3$.\\

\noindent
If $p_{31}\neq0$ and $p_{32}=0$, then $p_{21}p_{22}=0 $ and $p_{22}p_{23}=0$. This implies  $p_{21}=p_{23}=p_{33}=0$ and again $P_{B'B}\in S_3\rtimes (\K^\times)^3$.\\

\noindent
If $p_{31}p_{32}\neq0$ and $p_{33}=0$, then $p_{22}p_{23}=0$ and $p_{21}p_{23}=0$. Necessarily $p_{23}=0$. On the other hand, as $p_{31}p_{32} \neq 0$, $p_{21}p_{22}\neq 0$. So, $p_{22}=\frac{-p_{31}p_{32}}{p_{21}}$ and

\begin{eqnarray} \label{matrices1}
P_{B'B}=
\begin{pmatrix}
0 & 0 & p_{13}\\
p_{21} & \dfrac{-p_{31}p_{32}}{p_{21}} &  0\\
p_{31} & p_{32} & 0
\end{pmatrix},
\end{eqnarray}

\noindent
with $p_{13}p_{22}p_{31}p_{32} \neq 0$ and $ p_{31}^2+p_{21}^2 \neq 0$ in order to have $\vert P_{B'B} \vert \neq 0$.

If we suppose that $p_{11}=p_{13}=0$, reasoning in the same way as before, we obtain that the matrices $P_{B'B}$ are in $S_3\rtimes (\K^\times)^3$ or they are as follows:

\begin{eqnarray}\label{matrices2}
P_{B'B}=
\begin{pmatrix}
0 & p_{12} & 0 \\
p_{21} & 0 & \dfrac{- p_{33}p_{31}}{p_{21}}\\
p_{31} & 0 & p_{33}
\end{pmatrix},
\end{eqnarray}

\noindent
with $p_{12}p_{21}p_{31}p_{33} \neq 0$ and $p_{31}^2+p_{21}^2 \neq 0$.

Finally, if $p_{12}=p_{13}=0$, we obtain that the different matrices $P_{B'B}$ that appear are in $S_3\rtimes (\K^\times)^3$ or are of the form:

\begin{eqnarray}\label{matrices3}
P_{B'B}=
\begin{pmatrix}
p_{11} & 0 & 0 \\
0 & p_{22} & \dfrac{-p_{32}p_{33}}{p_{22}}\\
0 & p_{32} & p_{33}
\end{pmatrix},
\end{eqnarray}

\noindent
with $p_{11}p_{22}p_{32}p_{33} \neq 0$ and $ p_{32}^2+p_{22}^2 \neq 0$.
\smallskip

\noindent
If
$
P_{B'B}=
\begin{pmatrix}
0 & 0 & p_{13} \\
p_{21} & \dfrac{-p_{31}p_{32}}{p_{21}} & 0 \\
p_{31} & p_{32} & 0
\end{pmatrix}
$ then

$$
\tiny
M_{B'}=
\begin{pmatrix}
\omega_{22} p_{21} + \omega_{32}p_{31}  & \dfrac{p_{32}^2 (\omega_{22} p_{21} + \omega_{32}p_{31})}{p_{21}^2} & \dfrac{p_{13}^2(\omega_{21} p_{21} + \omega_{31}p_{31})}{p_{21}^2+p_{31}^2}\\ \\
\dfrac{p_{21} (\omega_{32} p_{21} - \omega_{22}p_{31})}{p_{32}} & \dfrac{p_{32} (\omega_{32} p_{21} - \omega_{22}p_{31})}{p_{21}} & \dfrac{p_{13}^2p_{21}(\omega_{31} p_{21} - \omega_{21}p_{31})}{p_{32}(p_{21}^2+p_{31}^2)} \\ \\
\dfrac{\omega_{12} ( p_{21}^2 + p_{31}^2)}{p_{13}}  & \dfrac{\omega_{12}p_{32}^2 ( p_{21}^2 + p_{31}^2)}{p_{13}p_{21}^2} & \omega_{11} p_{13} \\
\end{pmatrix}.$$

\medskip

\noindent
If
$
P_{B'B}=
\begin{pmatrix}
0 & p_{12} & 0 \\
p_{21} & 0 & \dfrac{-p_{31}p_{33}}{p_{21}}  \\
p_{31} & 0 & p_{33}
\end{pmatrix}
$ then

$$
\tiny
M_{B'}=
\begin{pmatrix}
\omega_{22} p_{21} + \omega_{32}p_{31}  & \dfrac{p_{12}^2 (\omega_{21} p_{21} + \omega_{31}p_{31})}{p_{21}^2+p_{31}^2} & \dfrac{p_{33}^2(\omega_{22} p_{21} + \omega_{32}p_{31})}{p_{21}^2}\\ \\
\dfrac{\omega_{12} ( p_{21}^2 + p_{31}^2)}{p_{12}} & \omega_{11}p_{12} & \dfrac{\omega_{12}p_{33}^2(p_{21}^2 + p_{31}^2)}{p_{12}p_{21}^2} \\ \\
\dfrac{p_{21} ( \omega_{32}p_{21} - \omega_{22}p_{31})}{p_{33}}  & \dfrac{p_{12}^2p_{21} ( \omega_{31}p_{21} -\omega_{21} p_{31})}{p_{33}(p_{21}^2 + p_{31}^2)} & \dfrac{p_{33} ( \omega_{32}p_{21} - \omega_{22}p_{31})}{p_{21}}  \\
\end{pmatrix}.$$

\noindent
If
$
P_{B'B}=
\begin{pmatrix}
p_{11} & 0 & 0 \\
0 & p_{22} & \dfrac{-p_{32}p_{33}}{p_{22}}  \\
0 & p_{32} & p_{33}
\end{pmatrix}
$ then

$$
\tiny
M_{B'}=
\begin{pmatrix}
\omega_{11} p_{11}   & \dfrac{\omega_{12}(p_{22}^2 + p_{32}^2)}{p_{11}} & \dfrac{\omega_{12}p_{33}^2(p_{22}^2 + p_{32}^2)}{p_{11}p_{22}^2}\\ \\
\dfrac{p_{11}^2 (\omega_{21} p_{22} + \omega_{31}p_{32})}{p_{22}^2+p_{32}^2} & \omega_{22} p_{22} + \omega_{32}p_{32} & \dfrac{p_{33}^2(\omega_{22} p_{22} + \omega_{32}p_{32})}{p_{22}^2} \\ \\
\dfrac{p_{11}^2p_{22}(\omega_{31} p_{22}- \omega_{21}p_{32})}{p_{33}(p_{22}^2+p_{32}^2)}  & \dfrac{p_{22}( \omega_{32}p_{22} - \omega_{22}p_{32})}{p_{33}} &  \dfrac{p_{33}( \omega_{32}p_{22} - \omega_{22}p_{32})}{p_{22}}\\
\end{pmatrix}.$$

\noindent
Taking into account that we were assuming  $e_{3}^{2}=e_{2}^2$, then the possible change of basis matrices are the following:

\begin{equation}\label{conjuntoc1zero}
\footnotesize{
\left\{
\begin{pmatrix}
p_{11} & 0 &  0 \\
0 & p_{22} & 0  \\
0 & 0 & p_{33}
\end{pmatrix},
\begin{pmatrix}
p_{11} & 0 &  0 \\
0 & 0 & p_{23}  \\
0 & p_{32} & 0
\end{pmatrix},
\begin{pmatrix}
p_{11} & 0 & 0 \\
0 & p_{22} & \dfrac{-p_{32}p_{33}}{p_{22}}\\
0 & p_{32} & p_{33}
\end{pmatrix}\  \vert\ p_{11}, p_{22}, p_{32}, p_{33}  \in \K^\times \right\}}.
\end{equation}

In what follows we will classify in three steps: we start by taking into account the first two families of change of basis matrices of the set \eqref{conjuntoc1zero} which leave invariant the number of non-zero entries in the first and second columns. Then, we will analyze if the resulting families of evolution algebras are or not isomorphic under the action of one matrix of the third family in  \eqref{conjuntoc1zero}, i.e., we will see if some families of evolution algebras are included into other families when applying the change of basis matrices of the third type. Finally, we will analyze, for each of the resulting parametric families, if their algebras are mutually isomorphic.

We list the different matrices into tables taking in account the number of zeros in the first and second columns. Each of these tables will receive the name of ``Figure $m$". According to \eqref{producto2} we will write as many 1 as possible and the others non-zero entries will be arbitrary parameters $\alpha$, $\beta$, $\gamma$ and $\lambda$ under the restriction $e_2^2=e_3^2$. We start by the first one and applying the action of the elements: \\

\begin{center}
$\begin{pmatrix}
1 & 0 &  0 \\
0 & 0 & 1  \\
0 & 1 & 0
\end{pmatrix}
$,    $\;  \; \; \, \, Q=\tiny \begin{pmatrix}
p_{11} & 0 & 0 \\
0 & p_{22} & \dfrac{- p_{32}p_{33}}{p_{22}}\\
0 & p_{32} & p_{33}
\end{pmatrix}
$.
\end{center}

\noindent
with  $p_{11}$, $p_{22}$, $p_{32}$, $p_{33}$ $\in \K^\times$ and $ p_{32}^2+p_{22}^2 \neq 0$.
\medskip

\noindent
{\bf Case 2.1} $M_B$ has two non-zero entries in the first and second columns.
\medskip

 There are $\binom{6}{4} = 15$ possible places where to put four zeros. Since some of the resulting matrices have rank 1,  they must be removed from the 15 cases. This happens whenever the first or the second columns is zero (2 cases) and the remaining zeros can be settled in three different places. This produces 6 cases. We also eliminate the cases in which two different rows are zero (3 options). Therefore we have $15-6-3 = 6$ different matrices classified in 3 types. Their structure matrices appear in the first column of the table that follows.

\begin{center}
\scalebox{0.7}{
\begin{tabular}{|c||c||c||c|}

\hline
&& &\\
Type & &(2,3)& Q
\\
&&&\\[-0.2cm]
\hline
\hline
&&&\\[-0.2cm]
1
&
$\color{ao(english)}{
\begin{pmatrix}
1 & 0 & 0 \\
0 & 1 & 1\\
0  & 0 & 0
\end{pmatrix}}
$  &

$\color{ao(english)}{
 \begin{pmatrix}
1 & 0 & 0 \\
0 & 0 & 0 \\
0  & 1 & 1
\end{pmatrix}}$
&
$\tiny
\begin{pmatrix}
p_{11} & 0 & 0 \\
0 & p_{22} & \dfrac{p_{33}^2}{p_{22}} \\ \\
0  & \dfrac{-p_{22}p_{32}}{p_{33}} & \dfrac{-p_{32}p_{33}}{p_{22}}
\end{pmatrix} $ \\
&&&\\[-0.2cm]
\hline
\hline
&&&\\[-0.2cm]
2
&
$\color{ao(english)}{
\begin{pmatrix}
0 & 1 & 1 \\
1 & 0 & 0 \\
0  & 0 & 0
\end{pmatrix}}
$  &
$\color{ao(english)}{
 \begin{pmatrix}
0 & 1 & 1 \\
0 & 0 & 0 \\
1  & 0 & 0
\end{pmatrix}}$
&
$ \tiny{
\begin{pmatrix}
0 & \dfrac{p_{22}^2+p_{32}^2}{p_{11}} & \dfrac{(p_{22}^2+p_{32}^2)p_{33}^2}{p_{11}p_{22}^2} \\
\dfrac{p_{11}^2 p_{22}}{p_{22}^2+p_{32}^2} & 0 & 0 \\ \\
\dfrac{-p_{11}^2p_{22}p_{32}}{p_{33}(p_{22}^2+p_{32}^2)}  & 0 & 0
\end{pmatrix}}
$ \\
&&&\\[-0.2cm]
\hline
\hline
&&&\\[-0.2cm]
3
&
$\color{ao(english)}{
\begin{pmatrix}
0 & 0 & 0 \\
1 & 0 & 0 \\
0  & 1 & 1
\end{pmatrix}}$
&
 $\color{ao(english)}{
 \begin{pmatrix}
0 & 0 & 0 \\
0 & 1 & 1 \\
1  & 0 & 0
\end{pmatrix}}$
&
 $\tiny
 \begin{pmatrix}
0 & 0 & 0 \\
\dfrac{p_{11}^2p_{22}}{p_{22}^2+p_{32}^2} & p_{32} & \dfrac{p_{32}p_{33}^2}{p_{22}^2} \\ \\
\dfrac{-p_{11}^2p_{22}p_{32}}{p_{33}(p_{22}^2+p_{32}^2)}  & \dfrac{p_{22}^2}{p_{33}} & p_{33}
\end{pmatrix}$ \\
&&&\\
\hline
\end{tabular}}

{\footnotesize {FIGURE 1. ${\rm dim}(A^2)=2$; ${\rm dim}({\rm ann}(A))=0$;}}\label{tabla7} \\
\footnotesize {$A$ has not Property (2LI); two non-zero entries in the first and second columns.}

\end{center}
\newpage

 \noindent
 {\bf Case 2.2} $M_B$ has three non-zero entries in the first and second columns.\\
\medskip

There exist $\binom{6}{3} = 20$ possible places where to write three zeros. We remove the matrices which have rank 1. This happens 2 times: when the first or the second column is zero. Therefore we have $20-2 = 18$ cases. There are 10 types listed in the two tables below.

\begin{center}

\scalebox{0.7}{
\begin{tabular}{|c||c||c||c|}

\hline
&&&\\
Type & &(2,3)& Q \\
&&&\\[-0.2cm]
\hline
\hline
&& &\\[-0.2cm]
4
&
$\color{ao(english)}{
\begin{pmatrix}
1 & 0 & 0 \\
0 & 1 & 1 \\
0  & \alpha & \alpha
\end{pmatrix}}
$  &

$
\begin{pmatrix}
1 & 0 & 0 \\
0 & \alpha & \alpha \\
0  & 1 & 1
\end{pmatrix}
$
&
$\tiny
\begin{pmatrix}
p_{11} & 0 & 0 \\
0 & p_{22}+\alpha p_{32} & \dfrac{p_{33}^2(p_{22}+\alpha p_{32})}{p_{22}^2} \\
0  & \dfrac{p_{22}(\alpha p_{22}-p_{32})}{p_{33}} &\dfrac{p_{33}(\alpha p_{22}-p_{32})}{p_{22}}
\end{pmatrix} $
\\
&&&\\[-0.2cm]
\hline
\hline
&&&\\[-0.2cm]
5
&
$\color{ao(english)}{
\begin{pmatrix}
1 & 0 & 0 \\
\alpha & 1 & 1 \\
0  & 0 & 0
\end{pmatrix}}
$  &

$\color{ao(english)}{
\begin{pmatrix}
1 & 0 & 0 \\
0 & 0 & 0 \\
\alpha  & 1 & 1
\end{pmatrix}}$
&
$\tiny
\begin{pmatrix}
p_{11} & 0 & 0 \\
\dfrac{\alpha p_{11}^2 p_{22}}{p_{22}^2+p_{32}^2} & p_{22} & \dfrac{p_{33}^2}{p_{22}} \\ \\
\dfrac{-\alpha p_{11}^2 p_{22}p_{32}}{p_{33}(p_{22}^2+p_{32}^2)}  & \dfrac{-p_{22} p_{32}}{p_{33}} & \dfrac{- p_{32}p_{33}}{p_{22}}
\end{pmatrix}$
\\
&&&\\[-0.2cm]
\hline
\hline
&&&\\[-0.2cm]
6
&
$\color{ao(english)}{
\begin{pmatrix}
0 & 1 & 1 \\
1 & 0 & 0 \\
0  & \alpha & \alpha
\end{pmatrix}}$
&
$\color{ao(english)}{
\begin{pmatrix}
0 & 1 & 1 \\
0 & \alpha & \alpha \\
1  & 0 & 0
\end{pmatrix}}$
&
$\tiny
\begin{pmatrix}
0 & \dfrac{p_{22}^2+p_{32}^2}{p_{11}} & \dfrac{ p_{33}^2(p_{22}^2+p_{32}^2)}{p_{11}p_{22}^2} \\ \\
\dfrac{p_{11}^2p_{22}}{p_{22}^2+p_{32}^2} & \alpha p_{32} & \dfrac{{\alpha}p_{32}p_{33}^2}{p_{22}^2} \\ \\
\dfrac{-p_{11}^2p_{22}p_{32}}{p_{33}(p_{22}^2+p_{32}^2)}  & \dfrac{\alpha p_{22}^2}{p_{33}} & \alpha p_{33}
\end{pmatrix}
$ \\
&&&\\[-0.2cm]
\hline
\hline
&&&\\[-0.2cm]
7
&
$\color{ao(english)}{
\begin{pmatrix}
0 & 1 & 1  \\
\alpha & 1 & 1 \\
0  & 0 & 0
\end{pmatrix}}
$  &

$\color{ao(english)}{
\begin{pmatrix}
0 & 1 & 1 \\
0 & 0 & 0 \\
\alpha  & 1 & 1
\end{pmatrix}}$
&
$\tiny
\begin{pmatrix}
0 & \dfrac{p_{22}^2+p_{32}^2}{p_{11}} & \dfrac{p_{33}^2(p_{22}^2+p_{32}^2)}{p_{11}p_{22}^2} \\ \\
\dfrac{\alpha p_{11}^2p_{22}}{p_{22}^2+p_{32}^2} & p_{22} & \dfrac{p_{33}^2}{p_{22}} \\ \\
\dfrac{-\alpha p_{11}^2p_{22}p_{32}}{p_{33}(p_{22}^2+p_{32}^2)}  & \dfrac{-p_{22}p_{32}}{p_{33}} & \dfrac{-p_{32}p_{33}}{p_{22}}
\end{pmatrix}
$
\\
&&&\\[-0.2cm]
\hline
\hline
&&&\\[-0.2cm]

8

&
$\color{ao(english)}{
\begin{pmatrix}
1 & 1 & 1 \\
0 & 0 & 0 \\
\alpha  & 0 & 0
\end{pmatrix}}
$  &

$\color{ao(english)}{
\begin{pmatrix}
1 & 1 & 1 \\
\alpha & 0 & 0 \\
0  & 0 & 0
\end{pmatrix}}$
&
$\tiny
\begin{pmatrix}
p_{11} & \dfrac{p_{22}^2+p_{32}^2}{p_{11}} & \dfrac{p_{33}^2(p_{22}^2+p_{32}^2)}{p_{11}p_{22}^2} \\ \\
\dfrac{\alpha p_{11}^2p_{32}}{p_{22}^2+p_{32}^2} & 0 & 0 \\ \\
\dfrac{\alpha p_{11}^2p_{22}^2}{p_{33}(p_{22}^2+p_{32}^2)}  & 0 & 0
\end{pmatrix}
$
\\
&&&\\[-0.2cm]
\hline
\hline
&&&\\[-0.2cm]
9
&
$\color{ao(english)}{
\begin{pmatrix}
1 & 1 & 1 \\
0 & 0 & 0 \\
0  & \alpha & \alpha
\end{pmatrix}}
$  &

$\color{ao(english)}{
\begin{pmatrix}
1 & 1 & 1 \\
0 & \alpha & \alpha \\
0  & 0 & 0
\end{pmatrix}}$
&
$\tiny
\begin{pmatrix}
p_{11} & \dfrac{p_{22}^2+p_{32}^2}{p_{11}} & \dfrac{p_{33}^2(p_{22}^2+p_{32}^2)}{p_{11}p_{22}^2} \\ \\
0 & \alpha p_{32}& \dfrac{\alpha p_{33}^2p_{32}}{p_{22}^2}  \\ \\
0 & \dfrac{\alpha p_{22}^2}{p_{33}}  & \alpha p_{33}
\end{pmatrix}
$
 \\
&&&\\
\hline
\end{tabular}}

\end{center}
\newpage

\begin{center}

\scalebox{0.7}{
\begin{tabular}{|c||c||c||c|}
\hline
&&&\\
Type & &(2,3)& Q \\
&&&\\[-0.2cm]
\hline
\hline
&&&\\[-0.2cm]
10
&
$\color{ao(english)}{
\begin{pmatrix}
1 & 0  & 0 \\
\alpha & 0 & 0 \\
0  & 1 & 1
\end{pmatrix}}
$  &
$\color{ao(english)}{
\begin{pmatrix}
1 & 0  & 0 \\
0 & 1 & 1 \\
\alpha  & 0 & 0
\end{pmatrix}}
$  &
$\tiny
\begin{pmatrix}
p_{11} & 0 & 0 \\
\dfrac{\alpha p_{11}^2p_{22}}{p_{22}^2+p_{32}^2} & p_{32} &  \dfrac{p_{32}p_{33}^2}{p_{22}^2} \\ \\
\dfrac{-\alpha p_{11}^2p_{22}p_{32}}{p_{33}(p_{22}^2+p_{32}^2)}& \dfrac{p_{22}^2}{p_{33}} & p_{33}  \\
\end{pmatrix}
$
\\
&&&\\[-0.2cm]
\hline
\hline
&&&\\[-0.2cm]
11
&
$\color{ao(english)}{
\begin{pmatrix}
0 & 1 & 1 \\
1 & 0 & 0 \\
\alpha  & 0 & 0
\end{pmatrix}}
$  &
$ \begin{pmatrix}
0 & 1 & 1 \\
\alpha & 0 & 0 \\
1  & 0 & 0
\end{pmatrix}$
&
$\tiny
\begin{pmatrix}
0 & \dfrac{p_{22}^2+p_{32}^2}{p_{11}} & \dfrac{p_{33}^2(p_{22}^2+p_{32}^2)}{p_{11}p_{22}^2} \\ \\
\dfrac{p_{11}^2(p_{22}+\alpha p_{32})}{p_{22}^2+p_{32}^2} & 0 & 0 \\ \\
\dfrac{p_{11}^2p_{22}(\alpha p_{22}-p_{32})}{p_{33}(p_{22}^2+p_{32}^2)}  & 0 & 0
\end{pmatrix} $
\\
&&&\\[-0.2cm]
\hline
\hline
&&&\\[-0.2cm]
12
&
$\color{ao(english)}{
\begin{pmatrix}
0 & 0 & 0 \\
1 & 1 & 1 \\
0  & \alpha & \alpha
\end{pmatrix}}
$  &
$\color{ao(english)}{
\begin{pmatrix}
0 & 0 & 0 \\
0 & \alpha & \alpha \\
1  & 1 & 1
\end{pmatrix}}
$
&
$\tiny
\begin{pmatrix}
0 & 0 & 0 \\
\dfrac{p_{11}^2p_{22}}{p_{22}^2+p_{32}^2} & p_{22}+\alpha p_{32} & \dfrac{p_{33}^2(p_{22}+\alpha p_{32})}{p_{22}^2} \\ \\
\dfrac{-p_{11}^2p_{22}p_{32}}{p_{33}(p_{22}^2+p_{32}^2)} &\dfrac{p_{22}(\alpha p_{22}-p_{32})}{p_{33}} & \dfrac{p_{33}(\alpha p_{22}-p_{32})}{p_{22}}
\end{pmatrix} $
\\
&&&\\[-0.2cm]
\hline
\hline
&&&\\[-0.2cm]
13
&
$\color{ao(english)}{
\begin{pmatrix}
0 & 0 & 0 \\
1 & 1 & 1 \\
\alpha  & 0 & 0
\end{pmatrix}}
$
&
$\color{ao(english)}{
\begin{pmatrix}
0 & 0 & 0 \\
\alpha & 0 & 0 \\
1  & 1 & 1
\end{pmatrix}}$
&
$\tiny
\begin{pmatrix}
0 & 0 & 0 \\
\dfrac{p_{11}^2(p_{22}+\alpha p_{32})}{p_{22}^2+p_{32}^2} & p_{22} & \dfrac{p_{33}^2}{p_{22}}  \\ \\
\dfrac{p_{11}^2p_{22}(\alpha p_{22}-p_{32})}{p_{33}(p_{22}^2+p_{32}^2)}  &\dfrac{-p_{22}p_{32}}{p_{33}} & \dfrac{-p_{32}p_{33}}{p_{22}}
\end{pmatrix} $
 \\
&&&\\
\hline
\end{tabular}}

{\footnotesize {FIGURE 2. ${\rm dim}(A^2)=2$; $\alpha \neq 0$; ${\rm dim}({\rm ann}(A))=0$;}}\label{tabla9}\\
\footnotesize{$A$ has not Property (2LI); three non-zero entries in the first and second columns.}
\end{center}

\noindent
{\bf Case 2.3} $M_B$ has four non-zero entries in the first and second columns.
\medskip

\noindent
There exists $\binom{6}{3} = 15$ possible places where to write four zeros. The non-zero parameters $\alpha$, $\beta$ satisfy that $\alpha \neq \beta $ in the matrices appearing as types 14 and 20. This is because the rank of those matrices has to be two. There are nine different types. They are listed below.
\newpage

\begin{center}

\scalebox{0.7}{

\begin{tabular}{|c||c||c||c|}

\hline
&&&\\
Type & &(2,3)&Q\\
&&&\\[-0.2cm]
\hline
\hline
&&&\\[-0.2cm]
14
&
$\color{ao(english)}{
\begin{pmatrix}
\alpha & 1 & 1 \\
\beta & 1 & 1 \\
0  & 0 & 0
\end{pmatrix}}
$  &

$\color{ao(english)}{
\begin{pmatrix}
\alpha & 1 & 1 \\
0 & 0 & 0 \\
\beta  & 1 & 1
\end{pmatrix}}$
&
$\tiny
\begin{pmatrix}
\alpha p_{11} & \dfrac{p_{22}^2+p_{32}^2}{p_{11}} & \dfrac{p_{33}^2(p_{22}^2+p_{32}^2)}{p_{11}p_{22}^2} \\ \\
\dfrac{ \beta p_{11}^2p_{22}}{p_{22}^2+p_{32}^2} & p_{22} & \dfrac{ p_{33}^2}{p_{22}} \\ \\
\dfrac{-\beta p_{11}^2p_{22}p_{32}}{p_{33}(p_{22}^2+p_{32}^2)}  & \dfrac{-p_{22}p_{32}}{p_{33}} & \dfrac{-p_{32}p_{33}}{p_{22}}
\end{pmatrix}$
 \\
&&&\\[-0.2cm]
\hline
\hline
&&&\\[-0.2cm]
15

&
$\color{ao(english)}{
\begin{pmatrix}
\alpha & 1 & 1 \\
0 & 1 & 1 \\
\beta  & 0 & 0
\end{pmatrix}}
$  &

$\color{ao(english)}{
\begin{pmatrix}
\alpha & 1 & 1 \\
\beta & 0 & 0 \\
0  & 1 & 1
\end{pmatrix}}$
&
$\tiny
\begin{pmatrix}
\alpha p_{11} & \dfrac{p_{22}^2+p_{32}^2}{p_{11}} & \dfrac{p_{33}^2(p_{22}^2+p_{32}^2)}{p_{11}p_{22}^2} \\ \\
\dfrac{\beta p_{11}^2p_{32}}{p_{22}^2+p_{32}^2} &  p_{22} & \dfrac{p_{33}^2}{p_{22}} \\ \\
\dfrac{\beta p_{11}^2p_{22}^2}{p_{33}(p_{22}^2+p_{32}^2)}  & \dfrac{- p_{22}p_{32}}{p_{33}} & \dfrac{-p_{32}p_{33}}{p_{22}}
\end{pmatrix}
$
 \\
&&&\\[-0.2cm]
\hline
\hline
&&&\\[-0.2cm]
16
&
$\color{ao(english)}{
\begin{pmatrix}
1 & 0 & 0 \\
\alpha & 1 & 1 \\
\beta  & 0 & 0
\end{pmatrix}}
$  &

$\color{ao(english)}{
\begin{pmatrix}
1 & 0 & 0 \\
\beta & 0 & 0 \\
\alpha  & 1 & 1
\end{pmatrix}}$
&
$\tiny
\begin{pmatrix}
p_{11} & 0 & 0 \\
\dfrac{p_{11}^2(\alpha p_{22}+\beta p_{32})}{p_{22}^2+p_{32}^2} & p_{22} & \dfrac{p_{33}^2}{p_{22}} \\ \\
\dfrac{p_{11}^2p_{22}(\beta p_{22}-\alpha p_{32})}{p_{33}(p_{22}^2+p_{32}^2)}  & \dfrac{-p_{22}p_{32}}{p_{33}} & \dfrac{-p_{32}p_{33}}{p_{22}}
\end{pmatrix}$
 \\
&&&\\[-0.2cm]
\hline
\hline
&&&\\[-0.2cm]
17
&
$\color{ao(english)}{
\begin{pmatrix}
1 & 0 & 0 \\
0 & 1 & 1 \\
\alpha  &  \beta & \beta
\end{pmatrix}}
$  &

$\color{ao(english)}{
\begin{pmatrix}
1 & 0 & 0 \\
\alpha & \beta & \beta \\
0  & 1 & 1
\end{pmatrix}}$
&
$\tiny
\begin{pmatrix}
p_{11} & 0 & 0 \\
\dfrac{\alpha p_{11}^2p_{32}}{p_{22}^2+p_{32}^2} & p_{22}+\beta p_{32} & \dfrac{p_{33}^2(p_{22}+\beta p_{32})}{p_{22}^2} \\ \\
\dfrac{\alpha p_{11}^2p_{22}^2}{p_{33}(p_{22}^2+p_{32}^2)}  & \dfrac{p_{22}(\beta p_{22}-p_{32})}{p_{33}} & \dfrac{p_{33}(\beta p_{22}-p_{32})}{p_{22}}
\end{pmatrix}$
 \\
&&&\\[-0.2cm]
\hline
\hline
&&&\\[-0.2cm]
18
&

$\color{ao(english)}{
\begin{pmatrix}
\alpha & 1 & 1  \\
0 & 1 & 1 \\
0  & \beta & \beta
\end{pmatrix}}
$  &
${
\begin{pmatrix}
\alpha & 1 & 1 \\
0 & \beta & \beta \\
0  & 1 & 1
\end{pmatrix}}$
&
$\tiny
\begin{pmatrix}
\alpha p_{11} & \dfrac{p_{22}^2+p_{32}^2}{p_{11}} & \dfrac{p_{33}^2(p_{22}^2+p_{32}^2)}{p_{11}p_{22}^2} \\
0 & p_{22}+\beta p_{32} &  \dfrac{p_{33}^2(p_{22}+\beta p_{32})}{p_{22}^2} \\ \\
0 & \dfrac{p_{22}(\beta p_{22}- p_{32})}{p_{33}}  & \dfrac{p_{33}(\beta p_{22}-p_{32})}{p_{22}}
\end{pmatrix}
$
\\
&&&\\[-0.2cm]
\hline
\hline
&&&\\[-0.2cm]
19
&
$\color{ao(english)}{
\begin{pmatrix}
0 & 1 & 1 \\
\alpha & 0 & 0 \\
\beta  & 1 & 1
\end{pmatrix}}
$  &

$\color{ao(english)}{
\begin{pmatrix}
0 & 1 & 1 \\
\beta & 1 & 1 \\
\alpha  & 0 & 0
\end{pmatrix}}$
&
$\tiny
\begin{pmatrix}
0 & \dfrac{p_{22}^2+p_{32}^2}{p_{11}} & \dfrac{p_{33}^2(p_{22}^2+p_{32}^2)}{p_{11}p_{22}^2} \\
\dfrac{p_{11}^2(\alpha p_{22}+\beta p_{32})}{p_{22}^2+p_{32}^2} &  p_{32}  & \dfrac{ p_{32}p_{33}^2}{p_{22}^2}   \\ \\
\dfrac{p_{11}^2p_{22}(\beta p_{22}-\alpha p_{32})}{p_{33}(p_{22}^2+p_{32}^2)}  & \dfrac{p_{22}^2}{p_{33}} &  p_{33}
\end{pmatrix}
$
 \\
&&&\\
\hline
\end{tabular}}

\end{center}
\bigskip

\begin{center}

\scalebox{0.7}{

\begin{tabular}{|c||c||c||c|}

\hline
&&&\\
Type & &(2,3)&Q\\
&&&\\[-0.2cm]
\hline
\hline
&&&\\[-0.2cm]

20
&
$\color{ao(english)}{
\begin{pmatrix}
0 & 0 & 0 \\
1 & 1 & 1 \\
\alpha  & \beta & \beta
\end{pmatrix}}
$  &

${
\begin{pmatrix}
0 & 0 & 0 \\
\alpha & \beta & \beta \\
1  & 1 & 1
\end{pmatrix}}$
&
$\tiny
\begin{pmatrix}
0 & 0 & 0 \\
\dfrac{p_{11}^2(p_{22}+\alpha p_{32})}{p_{22}^2+p_{32}^2} & p_{22}+\beta p_{32} & \dfrac{p_{33}^2(p_{22}+\beta p_{32})}{p_{22}^2}\\ \\
\dfrac{p_{11}^2p_{22}(\alpha p_{22}-p_{32})}{p_{33}(p_{22}^2+p_{32}^2)} & \dfrac{p_{22}(\beta p_{22}-p_{32})}{p_{33}}  & \dfrac{p_{33}(\beta p_{22}-p_{32})}{p_{22}}   \\
\end{pmatrix}
$
 \\
&&&\\[-0.2cm]
\hline
\hline
&&&\\[-0.2cm]
21
&
$\color{ao(english)}{
\begin{pmatrix}
1 & 1 & 1 \\
\alpha & 0 & 0 \\
\beta  & 0 & 0
\end{pmatrix}}
$  &
${
\begin{pmatrix}
1 & 1 & 1 \\
\beta & 0 & 0 \\
\alpha  & 0 & 0
\end{pmatrix}}$
&
$\tiny
\begin{pmatrix}
p_{11} & \dfrac{p_{22}^2+p_{32}^2}{p_{11}} & \dfrac{p_{33}^2(p_{22}^2+p_{32}^2)}{p_{11}p_{22}^2}\\ \\
\dfrac{p_{11}^2(\alpha p_{22}+\beta p_{32})}{p_{22}^2+p_{32}^2} & 0 & 0 \\ \\
\dfrac{p_{11}^2p_{22}(\beta p_{22}-\alpha p_{32})}{p_{33}(p_{22}^2+p_{32}^2)} & 0 & 0   \\
\end{pmatrix}
$
 \\
&&&\\[-0.2cm]
\hline
\hline
&&&\\[-0.2cm]
22
&
$\color{ao(english)}{
\begin{pmatrix}
0 & 1 & 1 \\
0 & 1 & 1 \\
\alpha  & \beta & \beta
\end{pmatrix}}
$  &

$\color{ao(english)}{
\begin{pmatrix}
0 & 1 & 1 \\
\alpha & \beta & \beta \\
0  & 1 & 1
\end{pmatrix}}$
&
$\tiny
\begin{pmatrix}
0 & \dfrac{p_{22}^2+p_{32}^2}{p_{11}} & \dfrac{p_{33}^2(p_{22}^2+p_{32}^2)}{p_{11}p_{22}^2} \\ \\
\dfrac{\alpha p_{11}^2p_{32}}{p_{22}^2+p_{32}^2} & p_{22}+\beta p_{32} & \dfrac{p_{33}^2(p_{22}+\beta p_{32})}{p_{22}^2}\\ \\
\dfrac{\alpha p_{11}^2p_{22}^2}{p_{33}(p_{22}^2+p_{32}^2)} & \dfrac{p_{22}(\beta p_{22}-p_{32})}{p_{33}}  & \dfrac{p_{33}(\beta p_{22}-p_{32})}{p_{22}}   \\
\end{pmatrix}
$
 \\
&&&\\
\hline
\end{tabular}}

{\footnotesize {FIGURE 3. ${\rm dim}(A^2)=2$; $\alpha, \beta \neq 0$; ${\rm dim}({\rm ann}(A))=0$;}}\label{tabla11}\\
\footnotesize{$A$ has not Property (2LI); four non-zero entries in the first and second columns.}
\end{center}
\bigskip

 \noindent
 {\bf Case 2.4} $M_B$ has five non-zero entries in the first and second columns.
\medskip

\noindent
There are only 6 possibilities: those for which we place only one zero in one place of the first column or of the second column. There are four types which are listed below.
\bigskip

\begin{center}

\scalebox{0.7}{

\begin{tabular}{|c||c||c||c|}

\hline
&&&\\
Type & &(2,3)& Q\\
&&&\\[-0.2cm]
\hline
\hline
&&&\\[-0.2cm]
23
&
$\color{ao(english)}{
\begin{pmatrix}
1 & 0 & 0 \\
\alpha & 1 & 1 \\
\beta  & \gamma & \gamma
\end{pmatrix}}
$  &

$ \begin{pmatrix}
1 & 0 & 0 \\
\beta & \gamma & \gamma \\
\alpha  & 1 & 1
\end{pmatrix}$
&
$\tiny
\begin{pmatrix}
p_{11} & 0 & 0 \\
\dfrac{p_{11}^2(\alpha p_{22}+\beta p_{32})}{p_{22}^2+p_{32}^2} & p_{22}+\gamma p_{32} & \dfrac{p_{33}^2(p_{22}+\gamma p_{32})}{p_{22}^2}  \\ \\
\dfrac{p_{11}^2p_{22}(\beta p_{22}-\alpha p_{32})}{p_{33}(p_{22}^2+p_{32}^2)} & \dfrac{p_{22}(\gamma p_{22}-p_{32})}{p_{33}} & \dfrac{p_{33}(\gamma p_{22}-p_{32})}{p_{22}}
\end{pmatrix}
$ \\
&&&\\[-0.2cm]
\hline
\hline
&&&\\[-0.2cm]
24
&
$\color{ao(english)}{
\begin{pmatrix}
\alpha & 1 & 1 \\
\beta & 1  & 1 \\
0 & \gamma & \gamma
\end{pmatrix}}
$  &

$\color{ao(english)}{
\begin{pmatrix}
\alpha & 1 & 1 \\
0 & \gamma & \gamma \\
\beta & 1 & 1
\end{pmatrix}}
$
&
$\tiny
\begin{pmatrix}
\alpha p_{11} & \dfrac{p_{22}^2+ p_{32}^2}{p_{11}} & \dfrac{p_{33}^2(p_{22}^2+p_{32}^2)}{p_{11}p_{22}^2} \\
\dfrac{\beta p_{22}p_{11}^2}{p_{22}^2+p_{32}^2} & \gamma p_{32}+ p_{22} & \dfrac{p_{33}^2( p_{22}+\gamma p_{32})}{p_{22}^2}  \\ \\
\dfrac{-\beta p_{11}^2p_{22}p_{32}}{p_{33}(p_{22}^2+p_{32}^2)} & \dfrac{p_{22}(\gamma p_{22}-p_{32})}{p_{33}} & \dfrac{p_{33}(\gamma p_{22}- p_{32})}{p_{22}}
\end{pmatrix}$
\\
&&&\\[-0.2cm]
\hline
\hline
&&&\\[-0.2cm]
25
&
$\color{ao(english)}{
\begin{pmatrix}
0 & 1 & 1 \\
\alpha & 1 & 1 \\
\beta  & \gamma & \gamma
\end{pmatrix}}
$  &

$ \begin{pmatrix}
0 & 1 & 1 \\
\beta & \gamma & \gamma \\
\alpha  & 1 & 1
\end{pmatrix}$
&
$\tiny
\begin{pmatrix}
0 & \dfrac{p_{22}^2+ p_{32}^2}{p_{11}} & \dfrac{p_{33}^2(p_{22}^2+p_{32}^2)}{p_{11}p_{22}^2} \\
\dfrac{p_{11}^2(\alpha p_{22}+\beta p_{32})}{p_{22}^2+p_{32}^2} & p_{22}+\gamma p_{32} & \dfrac{p_{33}^2(p_{22}+\gamma p_{32})}{p_{22}^2}  \\ \\
\dfrac{-p_{11}^2p_{22}(\beta p_{22}-\alpha p_{32})}{p_{33}(p_{22}^2+p_{32}^2)} & \dfrac{p_{22}(\gamma p_{22}-\beta p_{32})}{p_{33}} & \dfrac{p_{33}(\gamma p_{22}-\beta p_{32})}{p_{22}}
\end{pmatrix} $
 \\
&&&\\[-0.2cm]
\hline
\hline
&&&\\[-0.2cm]
26
&
$\color{ao(english)}{
\begin{pmatrix}
\alpha & 1 & 1 \\
\beta & 0 & 0 \\
\gamma  & 1 & 1
\end{pmatrix}}
$  &
$\color{ao(english)}{
\begin{pmatrix}
\alpha & 1 & 1 \\
\gamma & 1 & 1 \\
\beta  & 0 & 0
\end{pmatrix}}$
&
$\tiny
\begin{pmatrix}
\alpha p_{11} & \dfrac{p_{22}^2+ p_{32}^2}{p_{11}} & \dfrac{p_{33}^2(p_{22}^2+p_{32}^2)}{p_{11}p_{22}^2} \\
\dfrac{p_{11}^2(\beta p_{22}+\gamma p_{32})}{p_{22}^2+p_{32}^2} &  p_{32} & \dfrac{ p_{32}p_{33}^2}{p_{22}^2}  \\ \\
\dfrac{p_{11}^2p_{22}(\gamma p_{22}-\beta p_{32})}{p_{33}(p_{22}^2+p_{32}^2)} & \dfrac{ p_{22}^2}{p_{33}} &  p_{33}
\end{pmatrix}$
 \\
&&&\\
\hline
\end{tabular}}

{\footnotesize {FIGURE 4. ${\rm dim}(A^2)=2$; $\alpha, \beta, \gamma \neq 0$; ${\rm dim}({\rm ann}(A))=0$;}}\label{tabla12}\\
\footnotesize{$A$ has not Property (2LI); five non-zero entries in the first and second columns.}
\end{center}
\bigskip

\noindent
  {\bf Case 2.5} $M_B$ has six non-zero entries in the first and second columns.
\medskip

\noindent
The condition that the entries of the matrix must satisfy is one of the following: $\alpha \neq \beta$, or $\lambda \neq \gamma$ or $\alpha\lambda \neq \beta \gamma$. There is only one possibility.
\bigskip

\begin{center}

\scalebox{0.7}{

\begin{tabular}{|c||c||c||c|}

\hline
&&&\\
Type & &(2,3)& Q\\
&&&\\[-0.2cm]
\hline
\hline
&&&\\[-0.2cm]
27
&
$\color{ao(english)}{
\begin{pmatrix}
\alpha & 1 & 1  \\
\beta & 1 & 1 \\
\gamma  & \lambda & \lambda
\end{pmatrix}}
$  &
${
\begin{pmatrix}
\alpha & 1 & 1 \\
\gamma & \lambda & \lambda \\
\beta  & 1 & 1
\end{pmatrix}}$
&
$\tiny
\begin{pmatrix}
\alpha p_{11} & \dfrac{p_{22}^2+ p_{32}^2}{p_{11}} & \dfrac{p_{33}^2(p_{22}^2+p_{32}^2)}{p_{11}p_{22}^2} \\
\dfrac{p_{11}^2(\gamma p_{32}-\beta p_{22})}{p_{22}^2+p_{32}^2} & \lambda p_{32}+ p_{22} & \dfrac{p_{33}^2(\lambda p_{32}+p_{22})}{p_{22}^2}  \\ \\
\dfrac{p_{11}^2p_{22}(\gamma p_{22}-\beta p_{32})}{p_{33}(p_{22}^2+p_{32}^2)} & \dfrac{p_{22}(\lambda p_{22}-p_{32})}{p_{33}} & \dfrac{p_{33}(\lambda p_{22}-p_{32})}{p_{22}}
\end{pmatrix}$
\\
&&&\\
\hline
\end{tabular}}

{\footnotesize {FIGURE 5. ${\rm dim}(A^2)=2$; $\alpha, \beta, \gamma, \lambda \neq 0$; ${\rm dim}({\rm ann}(A))=0$;}}\label{tabla13}\\
\footnotesize{$A$ has not Property (2LI); six non-zero entries in the first and second columns.}
\end{center}

These tables give us a first classification, that can be redundant in some cases. Now we  study if algebras having different types are isomorphic or not. The last step will be to study if algebras in the same type are isomorphic.

\begin{enumerate}[\rm $\circ$]
\item The evolution algebra given in Type 1 is included in the parametric family of algebras of Type 4.
\item The evolution algebra given inType 2 is included in the parametric family of  algebras of Type 11.
\item The evolution algebra given in Types 3, 12 and 13 are included in the parametric family of algebras of Type 20.
\item The parametric families of evolution algebras given in Types 5, 10, 16  and 17 are included in the parametric family of  algebras of Type 23.
\item The parametric families of evolution algebras given in Types 6, 7, 19 and 22 are included in the parametric family of  algebras of Type 25.
\item The parametric family of evolution algebras given in Type 8 is included in the parametric family of  algebras of Type 21.
\item The parametric family of evolution algebras given in Type 9 is included in the parametric family of  algebras of Type 18.
\item The parametric families of evolution algebras given in Types 14, 15, 24 and 26 are included in the parametric family of  algebras of Type 27.
\end{enumerate}
 \noindent

Therefore, there are eight subtypes of parametric families of evolution algebras, which are listed below.
\medskip
\begin{center}
{\rm$S$}\
\scalebox{0.78}
{
$
\begin{pmatrix}
1 & 0 &  0 \\
0 & 1 & 1  \\
0 & \alpha & \alpha
\end{pmatrix},
$
$
\begin{pmatrix}
0 & 1 & 1 \\
1 & 0 & 0 \\
\alpha & 0 & 0
\end{pmatrix},
$
$
\begin{pmatrix}
\alpha & 1 &  1 \\
0 & 1 & 1  \\
0 & \beta & \beta
\end{pmatrix},
$
$
\begin{pmatrix}
0 & 0 & 0 \\
1 & 1 & 1  \\
\alpha & \beta & \beta
\end{pmatrix},
$
$
\begin{pmatrix}
1 & 1 & 1 \\
\alpha & 0 & 0 \\
\beta & 0 & 0
\end{pmatrix},
$
$
\begin{pmatrix}
1 & 0 & 0 \\
\alpha & 1 & 1 \\
\beta & \gamma & \gamma
\end{pmatrix},
$
$\begin{pmatrix}
0 & 1 & 1 \\
\alpha & 1 & 1 \\
\beta & \gamma & \gamma
\end{pmatrix},
$
$
\begin{pmatrix}
\alpha & 1 & 1 \\
\beta & 1 & 1 \\
\gamma & \lambda & \lambda
\end{pmatrix}.
$
}
\end{center}
\medskip

\begin{remark}\label{CerosInvariantesEnS}
\rm
Note that these matrices are precisely those appearing in the Figures for which the change of basis matrices of type Q leaves invariant the number of non-zero entries and its place in the structure matrix. This does not mean that the number of non-zero entries is preserved (see, for example, in Figure 2, that the first matrix of Type 5 has four non-zero entries while the third matrix in the same line has seven).
\end{remark}
Now, we will analyze when the resulting parametric families of evolution algebras are mutually isomorphic. In some cases, we will reduce the number of parameters and some of these parametric families will be isomorphic to one of the known evolution algebras.

\noindent
Every evolution algebra with structure matrix $\begin{pmatrix}
1 & 0 &  0 \\
0 & 1 & 1  \\
0 & \alpha & \alpha
\end{pmatrix}$   satifying $\alpha ^2 + 1 \neq 0$ is isomorphic to the evolution algebra given by the structure matrix $\begin{pmatrix}
1 & 0 &  0 \\
0 & 1 & 1  \\
0 & 1 & 1
\end{pmatrix}$. Indeed, if $\alpha \neq -1$, we take the change of basis matrix
$$\tiny
\begin{pmatrix}
1 & 0 &  0 \\
0 & \dfrac{1+\alpha}{1+\alpha^2} & \dfrac{1-\alpha}{1+\alpha^2}  \\ \\
0 & \dfrac{-1+\alpha}{1+\alpha^2} & \dfrac{1+\alpha}{1+\alpha^2}
\end{pmatrix}.$$

In case of $\alpha = -1$, we assume the change of basis matrix

\begin{eqnarray}\label{matrizigual}
\begin{pmatrix}
1 & 0 &  0  \\
0 & 1 & 0 \\
0 & 0 & -1
\end{pmatrix}.
\end{eqnarray}

\noindent
The evolution algebra with structure matrix
$\begin{pmatrix}
0 & 1 &  1 \\
1 & 0 & 0  \\
\alpha & 0 & 0
\end{pmatrix}$  satisfying $\alpha^2 + 1 \neq 0$ is isomorphic to the evolution algebra given by the structure matrix
$\begin{pmatrix}
0 & 1 & 1 \\
1 & 0 & 0  \\
1 & 0 & 0
\end{pmatrix}$. Indeed, if $\alpha \neq 1, -1$, we take the change of basis matrix
$$\tiny
\begin{pmatrix}
 \sqrt[3]{\frac{2}{1+\alpha^2}} & 0 &  0 \\ \\
0 & \dfrac{1-\alpha}{\sqrt[3]{2(1+\alpha^2)^2}}& \dfrac{1+\alpha}{\sqrt[3]{2(1+\alpha^2)^2}}\\ \\
0 & \dfrac{1+\alpha}{\sqrt[3]{2(1+\alpha^2)^2}} & \dfrac{\alpha-1}{\sqrt[3]{2(1+\alpha^2)^2}}
\end{pmatrix}.$$

If $\alpha = -1$, we consider again the change of basis matrix given in \eqref{matrizigual}.

\noindent
Every evolution algebra  with structure matrix $\begin{pmatrix}
\alpha & 1 &  1 \\
0 & 1 & 1  \\
0 & \beta & \beta
\end{pmatrix}$  satisfying $\beta^2 + 1 \neq 0$ is isomorphic to the evolution algebra given by the structure matrix
$\begin{pmatrix}
\alpha' & 1 &  1 \\
0 & 1 & 1  \\
0 & 1 & 1
\end{pmatrix}$ for some $\alpha'\in \K$. Indeed, if $\beta \neq -1$, we take the change of basis matrix

$$\tiny
\begin{pmatrix}
\dfrac{2}{1+\beta^2} & 0 &  0 \\
0 & \dfrac{1-\beta}{1+\beta^2} & \dfrac{1+\beta}{1+\beta^2}  \\ \\
0 & \dfrac{1+\beta}{1+\beta^2} & \dfrac{\beta -1}{1+\beta^2}
\end{pmatrix}.$$

In case of $\beta = -1$, we can also consider the change of basis matrix given in \eqref{matrizigual}.

\noindent
The evolution algebra with structure matrix
$\begin{pmatrix}
0 & 0 &  0 \\
1 & 1 & 1  \\
\alpha & \beta & \beta
\end{pmatrix}$  satisfying $\beta^2 + 1 \neq 0$ is isomorphic to the evolution algebra  given by the structure matrix $\begin{pmatrix}
0 & 0 &  0 \\
1 & 1 & 1  \\
\alpha' & 1 & 1
\end{pmatrix}$ for some ${\alpha'}\in \K$. Indeed, if $\beta \neq -1$, we take the change of basis matrix

$$\tiny
\begin{pmatrix}
\dfrac{\sqrt{2}}{\sqrt{1+\alpha+\beta(-1+\alpha)}} & 0 &  0 \\
0 & \dfrac{1-\beta}{1+\beta^2} & \dfrac{1+\beta}{1+\beta^2}  \\ \\
0 & \dfrac{1+\beta}{1+\beta^2} & \dfrac{\beta -1}{1+\beta^2}
\end{pmatrix}.$$

In case of $\beta = -1$, we take again the change of basis matrix  given in \eqref{matrizigual}.

\noindent
Every evolution algebra with structure matrix
$\begin{pmatrix}
1 & 1 & 1 \\
\alpha & 0 & 0  \\
\beta & 0 & 0
\end{pmatrix}$  satisfying $\alpha^2 + \beta^2  \neq 0$ and $ \beta ^2 \neq 1$ is isomorphic to the evolution algebra having  structure matrix $\begin{pmatrix}
1 & 1 &  1 \\
1 & 0 & 0  \\
\beta' & 0 & 0
\end{pmatrix}$, for some $\beta'\in\K$. Indeed, if $\alpha \neq -1$, we take the change of basis matrix
$$\tiny
\begin{pmatrix}
1 & 0 &  0 \\
0 &  \dfrac{\alpha-\beta {\rm s}}{\alpha ^2+\beta ^2} & \dfrac{-(\beta+\alpha {\rm s})}{\alpha^2+\beta^2} \\ \\
0 &  \dfrac{\beta+\alpha {\rm s}}{\alpha^2+\beta^2} & \dfrac{\alpha-\beta {\rm s}}{\alpha ^2+\beta ^2}
\end{pmatrix},$$

\noindent
where s=$\sqrt{-1+\alpha^2+\beta^2}$.

\noindent
For $\alpha = -1$, consider the change of basis matrix:
$\begin{pmatrix}
1 & 0 &  0 \\
0 & -1 & 0  \\
0 & 0 & 1
\end{pmatrix}.$

On the other hand, every evolution algebra with structure matrix $\begin{pmatrix}
1 & 1 &  1 \\
\alpha & 0 & 0  \\
1 & 0 & 0
\end{pmatrix}$ $(\beta =1)$ and $\begin{pmatrix}
1 & 1 &  1 \\
\alpha & 0 & 0  \\
-1 & 0 & 0
\end{pmatrix}$ $(\beta =-1)$ is isomorphic to the evolution algebra with structure matrix
$\begin{pmatrix}
1 & 1 &  1 \\
1 & 0 & 0  \\
\beta & 0 & 0
\end{pmatrix}$. Indeed, take the new natural bases $\{e_1,e_3,e_2\}$ and $\{e_1,-e_3,e_2\}$, respectively.

\noindent
The evolution algebra with structure matrix
$\begin{pmatrix}
1 & 0 &  0 \\
\alpha & 1 & 1  \\
\beta & \gamma & \gamma
\end{pmatrix}$  with $\gamma^2 + 1 \neq 0$ is isomorphic to the evolution algebra given by the structure matrix $\begin{pmatrix}
1 & 0 &  0 \\
\alpha' & 1 & 1  \\
\beta' & 1 & 1
\end{pmatrix}$ for certain $\alpha', \beta' \in \K$. Indeed, if $\gamma \neq -1$, we take the change of basis matrix
$$\tiny
\begin{pmatrix}
1 & 0 &  0 \\
0 & \dfrac{1+\gamma}{1+\gamma^2} & \dfrac{1-\gamma}{1+\gamma^2}  \\ \\
0 & \dfrac{\gamma -1}{1+\gamma^2} & \dfrac{1+\gamma}{1+\gamma^2}
\end{pmatrix}.$$

If $\gamma = -1$, take again the change of basis matrix \eqref{matrizigual}.

\noindent
The evolution algebra with structure matrix
$\begin{pmatrix}
0 & 1 & 1 \\
\alpha & 1 & 1  \\
\beta & \gamma & \gamma
\end{pmatrix}$  with $\gamma^2 + 1 \neq 0$ is isomorphic to the  evolution algebra given by the structure matrix $\begin{pmatrix}
0 & 1 &  1 \\
\alpha' & 1 & 1  \\
\beta' & 1 & 1
\end{pmatrix}$ for certain $\alpha', \beta' \in \K$. Indeed, if $\gamma \neq -1$, we take the change of basis matrix
$$\tiny
\begin{pmatrix}
\dfrac{2}{1+\gamma^2} & 0 &  0 \\
0 & \dfrac{1+\gamma}{1+\gamma^2} & \dfrac{1-\gamma}{1+\gamma^2}  \\ \\
0 & \dfrac{\gamma -1}{1+\gamma^2} & \dfrac{1+\gamma}{1+\gamma^2}
\end{pmatrix}.$$

If $\gamma = -1$, also we take the change of basis matrix given in \eqref{matrizigual}.

\noindent
The  evolution algebra with structure matrix
$\begin{pmatrix}
\alpha & 1 & 1 \\
\beta & 1 & 1  \\
\gamma & \lambda & \lambda
\end{pmatrix}$  with $\lambda^2 + 1 \neq 0$ is isomorphic to the evolution algebra given by the structure matrix $\begin{pmatrix}
\alpha' & 1 &  1 \\
\beta' & 1 & 1  \\
\gamma' & 1 & 1
\end{pmatrix}$ for certain $\alpha', \beta' , \gamma' \in \K$. Indeed, if $\lambda \neq -1$, we take the change of basis matrix:

$$\tiny
\begin{pmatrix}
\dfrac{2}{1+\lambda^2} & 0 &  0 \\
0 & \dfrac{1+\lambda}{1+\lambda^2} & \dfrac{1-\lambda}{1+\lambda^2}  \\ \\
0 & \dfrac{\lambda -1}{1+\lambda^2} & \dfrac{1+\lambda}{1+\lambda^2}
\end{pmatrix}.$$

If $\gamma = -1$, we consider again the change of basis matrix determined in \eqref{matrizigual}.

Sumarizing,  whenever $e_3^2=e_2^2$ we obtain the following mutually non-isomorphic families of evolution algebras:

\medskip

\begin{center}

\scalebox{0.7}{

\begin{tabular}{|c||c||c|}

\hline
&&\\[-0.2cm]

$
\begin{pmatrix}
0 & 1 & 1 \\
1 & 0 & 0 \\
1 & 0 & 0
\end{pmatrix}$
&
$\begin{pmatrix}
0 & 1 & 1 \\
1 & 0 & 0 \\
\sqrt{-1} & 0 & 0
\end{pmatrix}$

  &
$
\begin{pmatrix}
0 & 1 & 1 \\
1 & 0 & 0 \\
-\sqrt{-1} & 0 & 0
\end{pmatrix}$

\\

&&\\
\hline

\end{tabular}}

{\footnotesize {TABLE 8.  ${\rm dim}(A^2)=2$; ${\rm dim}({\rm ann}(A))=0$;  }}\\
{\footnotesize {$A$ has not Property (2LI); four non-zero entries of the matrices in $S$}}
\end{center}

\medskip

\begin{center}

\scalebox{0.7}{

\begin{tabular}{|c||c||c|}

\hline
&&\\[-0.2cm]
$
\begin{pmatrix}
1 & 0 &  0 \\
0 & 1 & 1  \\
0 & 1 & 1
\end{pmatrix}$

  &
$\begin{pmatrix}
1 & 0 &  0 \\
0 & 1 & 1  \\
0 & \sqrt{-1} & \sqrt{-1}
\end{pmatrix}$

&
$
\begin{pmatrix}
1 & 0 &  0 \\
0 & 1 & 1  \\
0 & -\sqrt{-1} & -\sqrt{-1}
\end{pmatrix}$

\\
&&\\[-0.2cm]
\hline
\hline
&&\\[-0.2cm]

$
\begin{pmatrix}
1 & 1 & 1 \\
1 & 0 & 0 \\
\alpha & 0 & 0
\end{pmatrix}$
  &
$
\begin{pmatrix}
1 & 1 & 1 \\
\alpha \sqrt{-1} & 0 & 0 \\
\alpha & 0 & 0
\end{pmatrix}$

&
$
\begin{pmatrix}
1 & 1 & 1 \\
-\alpha \sqrt{-1} & 0 & 0 \\
\alpha & 0 & 0
\end{pmatrix}$

\\

&&\\
\hline

\end{tabular}}

{\footnotesize {TABLE 9.  ${\rm dim}(A^2)=2$; $\alpha \neq 0$; ${\rm dim}({\rm ann}(A))=0$;  }}\\
{\footnotesize {$A$ has not Property (2LI); five non-zero entries of the matrices in $S$. }}

\end{center}

\begin{center}

\scalebox{0.7}{

\begin{tabular}{|c||c||c|}

\hline
&&\\
$M_{B}$& $P_{B'B}$ & $M_{B'}$\\
&&\\[-0.2cm]
\hline
\hline
&&\\[-0.2cm]
${
\begin{pmatrix}
1 & 1 & 1 \\
1 & 0 & 0 \\
\alpha & 0 & 0
\end{pmatrix}}
$  &
${
\begin{pmatrix}
1 & 0 & 0 \\
0 & 1 & 0 \\
0 & 0 & -1
\end{pmatrix}}$
&
${
\begin{pmatrix}
1 & 1 & 1\\
1 & 0 & 0 \\
-\alpha & 0 & 0
\end{pmatrix}} $
\\

&&\\[-0.2cm]
\hline
\hline
&&\\[-0.2cm]

${
\begin{pmatrix}
1 & 1 & 1 \\
\alpha \sqrt{-1} & 0 & 0 \\
\alpha & 0 & 0
\end{pmatrix}}
$  &
${
\begin{pmatrix}
1 & 0 & 0 \\
0 & p_{22} & -\sqrt{1-p_{22}^2} \\
0 & \sqrt{1-p_{22}^2} & p_{22}
\end{pmatrix}}$
&
${
\begin{pmatrix}
1 & 1 & 1\\
\alpha \sqrt{-1}(p_{22}-\sqrt{-1+p_{22}^2}) & 0 & 0 \\
\alpha(p_{22}-\sqrt{-1+p_{22}^2}) & 0 & 0
\end{pmatrix}} $
\\

&&\\[-0.2cm]
\hline
\hline
&&\\[-0.2cm]

${
\begin{pmatrix}
1 & 1 & 1 \\
\alpha \sqrt{-1} & 0 & 0 \\
\alpha & 0 & 0
\end{pmatrix}}
$  &
${
\begin{pmatrix}
1 & 0 & 0 \\
0 & p_{22} & \sqrt{1-p_{22}^2} \\
0 & -\sqrt{1-p_{22}^2} & p_{22}
\end{pmatrix}}$
&
${
\begin{pmatrix}
1 & 1 & 1\\
\alpha \sqrt{-1}(p_{22}+\sqrt{-1+p_{22}^2}) & 0 & 0 \\
\alpha(p_{22}+\sqrt{-1+p_{22}^2}) & 0 & 0
\end{pmatrix}} $
\\
&&\\[-0.2cm]
\hline
\hline
&&\\[-0.2cm]

${
\begin{pmatrix}
1 & 1 & 1 \\
-\alpha \sqrt{-1} & 0 & 0 \\
\alpha & 0 & 0
\end{pmatrix}}
$  &
${
\begin{pmatrix}
1 & 0 & 0 \\
0 & p_{22} & -\sqrt{1-p_{22}^2} \\
0 & \sqrt{1-p_{22}^2} & p_{22}
\end{pmatrix}}$
&
${
\begin{pmatrix}
1 & 1 & 1\\
-\alpha \sqrt{-1}(p_{22}+\sqrt{-1+p_{22}^2}) & 0 & 0 \\
\alpha(p_{22}+\sqrt{-1+p_{22}^2}) & 0 & 0
\end{pmatrix}} $
\\
&&\\[-0.2cm]
\hline
\hline
&&\\[-0.2cm]
${
\begin{pmatrix}
1 & 1 & 1 \\
-\alpha \sqrt{-1} & 0 & 0 \\
\alpha & 0 & 0
\end{pmatrix}}
$  &
${
\begin{pmatrix}
1 & 0 & 0 \\
0 & p_{22} & \sqrt{1-p_{22}^2} \\
0 & -\sqrt{1-p_{22}^2} & p_{22}
\end{pmatrix}}$
&
${
\begin{pmatrix}
1 & 1 & 1\\
-\alpha \sqrt{-1}(p_{22}-\sqrt{-1+p_{22}^2}) & 0 & 0 \\
\alpha(p_{22}-\sqrt{-1+p_{22}^2}) & 0 & 0
\end{pmatrix}} $
\\
&&\\
\hline

\end{tabular}}

{\footnotesize {TABLE $9^\prime$.}}

\end{center}

\medskip

\begin{center}

\scalebox{0.7}{

\begin{tabular}{|c||c||c|}

\hline
&&\\[-0.2cm]

$
\begin{pmatrix}
0 & 0 & 0 \\
1 & 1 & 1  \\
\alpha & 1 & 1
\end{pmatrix}$

&

$
\begin{pmatrix}
0 & 0 & 0 \\
1 & 1 & 1  \\
\alpha & \sqrt{-1} & \sqrt{-1}
\end{pmatrix}$
&

$
\begin{pmatrix}
0 & 0 & 0 \\
1 & 1 & 1  \\
\alpha & -\sqrt{-1} & -\sqrt{-1}
\end{pmatrix}$

\\

&&\\
\hline

\end{tabular}}

{\footnotesize {TABLE 10.  ${\rm dim}(A^2)=2$; $\alpha \neq 0$; ${\rm dim}({\rm ann}(A))=0$;  }}\\
{\footnotesize {$A$ has not Property (2LI); six non-zero entries of the matrices in $S$. }}

\end{center}

\begin{center}

\scalebox{0.7}{

\begin{tabular}{|c||c||c|}

\hline
&&\\
$M_{B}$& $P_{B'B}$ & $M_{B'}$\\
&&\\[-0.2cm]
\hline
\hline
&&\\[-0.2cm]

${
\begin{pmatrix}
0 & 0 & 0 \\
1 & 1 & 1 \\
\alpha & 1 & 1
\end{pmatrix}}
$  &
${
\begin{pmatrix}
\dfrac{1}{\sqrt{\alpha}} & 0 & 0 \\
0 & 0 &  1 \\
0  & 1 & 0
\end{pmatrix}}$
&
${
\begin{pmatrix}
0 & 0 & 0 \\
1 & 1 & 1 \\
\dfrac{1}{\alpha} & 1 & 1
\end{pmatrix}} $
\\

&&\\[-0.2cm]
\hline
\hline
&&\\[-0.2cm]

${
\begin{pmatrix}
0 & 0 & 0 \\
1 & 1 & 1 \\
\alpha & \sqrt{-1} & \sqrt{-1}
\end{pmatrix}}
$  &
${
\begin{pmatrix}
p_{11} & 0 & 0 \\
0 & \dfrac{\sqrt{-1}+\alpha p_{11}^2}{2\sqrt{-1}+(-\sqrt{-1}+\alpha)p_{11}^2} & \dfrac{-1+p_{11}^2}{2\sqrt{-1}+(-\sqrt{-1}+\alpha)p_{11}^2}  \\
0  & \dfrac{1-p_{11}^2}{2\sqrt{-1}+(-\sqrt{-1}+\alpha)p_{11}^2} & \dfrac{\sqrt{-1}+\alpha p_{11}^2}{2\sqrt{-1}+(-\sqrt{-1}+\alpha)p_{11}^2}
\end{pmatrix}}$
&
${
\begin{pmatrix}
0 & 0 & 0 \\
1 & 1 & 1 \\
\sqrt{-1}+(-\sqrt{-1}+\alpha)p_{11}^2 & \sqrt{-1} & \sqrt{-1}
\end{pmatrix}} $
\\

&&\\[-0.2cm]
\hline
\hline
&&\\[-0.2cm]

${
\begin{pmatrix}
0 & 0 & 0 \\
1 & 1 & 1 \\
\alpha & -\sqrt{-1} & -\sqrt{-1}
\end{pmatrix}}
$  &
${
\begin{pmatrix}
p_{11} & 0 & 0 \\
0 & \dfrac{-\sqrt{-1}+\alpha p_{11}^2}{-2\sqrt{-1}+(\sqrt{-1}+\alpha)p_{11}^2} & \dfrac{-1+p_{11}^2}{-2\sqrt{-1}+(\sqrt{-1}+\alpha)p_{11}^2}  \\
0  & \dfrac{1-p_{11}^2}{-2\sqrt{-1}+(\sqrt{-1}+\alpha)p_{11}^2} & \dfrac{-\sqrt{-1}+\alpha p_{11}^2}{-2\sqrt{-1}+(\sqrt{-1}+\alpha)p_{11}^2}
\end{pmatrix}}$
&
${
\begin{pmatrix}
0 & 0 & 0 \\
1 & 1 & 1 \\
-\sqrt{-1}+(\sqrt{-1}+\alpha)p_{11}^2 & -\sqrt{-1} & -\sqrt{-1}
\end{pmatrix}} $
\\

&&\\
\hline

\end{tabular}}

{\footnotesize {TABLE $10^\prime$. }}

\end{center}

\medskip

\begin{center}

\scalebox{0.7}{

\begin{tabular}{|c||c||c|}

\hline
&&\\[-0.2cm]

$
\begin{pmatrix}
\alpha & 1 &  1 \\
0 & 1 & 1  \\
0 & 1 & 1
\end{pmatrix}$
&
$
\begin{pmatrix}
\alpha & 1 &  1 \\
0 & 1 & 1  \\
0 & \sqrt{-1} & \sqrt{-1}
\end{pmatrix}$
&
$
\begin{pmatrix}
\alpha & 1 &  1 \\
0 & 1 & 1  \\
0 & -\sqrt{-1} & -\sqrt{-1}
\end{pmatrix}$

\\
&&\\[-0.2cm]
\hline
\hline
&&\\[-0.2cm]

$
\begin{pmatrix}
1 & 0 & 0 \\
\alpha & 1 & 1 \\
\beta & 1 & 1
\end{pmatrix}$
&
$
\begin{pmatrix}
1 & 0 & 0 \\
\alpha & 1 & 1 \\
\beta & \sqrt{-1} & \sqrt{-1}
\end{pmatrix}$
 &
$
\begin{pmatrix}
1 & 0 & 0 \\
\alpha & 1 & 1 \\
\beta & -\sqrt{-1} & -\sqrt{-1}
\end{pmatrix}$
\\

&&\\
\hline

\end{tabular}}

{\footnotesize {TABLE 11.  ${\rm dim}(A^2)=2$; $\alpha\beta\neq 0$; ${\rm dim}({\rm ann}(A))=0$;  }}\\
{\footnotesize {$A$ has not Property (2LI); seven non-zero entries of the matrices in $S$. }}

\end{center}

\newpage

\begin{center}

\scalebox{0.7}{

\begin{tabular}{|c||c||c|}

\hline
&&\\
$M_{B}$& $P_{B'B}$ & $M_{B'}$\\
&&\\[-0.2cm]
\hline
\hline
&&\\[-0.2cm]
${
\begin{pmatrix}
\alpha & 1 & 1 \\
0 & 1 & 1 \\
0 & \sqrt{-1}  & \sqrt{-1}
\end{pmatrix}}
$  &
${
 \begin{pmatrix}
-1+2p_{22} & 0 & 0 \\
0 & p_{22} &  \sqrt{-1}(1-p_{22}) \\
0  & \sqrt{-1}(-1+p_{22}) & p_{22}
\end{pmatrix}}$
&
${
\begin{pmatrix}
\alpha(-1+2p_{22}) & 1 & 1 \\
0 & 1 & 1 \\
0 & \sqrt{-1} & \sqrt{-1}
\end{pmatrix}} $
\\
&&\\[-0.2cm]
\hline
\hline
&&\\[-0.2cm]
${
\begin{pmatrix}
\alpha & 1 & 1 \\
0 & 1 & 1 \\
0 & -\sqrt{-1}  & -\sqrt{-1}
\end{pmatrix}}
$  &
${
 \begin{pmatrix}
-1+2p_{22} & 0 & 0 \\
0 & p_{22} &  \sqrt{-1}(-1+p_{22}) \\
0  & -\sqrt{-1}(-1+p_{22}) & p_{22}
\end{pmatrix}}$
&
${
\begin{pmatrix}
\alpha(-1+2p_{22}) & 1 & 1 \\
0 & 1 & 1 \\
0 & -\sqrt{-1} & -\sqrt{-1}
\end{pmatrix}} $
\\
\\
&&\\[-0.2cm]
\hline
\hline
&&\\[-0.2cm]

${
\begin{pmatrix}
1 & 0 & 0 \\
\alpha & 1 & 1 \\
\beta & \sqrt{-1} & \sqrt{-1}
\end{pmatrix}}
$  &
${
\begin{pmatrix}
1 & 0 & 0 \\
0 & p_{22} & -\sqrt{-1}(-1+p_{22}) \\
0 & \sqrt{-1}(-1+p_{22}) & p_{22}
\end{pmatrix}}$
&
${
\begin{pmatrix}
1 & 0 & 0  \\
\dfrac{\sqrt{-1}\beta(-1+p_{22})+\alpha p_{22}}{-1+2p_{22}} & 1 & 1 \\
\dfrac{-\sqrt{-1}\alpha(-1+p_{22})+\beta p_{22}}{-1+2p_{22}} & \sqrt{-1} & \sqrt{-1}
\end{pmatrix}} $
\\
&&\\[-0.2cm]
\hline
\hline
&&\\[-0.2cm]

${
\begin{pmatrix}
1 & 0 & 0 \\
\alpha & 1 & 1 \\
\beta & -\sqrt{-1} & -\sqrt{-1}
\end{pmatrix}}
$  &
${
\begin{pmatrix}
1 & 0 & 0 \\
0 & p_{22} & \sqrt{-1}(-1+p_{22}) \\
0 & -\sqrt{-1}(-1+p_{22}) & p_{22}
\end{pmatrix}}$
&
${
\begin{pmatrix}
1 & 0 & 0  \\
\dfrac{-\sqrt{-1}\beta(-1+p_{22})+\alpha p_{22}}{-1+2p_{22}} & 1 & 1 \\
\dfrac{\sqrt{-1}\alpha(-1+p_{22})+\beta p_{22}}{-1+2p_{22}} & \sqrt{-1} & \sqrt{-1}
\end{pmatrix}} $
\\
&&\\
\hline

\end{tabular}}

{\footnotesize {TABLE $11^\prime$.}}

\end{center}

\begin{center}

\scalebox{0.7}{

\begin{tabular}{|c||c||c|}

\hline
&&\\[-0.2cm]

$
\begin{pmatrix}
0 & 1 & 1 \\
\alpha & 1 & 1 \\
\beta & 1 & 1
\end{pmatrix}$

&

$
\begin{pmatrix}
0 & 1 & 1 \\
\alpha & 1 & 1 \\
\beta & \sqrt{-1} & \sqrt{-1}
\end{pmatrix}$
&

$
\begin{pmatrix}
0 & 1 & 1 \\
\alpha & 1 & 1 \\
\beta & -\sqrt{-1} & -\sqrt{-1}
\end{pmatrix}$

\\

&&\\
\hline

\end{tabular}}

{\footnotesize {TABLE 12.  ${\rm dim}(A^2)=2$; $\alpha\beta \neq 0$; ${\rm dim}({\rm ann}(A))=0$;  }}\\
{\footnotesize {$A$ has not Property (2LI); eight non-zero entries of the matrices in $S$. }}

\end{center}

\begin{center}

\scalebox{0.7}{

\begin{tabular}{|c||c||c|}

\hline
&&\\
$M_{B}$& $P_{B'B}$ & $M_{B'}$\\
&&\\[-0.2cm]
\hline
\hline
&&\\[-0.2cm]

${
\begin{pmatrix}
0 & 1 & 1 \\
\alpha & 1 & 1 \\
\beta & \sqrt{-1} & \sqrt{-1}
\end{pmatrix}}
$  &
${
\begin{pmatrix}
1+2 \sqrt{-1}p_{23} & 0 & 0 \\
0 & 1+\sqrt{-1}p_{23} & p_{23} \\
0 & -p_{23} & 1+\sqrt{-1}p_{23}
\end{pmatrix}}$
&
${
\begin{pmatrix}
0 & 1 & 1  \\
(\sqrt{-1}-2p_{23})(\sqrt{-1}\beta p_{23}+\alpha (-\sqrt{-1}+p_{23})) & 1 & 1 \\
(1+2\sqrt{-1}p_{23})(\beta+ \alpha p_{23}+\beta \sqrt{-1}p_{23}) & \sqrt{-1} & \sqrt{-1}
\end{pmatrix}} $
\\
&&\\[-0.2cm]
\hline
\hline
&&\\[-0.2cm]
${
\begin{pmatrix}
0 & 1 & 1 \\
\alpha & 1 & 1 \\
\beta & -\sqrt{-1} & -\sqrt{-1}
\end{pmatrix}}
$  &
${
\begin{pmatrix}
1-2 \sqrt{-1}p_{23} & 0 & 0 \\
0 & 1-\sqrt{-1}p_{23} & p_{23} \\
0 & -p_{23} & 1-\sqrt{-1}p_{23}
\end{pmatrix}}$
&
${
\begin{pmatrix}
0 & 1 & 1  \\
(-1+2\sqrt{-1}p_{23})(\beta p_{23}+\alpha (-1+\sqrt{-1}p_{23})) & 1 & 1 \\
(1-2\sqrt{-1}p_{23})(\beta+ \alpha p_{23}-\beta \sqrt{-1}p_{23}) & -\sqrt{-1} & -\sqrt{-1}
\end{pmatrix}} $
\\
&&\\
\hline

\end{tabular}}

{\footnotesize {TABLE $12^\prime$.}}

\end{center}

\medskip

\begin{center}

\scalebox{0.7}{

\begin{tabular}{|c||c||c|}

\hline
&&\\[-0.2cm]
$
\begin{pmatrix}
\alpha & 1 & 1 \\
\beta & 1 & 1 \\
\gamma & 1 & 1
\end{pmatrix}$

&

$
\begin{pmatrix}
\alpha & 1 & 1 \\
\beta & 1 & 1 \\
\gamma & \sqrt{-1} & \sqrt{-1}
\end{pmatrix}$
&
$
\begin{pmatrix}
\alpha & 1 & 1 \\
\beta & 1 & 1 \\
\gamma & -\sqrt{-1} & -\sqrt{-1}
\end{pmatrix}$

\\

&&\\
\hline

\end{tabular}}

{\footnotesize {TABLE 13.  ${\rm dim}(A^2)=2$; $\alpha\beta\gamma \neq 0$; ${\rm dim}({\rm ann}(A))=0$;  }} \\
{\footnotesize {$A$ has not Property (2LI); nine non-zero entries of the matrices in $S$. }}

\end{center}
\newpage

\begin{center}

\scalebox{0.7}{

\begin{tabular}{|c||c||c|}

\hline
&&\\
$M_{B}$& $P_{B'B}$ & $M_{B'}$\\
&&\\[-0.2cm]
\hline
\hline
&&\\[-0.2cm]
${
\begin{pmatrix}
\alpha & 1 & 1 \\
\beta & 1 & 1 \\
\gamma & \sqrt{-1} & \sqrt{-1}
\end{pmatrix}}
$  &
${
\begin{pmatrix}
-1+2 p_{22} & 0 & 0 \\
0 & p_{22} & \sqrt{-1}(1-p_{22}) \\
0 & -\sqrt{-1}(1-p_{22}) & p_{22}
\end{pmatrix}}$
&
${
\begin{pmatrix}
\alpha(-1+2p_{22}) & 1 & 1  \\
(-1+2p_{22})(\sqrt{-1}\gamma(-1+ p_{22})+\beta p_{22}) & 1 & 1 \\
(-1+2p_{22})(-\sqrt{-1}\beta(-1+ p_{22})+\gamma p_{22}) & \sqrt{-1} & \sqrt{-1}
\end{pmatrix}} $
\\

&&\\[-0.2cm]
\hline
\hline
&&\\[-0.2cm]
${
\begin{pmatrix}
\alpha & 1 & 1 \\
\beta & 1 & 1 \\
\gamma & -\sqrt{-1} & -\sqrt{-1}
\end{pmatrix}}
$  &
${
\begin{pmatrix}
-1+2 p_{22} & 0 & 0 \\
0 & p_{22} & \sqrt{-1}(-1+p_{22}) \\
0 & -\sqrt{-1}(-1+p_{22}) & p_{22}
\end{pmatrix}}$
&
${
\begin{pmatrix}
\alpha(-1+2p_{22}) & 1 & 1  \\
(-1+2p_{22})(-\sqrt{-1}\gamma(-1+ p_{22})+\beta p_{22}) & 1 & 1 \\
(-1+2p_{22})(\sqrt{-1}\beta(-1+ p_{22})+\gamma p_{22}) & -\sqrt{-1} & -\sqrt{-1}
\end{pmatrix}} $
\\

&&\\
\hline

\end{tabular}}

{\footnotesize {TABLE $13^\prime$.}}

\end{center}


 \noindent
 {\bf Case 3} Assume $c_2=0, c_1 \neq 0$.

\noindent
Considering the natural basis $B'=\{e_{2},e_{1},e_{3} \}$ we obtain the following structure matrix:

$$ M_{B'}=
 \begin{pmatrix}
 \omega_{22}  & \omega_{21} & c_1\omega_{21}\\
 \omega_{12}  & \omega_{11} & c_1\omega_{11}\\
 \omega_{32}  & \omega_{31} & c_1\omega_{31}\\
  \end{pmatrix}, $$

\noindent
and now we are in the same conditions as in Case 2.\\
\medskip

\noindent
{\bf Case 4} Suppose $c_1=c_2=0$.

\noindent
Recall by \eqref{matrizrangodos} that the structure matrix is

$$ M_{B}=
 \begin{pmatrix}
 \omega_{11}  & \omega_{12} & 0\\
 \omega_{21}  & \omega_{22} & 0\\
 \omega_{31}  & \omega_{32} & 0\\
  \end{pmatrix}. $$

\begin{remark}\label{EntradasNoNulasEnFilas}
\rm
In what follows we are going to prove that the number of zero entries in the first and in the second rows in the structure matrix is preserved by any change of basis.
\end{remark}

With the explained goal  in mind, we study all the possible change of basis matrices. {Let} $B'$ be another natural basis and consider the change of basis matrix $P_{B'B}$. The equations \eqref{cuartas}, \eqref{quintas}, \eqref{segundas} and \eqref{terceras} give:

\begin{align*}
p_{11}p_{12}=0;\quad p_{21}p_{22}=0;\\
p_{11}p_{13}=0; \quad p_{21}p_{23}=0;\\
p_{12}p_{13}=0; \quad p_{22}p_{23}=0.\\
\end{align*}

It is easy to check that $P_{B'B}$ has two zero entries in the first and the second rows. Moreover, since $\vert P_{B'B}\vert \neq 0$, necessarily $p_{1i}p_{2j} \neq 0$ for $i,j \in \{1,2,3\}$ with  $i \neq j$. We distinguish the six different cases that appear in order to study the structure matrix $M_{B'}$.
\\

If
$
P_{B'B}=
\begin{pmatrix}
p_{11} & 0 & 0 \\
0 & p_{22} & 0 \\
p_{31} & p_{32} & p_{33}
\end{pmatrix}
$ with $p_{11}p_{22}p_{33} \neq 0$ then

\begin{equation}\label{matrixpermutc0}
\tiny
M_{B'}=
\begin{pmatrix}
\omega_{11} p_{11} & \dfrac{\omega_{12} p_{22}^2}{p_{11}} & 0 \\
\dfrac{\omega_{21} p_{11}^2}{p_{22}} & \omega_{22} p_{22} & 0 \\
\dfrac{p_{11} (\omega_{31} p_{11} p_{22}-\omega_{11} p_{31} p_{22}-\omega_{21} p_{11} p_{32})}{p_{22} p_{33}}  & \dfrac{p_{22} (\omega_{32} p_{11} p_{22}-\omega_{12} p_{31} p_{22}-\omega_{22} p_{11} p_{32})}{p_{11} p_{33}} & 0
\end{pmatrix}.
\end{equation}
\medskip

\noindent
If
$
P_{B'B}=
\begin{pmatrix}
p_{11} & 0 & 0 \\
0 & 0 & p_{23} \\
p_{31} & p_{32} & p_{33}
\end{pmatrix}
$  for $p_{11}p_{23}p_{32} \neq 0$ then

$$
\tiny
M_{B'}=
\begin{pmatrix}
\omega_{11} p_{11} & 0 &  \dfrac{\omega_{12} p_{23}^2}{p_{11}}\\
\dfrac{p_{11} (\omega_{31} p_{11} p_{23}-\omega_{11} p_{31} p_{23}-\omega_{21} p_{11} p_{33})}{p_{23} p_{32}}  & 0 & \dfrac{p_{23} (\omega_{32} p_{11} p_{23}-\omega_{12} p_{31} p_{23}-\omega_{22} p_{11} p_{33})}{p_{11} p_{32}}  \\
\dfrac{\omega_{21} p_{11}^2}{p_{23}} & 0 & \omega_{22} p_{23}   \\
\end{pmatrix}.$$
\medskip

\noindent
If
$
P_{B'B}=
\begin{pmatrix}
0 & p_{12} & 0 \\
p_{21} & 0 & 0 \\
p_{31} & p_{32} & p_{33}
\end{pmatrix}
$  where $p_{12}p_{21}p_{33} \neq 0$ then

\begin{equation}\label{matrixpermutc1}
\tiny
M_{B'}=
\begin{pmatrix}
\omega_{22} p_{21} & \dfrac{\omega_{21} p_{12}^2}{p_{21}} & 0 \\
\dfrac{\omega_{12} p_{21}^2}{p_{12}} & \omega_{11} p_{12} & 0 \\
\dfrac{p_{21} (\omega_{32} p_{12} p_{21}-\omega_{12} p_{32} p_{21}-\omega_{22} p_{12} p_{31})}{p_{12} p_{33}}  & \dfrac{p_{12} (\omega_{31} p_{12} p_{21}-\omega_{11} p_{32} p_{21}-\omega_{21} p_{12} p_{31})}{p_{21} p_{33}} & 0
\end{pmatrix}.
\end{equation}
\medskip

\noindent
If
$
P_{B'B}=
\begin{pmatrix}
0 & p_{12} & 0 \\
0 & 0 & p_{23} \\
p_{31} & p_{32} & p_{33}
\end{pmatrix}
$  with $p_{12}p_{23}p_{31} \neq 0$ then

$$
\tiny
M_{B'}=
\begin{pmatrix}
0 & \dfrac{p_{12} (\omega_{31} p_{12} p_{23}-\omega_{11} p_{32} p_{23}-\omega_{21} p_{12} p_{33})}{p_{23} p_{31}}  & \dfrac{p_{23} (\omega_{32} p_{12} p_{23}-\omega_{12} p_{32} p_{23}-\omega_{22} p_{12} p_{33})}{p_{12} p_{31}}\\
0 & \omega_{11} p_{12} & \dfrac{\omega_{12} p_{23}^2}{p_{12}} \\
0 & \dfrac{\omega_{21} p_{12}^2}{p_{23}} & \omega_{22} p_{23} \\
\end{pmatrix}.$$
\medskip

\noindent
If
$
P_{B'B}=
\begin{pmatrix}
0 & 0 & p_{13} \\
p_{21} & 0 & 0 \\
p_{31} & p_{32} & p_{33}
\end{pmatrix}
$ for $p_{13}p_{21}p_{32} \neq 0$ then

$$
\tiny
M_{B'}=
\begin{pmatrix}
\omega_{22} p_{21} & 0 &  \dfrac{\omega_{21} p_{13}^2}{p_{21}}\\
\dfrac{p_{21} (\omega_{32} p_{13} p_{21}-\omega_{12} p_{33} p_{21}-\omega_{22} p_{13} p_{31})}{p_{13} p_{32}}  & 0 & \dfrac{p_{13} (\omega_{31} p_{13} p_{21}-\omega_{11} p_{33} p_{21}-\omega_{21} p_{13} p_{31})}{p_{21} p_{32}}  \\
\dfrac{\omega_{12} p_{21}^2}{p_{13}} & 0 & \omega_{11} p_{13}   \\
\end{pmatrix}.$$
\medskip

\noindent
If
$
P_{B'B}=
\begin{pmatrix}
0 & 0 & p_{13} \\
0 & p_{22} & 0 \\
p_{31} & p_{32} & p_{33}
\end{pmatrix}
$  where $p_{13}p_{21}p_{32} \neq 0$  then

$$
\tiny
M_{B'}=
\begin{pmatrix}
0 & \dfrac{p_{22} (\omega_{32} p_{13} p_{22}-\omega_{12} p_{33} p_{22}-\omega_{22} p_{13} p_{32})}{p_{13} p_{31}}  & \dfrac{p_{13} (\omega_{31} p_{13} p_{22}-\omega_{11} p_{33} p_{22}-\omega_{21} p_{13} p_{32})}{p_{22} p_{31}}\\
0 & \omega_{22} p_{22} & \dfrac{\omega_{21} p_{13}^2}{p_{22}} \\
0 & \dfrac{\omega_{12} p_{22}^2}{p_{13}} & \omega_{11} p_{13} \\
\end{pmatrix}.$$
\medskip

Note that we only have to take in to account the change of basis matrices which transform a structure matrix having the third column equals zero into another one of the same type. These are those $P_{B'B}$ appearing in the first and in the third cases. We denote them by $Q'$ and by $Q''$, respectively. Looking at the different $M_{B'}$ that appear, we obtain the claim.

Then, if we omit the structure matrices which can be obtained from the permutation $(1,2)$, the only possibilities are:

\begin{equation*}
\footnotesize{
\left\{
\begin{pmatrix}
\omega_{11} & 0 &  0 \\
0 & \omega_{22} & 0  \\
\omega_{31} & \omega_{32} & 0
\end{pmatrix},
\begin{pmatrix}
0 & \omega_{12} &  0 \\
\omega_{21} & 0 & 0  \\
\omega_{31} & \omega_{32} & 0
\end{pmatrix},
\begin{pmatrix}
\omega_{11} & 0 &  0 \\
\omega_{21} & \omega_{22} & 0  \\
\omega_{31} & \omega_{32} & 0
\end{pmatrix},
\begin{pmatrix}
0 & \omega_{12} &  0 \\
\omega_{21} & \omega_{22} & 0  \\
\omega_{31} & \omega_{32} & 0
\end{pmatrix},
\begin{pmatrix}
\omega_{11} & \omega_{12} &  0 \\
\omega_{21} & \omega_{22} & 0  \\
\omega_{31} & \omega_{32} & 0
\end{pmatrix}
\right\}\  \bigcup}
\end{equation*}

\begin{equation*}
\footnotesize{
\left\{
\begin{pmatrix}
\omega_{11} & 0 &  0 \\
0 & 0 & 0  \\
\omega_{31} & \omega_{32} & 0
\end{pmatrix},
\begin{pmatrix}
\omega_{11} & \omega_{12} &  0 \\
0 & 0 & 0  \\
\omega_{31} & \omega_{32} & 0
\end{pmatrix},
\begin{pmatrix}
\omega_{11} & 0 &  0 \\
\omega_{21} & 0 & 0  \\
\omega_{31} & \omega_{32} & 0
\end{pmatrix}
\begin{pmatrix}
0 & \omega_{12} &  0 \\
0 & 0 & 0  \\
\omega_{31} & \omega_{32} & 0
\end{pmatrix}
\right\}}
\end{equation*}

\noindent
with $\omega_{ij} \neq 0$ for $i, j\in \{1, 2\}$.

According to \eqref{matrixpermutc0} and \eqref{matrixpermutc1}, we claim that we can remove the third row of the structure matrices of the first set and write 0 if and only if $\omega_{11}\omega_{22}-\omega_{12}\omega_{21} \neq 0$. For the matrix \eqref{matrixpermutc0} we consider $p_{11}=p_{22}=1$ and we have

\begin{align*}
\omega_{31}p_{11}p_{22}-\omega_{11}p_{31}p_{22}-\omega_{21}p_{11}p_{32}=\omega_{31}-\omega_{11}p_{31}-\omega_{21}p_{32} & =0; \\
\omega_{32}p_{11}p_{22}-\omega_{12}p_{31}p_{22}-\omega_{22}p_{11}p_{32}=\omega_{31}-\omega_{12}p_{31}-\omega_{22}p_{32} & =0. \\
\end{align*}

\noindent
So, this linear system has solution if $\omega_{11}\omega_{22}- \omega_{12}\omega_{21} \neq 0$.

If we take \eqref{matrixpermutc1}, we reason in the same way and our claim has been proved. \\

Now we can place 0 instead of $\omega_{31}$ in the first three matrices of the second set. Indeed, as in these structure matrices $\omega_{11} \neq 0$ and supposing $p_{11}=p_{22}=1$ we have the equation $\omega_{31}-\omega_{11}p_{31}-\omega_{21}p_{32} =0$ if $\omega_{21}\neq 0$ and $\omega_{31}-\omega_{11}p_{31}=0$ if $\omega_{21}=0$. In any case, the equations have always solution.

In the last structure matrix of the second set we can write 0 instead of $\omega_{32}$. For this, it is enough to take $p_{22}=p_{11}=0$ and  $p_{31}=\frac{\omega_{32}}{\omega_{12}}$.

Finally, we can obtain the maximum number of entries equal 1 by using \eqref{producto2}. When placing 1 is not possible we write the parameters $\alpha$, $\beta$ and $\gamma$. 
Summarizing, there are ten possibilities which are listed below.

\begin{center}

\scalebox{0.7}{
\begin{tabular}{|c||c|}

\hline
&\\
&(1,2)\\
&\\[-0.2cm]
\hline
\hline
&\\[-0.2cm]
${
\begin{pmatrix}
1 & 0 & 0 \\
0 & 0 & 0 \\
0  & 1 & 0
\end{pmatrix}}
$  &

${
 \begin{pmatrix}
0 & 0 & 0 \\
0 & 1 & 0 \\
1  & 0 & 0
\end{pmatrix}}$
 \\
&\\[-0.2cm]
\hline
\hline
&\\[-0.2cm]

${
\begin{pmatrix}
0 & 1 & 0 \\
0 & 0 & 0 \\
1  & 0 & 0
\end{pmatrix}}
$  &

${
 \begin{pmatrix}
0 & 0 & 0 \\
1 & 0 & 0 \\
0  & 1 & 0
\end{pmatrix}}$
\\
&\\
\hline
\end{tabular}
}

{\footnotesize {TABLE 14. {${\rm dim}(A^2)=2$;  ${\rm dim}({\rm ann}(A))=1$; }}}\\

{\footnotesize {one non-zero entry in the first and second rows.}}

\end{center}

\begin{center}

\scalebox{0.7}{
\begin{tabular}{|c||c|}

\hline
&\\
&(1,2)\\
&\\[-0.2cm]
\hline
\hline
&\\[-0.2cm]

${
\begin{pmatrix}
1 & 0 & 0 \\
0 & 1 & 0 \\
0  & 0 & 0

\end{pmatrix}}
$  &

${
 \begin{pmatrix}
1 & 0 & 0 \\
0 & 1 & 0 \\
0  & 0 & 0
\end{pmatrix}}$
\\
&\\[-0.2cm]
\hline
\hline
&\\[-0.2cm]

${
\begin{pmatrix}
0 & 1 & 0 \\
1 & 0 & 0 \\
0  & 0 & 0
\end{pmatrix}}
$  &

${
 \begin{pmatrix}
0 & 1 & 0 \\
1 & 0 & 0 \\
0  & 0 & 0
\end{pmatrix}}$
\\
&\\[-0.2cm]
\hline
\hline
&\\[-0.2cm]

${
\begin{pmatrix}
1 & 1 & 0 \\
0 & 0 & 0 \\
1  & 0 & 0

\end{pmatrix}}
$  &

${
 \begin{pmatrix}
0 & 0 & 0 \\
1 & 1 & 0 \\
0  & 1 & 0

\end{pmatrix}}$
 \\
&\\[-0.2cm]
\hline
\hline
&\\[-0.2cm]

${
\begin{pmatrix}
1 & 0 & 0 \\
1 & 0 & 0 \\
0  & 1 & 0

\end{pmatrix}}
$  &

${
 \begin{pmatrix}
0 & 1 & 0 \\
0 & 1 & 0 \\
1  & 0 & 0

\end{pmatrix}}$
 \\
&\\
\hline
\end{tabular}
}

{\footnotesize {TABLE 15. {${\rm dim}(A^2)=2$; ${\rm dim}({\rm ann}(A))=1$;}}}\\

{\footnotesize {two non-zero entries in the first and second rows.}}

\end{center}

\begin{center}

\scalebox{0.7}{
\begin{tabular}{|c||c|}

\hline
&\\
&(1,2)\\
&\\[-0.2cm]
\hline
\hline
&\\[-0.2cm]
${
\begin{pmatrix}
\alpha & 0 & 0 \\
1 & 1 & 0 \\
0  & 0 & 0
\end{pmatrix}}
$  &

${
 \begin{pmatrix}
1 & 1 & 0 \\
0 & \alpha & 0 \\
0  & 0 & 0

\end{pmatrix}}$
 \\
&\\[-0.2cm]
\hline
\hline
&\\[-0.2cm]

${
\begin{pmatrix}
0 & \alpha & 0  \\
1 & 1 & 0 \\
0  & 0 & 0

\end{pmatrix}}
$  &

${
 \begin{pmatrix}
1 & 1 & 0 \\
\alpha & 0 & 0 \\
0  & 0 & 0

\end{pmatrix}}$
\\
&\\
\hline
\end{tabular}
}

{\footnotesize {TABLE 16. {${\rm dim}(A^2)=2$; ${\rm dim}({\rm ann}(A))=1$;}}}\\

{\footnotesize {three non-zero entries in the first and second rows.}}

\end{center}

\begin{center}

\scalebox{0.7}{

\begin{tabular}{|c||c||c|}

\hline
&&\\
$M_{B}$& $P_{B'B}$ & $M_{B'}$\\
&&\\[-0.2cm]
\hline
\hline
&&\\[-0.2cm]
${
\begin{pmatrix}
 \alpha & 0 & 0 \\
1 & 1 & 0 \\
0 & 0  & 0
\end{pmatrix}}
$  &
${
 \begin{pmatrix}
-1 & 0 &  0  \\
0 & 1  &  0 \\
0  & 0 & p_{33}
\end{pmatrix}}$
&
${
\begin{pmatrix}
-\alpha & 0 & 0 \\
1 & 1 & 0  \\
0 & 0  & 0
\end{pmatrix}} $
\\
&&\\[-0.2cm]
\hline
\hline
&&\\[-0.2cm]

${
\begin{pmatrix}
0 & \alpha & 0 \\
1 & 1  & 0 \\
0 & 0  & 0
\end{pmatrix}}
$  &
${
 \begin{pmatrix}
-1 & 0 & 0 \\
0 & 1 & 0 \\
0 & 0 & p_{33}
\end{pmatrix}}$
&
${
\begin{pmatrix}
0 & -\alpha & 0 \\
1 & 1 & 0 \\
0 & 0 & 0
\end{pmatrix}} $
\\

&&\\[-0.2cm]
\hline
\hline
&&\\[-0.2cm]

${
\begin{pmatrix}
1 & 1 & 0 \\
\alpha & \beta & 0 \\
0 & 0 & 0
\end{pmatrix}}
$  &
${
 \begin{pmatrix}
1 & 0 & 0 \\
0 & -1 & 0 \\
0 & 0 & p_{33}
\end{pmatrix}}$
&
${
\begin{pmatrix}
1 & 1 & 0 \\
- \alpha & -\beta & 0 \\
0 & 0 & 0
\end{pmatrix}} $
\\
&&\\
\hline
\end{tabular}
}

{\footnotesize {TABLE $16^\prime$.}}

\end{center}

\begin{center}

\scalebox{0.7}{
\begin{tabular}{|c||c|}

\hline
&\\
&(1,2)\\
&\\[-0.2cm]
\hline
\hline
&\\[-0.2cm]

${
\begin{pmatrix}
1 & 1 & 0 \\
\alpha & \beta & 0 \\
0  & 0 & 0

\end{pmatrix}}
$  &

${
 \begin{pmatrix}
\beta & \alpha & 0 \\
1 & 1 & 0 \\
0  & 0 & 0

\end{pmatrix}}$
 \\
&\\[-0.2cm]
\hline
\hline
&\\[-0.2cm]

${
\begin{pmatrix}
1 & 1 & 0 \\
\alpha & \alpha & 0 \\
1  & \beta & 0

\end{pmatrix}}
$  &

${
 \begin{pmatrix}
\alpha & \alpha & 0 \\
1 & 1 & 0 \\
\beta  & 1 & 0
\end{pmatrix}}$

\\
&\\
\hline
\end{tabular}
}

{\footnotesize {TABLE 17. {${\rm dim}(A^2)=2$; ${\rm dim}({\rm ann}(A))=1$;}}}\\

{\footnotesize {four non-zero entries in the first and second rows.}}

\end{center}
\newpage

\begin{center}

\scalebox{0.7}{

\begin{tabular}{|c||c||c|}

\hline
&&\\
$M_{B}$& $P_{B'B}$ & $M_{B'}$\\
&&\\[-0.2cm]
\hline
\hline
&&\\[-0.2cm]

${
\begin{pmatrix}
1 & 1 & 0 \\
\alpha & \beta & 0 \\
0 & 0 & 0
\end{pmatrix}}
$  &
${
 \begin{pmatrix}
1 & 0 & 0 \\
0 & -1 & 0 \\
0 & 0 & p_{33}
\end{pmatrix}}$
&
${
\begin{pmatrix}
1 & 1 & 0 \\
- \alpha & -\beta & 0 \\
0 & 0 & 0
\end{pmatrix}} $
\\

&&\\[-0.2cm]
\hline
\hline
&&\\[-0.2cm]

${
\begin{pmatrix}
1 & 1 & 0 \\
\alpha & \beta & 0 \\
0 & 0 & 0
\end{pmatrix}}
$  &
${
 \begin{pmatrix}
0 & \dfrac{\sqrt{\alpha \beta}}{\alpha \beta} & 0 \\
\dfrac{1}{\beta} & 0 & 0 \\
0 & 0 & p_{33}
\end{pmatrix}}$
&
${
\begin{pmatrix}
1 & 1 & 0 \\
\dfrac{\alpha \beta}{\beta^2} & \dfrac{\sqrt{\alpha \beta}}{\alpha \beta} & 0 \\
0 & 0 & 0
\end{pmatrix}} $
\\

&&\\[-0.2cm]
\hline
\hline
&&\\[-0.2cm]

${
\begin{pmatrix}
1 & 1 & 0 \\
\alpha & \beta & 0 \\
0 & 0 & 0
\end{pmatrix}}
$  &
${
 \begin{pmatrix}
0 & -\dfrac{\sqrt{\alpha \beta}}{\alpha \beta} & 0 \\
\dfrac{1}{\beta} & 0 & 0 \\
0 & 0 & p_{33}
\end{pmatrix}}$
&
${
\begin{pmatrix}
1 & 1 & 0 \\
-\dfrac{\alpha \beta}{\beta^2} & -\dfrac{\sqrt{\alpha \beta}}{\alpha \beta} & 0 \\
0 & 0 & 0
\end{pmatrix}} $
\\

&&\\[-0.2cm]
\hline
\hline
&&\\[-0.2cm]

${
\begin{pmatrix}
1 & 1 & 0 \\
\alpha & \beta & 0 \\
1 & 1 & 0
\end{pmatrix}}
$  &
${
 \begin{pmatrix}
1 & 0 & 0 \\
0 & -1 & 0 \\
p_{31} & 0 & 1-p_{31}
\end{pmatrix}}$
&
${
\begin{pmatrix}
1 & 1 & 0 \\
- \alpha & -\beta & 0 \\
1 & 1 & 0
\end{pmatrix}} $
\\
&&\\[-0.2cm]
\hline
\hline
&&\\[-0.2cm]

${
\begin{pmatrix}
1 & 1 & 0 \\
\alpha & \beta & 0 \\
1 & 1 & 0
\end{pmatrix}}
$  &
${
 \begin{pmatrix}
0 & \dfrac{\sqrt{\alpha \beta}}{\alpha \beta} & 0 \\
\dfrac{1}{\beta} & 0 & 0 \\
p_{31} & \dfrac{\sqrt{\alpha \beta}}{\alpha \beta} & - p_{31}
\end{pmatrix}}$
&
${
\begin{pmatrix}
1 & 1 & 0 \\
\dfrac{\alpha \beta}{\beta^2} & \dfrac{\sqrt{\alpha \beta}}{\alpha \beta} & 0 \\
1 & 1 & 0
\end{pmatrix}} $
\\

&&\\[-0.2cm]
\hline
\hline
&&\\[-0.2cm]

${
\begin{pmatrix}
1 & 1 & 0 \\
\alpha & \beta & 0 \\
1 & 1 & 0
\end{pmatrix}}
$  &
${
 \begin{pmatrix}
0 & -\dfrac{\sqrt{\alpha \beta}}{\alpha \beta} & 0 \\
\dfrac{1}{\beta} & 0 & 0 \\
p_{31} & -\dfrac{\sqrt{\alpha \beta}}{\alpha \beta} & - p_{31}
\end{pmatrix}}$
&
${
\begin{pmatrix}
1 & 1 & 0 \\
-\dfrac{\alpha \beta}{\beta^2} & -\dfrac{\sqrt{\alpha \beta}}{\alpha \beta} & 0 \\
1 & 1 & 0
\end{pmatrix}} $
\\
&&\\[-0.2cm]
\hline
\hline
&&\\[-0.2cm]

${
\begin{pmatrix}
1 & 1 & 0 \\
\alpha & \alpha & 0 \\
1 & \beta & 0
\end{pmatrix}}
$  &
${
 \begin{pmatrix}
1 & 0 & 0 \\
0 & -1 & 0 \\
1+\alpha p_{32}- p_{33} & p_{32}  &  p_{33}
\end{pmatrix}}$
&
${
\begin{pmatrix}
1 & 1 & 0 \\
-\alpha  & -\alpha & 0 \\
1 & \dfrac{-1+\beta + p_{33}}{p_{33}} & 0
\end{pmatrix}} $
\\

&&\\[-0.2cm]
\hline
\hline
&&\\[-0.2cm]

${
\begin{pmatrix}
1 & 1 & 0 \\
\alpha & \alpha & 0 \\
1 & \beta & 0
\end{pmatrix}}
$  &
${
 \begin{pmatrix}
1 & 0 & 0 \\
0 & 1 & 0 \\
1-\alpha p_{32}- p_{33} & p_{32}  &  p_{33}
\end{pmatrix}}$
&
${
\begin{pmatrix}
1 & 1 & 0 \\
\alpha  & \alpha & 0 \\
1 & \dfrac{-1+\beta + p_{33}}{p_{33}} & 0
\end{pmatrix}} $
\\

&&\\[-0.2cm]
\hline
\hline
&&\\[-0.2cm]

${
\begin{pmatrix}
1 & 1 & 0 \\
\alpha & \alpha & 0 \\
1 & \beta & 0
\end{pmatrix}}
$  &
${
 \begin{pmatrix}
0 & -\dfrac{1}{\alpha} & 0 \\
\dfrac{1}{\alpha} & 0 & 0 \\
\dfrac{\beta + \alpha (p_{32}- \alpha p_{33})}{\alpha ^2} & p_{32}  &  p_{33}
\end{pmatrix}}$
&
${
\begin{pmatrix}
1 & 1 & 0 \\
-\dfrac{1}{\alpha}  & -\dfrac{1}{\alpha}& 0 \\
1 & \dfrac{\alpha ^2 p_{33} +1 -\beta}{\alpha ^2 p_{33}} & 0
\end{pmatrix}} $
\\

&&\\[-0.2cm]
\hline
\hline
&&\\[-0.2cm]

${
\begin{pmatrix}
1 & 1 & 0 \\
\alpha & \alpha & 0 \\
1 & \beta & 0
\end{pmatrix}}
$  &
${
 \begin{pmatrix}
0 & \dfrac{1}{\alpha} & 0 \\
\dfrac{1}{\alpha} & 0 & 0 \\
\dfrac{\beta - \alpha (p_{32}- \alpha p_{33})}{\alpha ^2} & p_{32}  &  p_{33}
\end{pmatrix}}$
&
${
\begin{pmatrix}
1 & 1 & 0 \\
\dfrac{1}{\alpha}  & \dfrac{1}{\alpha}& 0 \\
1 & \dfrac{\alpha ^2 p_{33} +1 -\beta}{\alpha ^2 p_{33}} & 0
\end{pmatrix}} $
\\
&&\\
\hline
\end{tabular}
}

{\footnotesize {TABLE $17^\prime$.}}

\end{center}



We remark that in all the tables ``Table $m'$" the elements $p_{ij}$ have to satisfy the necessary conditions in order for $P_{B'B}$ to have rank 3.

\medskip
\noindent
{\bf Case ${\rm dim}(A^{2})=3$.}
\medskip

\noindent
In order to classify all the possible matrices corresponding to structure matrices of three-dimensional evolution algebras $A$ such that $A^2=A$ (equivalently
${\rm dim}(A^{2})=3$), we will use Proposition \ref{prop:numerodeceros}.
Notice that in this case the number of zeros in all the structure matrices of a given evolution algebra is invariant (see Proposition \ref{prop:numerodeceros} \eqref{numerodeceros1}). Equivalently, the number of non-zero entries is invariant. This is the reason because of which we will classify taking into account this last number. Note that the minimum number of non-zero entries in $M_B$ is exactly three.
\medskip

\noindent
{\bf Case 1.} $M_B$ has three non-zero elements.
\medskip

\noindent
We compute the determinant of $M_B$.
\begin{equation}
\label{determinanteNoNulo}\vert M_{B} \vert =\omega_{11}\omega_{22}\omega_{33}+ \omega_{12}\omega_{23}\omega_{31}+ \omega_{13}\omega_{21}\omega_{32}-\omega_{13}\omega_{22}\omega_{31}-\omega_{21}\omega_{12}\omega_{33}-\omega_{11}\omega_{32}\omega_{23}.
\end{equation}
Since $\vert M_B \vert \neq 0$,  only one of the six summands is non-zero. Assume, for example,
$\omega_{12}\omega_{23}\omega_{31}\neq 0$.
Take
  $\alpha=\frac{1}{\sqrt[7]{\omega_{12}\omega_{23}^2 \omega_{31}^4}}$, $\beta=\alpha^4 \omega_{23}\omega_{31}^2$ and $\gamma=\alpha^2 \omega_{31}$. Then
$$
(\alpha,\beta,\gamma) \cdot
\begin{pmatrix}
0 & \omega_{12} & 0 \\
0  & 0 & \omega_{23} \\
\omega_{31} & 0 & 0
\end{pmatrix}=
\begin{pmatrix}
0 & 1 & 0 \\
0  & 0 & 1 \\
1 & 0 & 0
\end{pmatrix}
$$
\par
Reasoning in this way with   $\omega_{1\sigma(1)}\omega_{2\sigma(2)}\omega_{3\sigma(3)}$ (where $\sigma \in S_3$) instead of with $\omega_{12}\omega_{23}\omega_{31}$, we obtain a natural basis $B'$ such that $M_{B'}= (\varpi_{ij})$, with $\varpi_{i\sigma(i)}=1$ and $\varpi_{ij}=0$ for any $j\neq \sigma(i)$.
 This justifies that these are the only matrices we consider in order to get the classification. Notice that there are are only six. Since we do not know which of them are in the same orbit (considering the action described in Section \ref{Action}), we start with one of them, say $M$, and consider $\tau \cdot  M$ for any $\tau\in S_3$.
 We display $\{\tau \cdot M \ \vert \  \tau \in S_3\}$ in a row. Then, we start with another of these matrices, say $M'$, not appearing in this row, and display $\{\tau \cdot M' \ \vert \  \tau \in S_3\}$ in a second row. We continue in this way until we get the six different matrices. We color these six matrices to make easier the reader to find them.

\begin{center}
\scalebox{0.7}{
\begin{tabular}{|c||c||c||c||c||c|}

\hline
&&&&&\\
&(1,2)&(1,3)&(2,3)&(1,2,3)&(1,3,2)\\
&&&&&\\[-0.2cm]
\hline
\hline
&&&&&\\[-0.2cm]

$\color{ao(english)}{
\begin{pmatrix}
1 & 0 & 0 \\
0 & 1 & 0 \\
0  & 0 & 1
\end{pmatrix}}
$  &

$ \begin{pmatrix}
1 & 0 & 0 \\
0 & 1 & 0 \\
0  & 0 & 1
\end{pmatrix}$
&
$\begin{pmatrix}
1 & 0 & 0 \\
0 & 1 & 0 \\
0  & 0 & 1
\end{pmatrix} $
 &

 $\begin{pmatrix}
1 & 0 & 0 \\
0 & 1 & 0 \\
0  & 0 & 1
\end{pmatrix}$
&
$
\begin{pmatrix}
1 & 0 & 0 \\
0 & 1 & 0 \\
0  & 0 & 1
\end{pmatrix}$
&
$
\begin{pmatrix}
1 & 0 & 0 \\
0 & 1 & 0 \\
0  & 0 & 1
\end{pmatrix}
$ \\
&&&&&\\[-0.2cm]
\hline
\hline
&&&&&\\[-0.2cm]
$\color{ao(english)}{
\begin{pmatrix}
0 & 1 & 0 \\
1 & 0 & 0 \\
0  & 0 & 1

\end{pmatrix}}
$  &

$ \begin{pmatrix}
0 & 1 & 0 \\
1 & 0 & 0 \\
0  & 0 & 1

\end{pmatrix}$
&
$\color{ao(english)}{
\begin{pmatrix}
1 & 0 & 0 \\
0 & 0 & 1 \\
0  & 1 & 0

\end{pmatrix}} $
 &

 $\color{ao(english)}{
 \begin{pmatrix}
0 & 0 & 1 \\
0 & 1 & 0 \\
1  & 0 & 0
\end{pmatrix}}$
&
$
\begin{pmatrix}
1 & 0 & 0 \\
0 & 0 & 1 \\
0  & 1 & 0
\end{pmatrix}$
&
$
\begin{pmatrix}
0 & 0 & 1 \\
0 & 1 & 0 \\
1  & 0 & 0
\end{pmatrix}
$
 \\
&&&&&\\[-0.2cm]
\hline
\hline
&&&&&\\[-0.2cm]
$\color{ao(english)}{
\begin{pmatrix}
0 & 0 & 1 \\
1 & 0 & 0 \\
0  & 1 & 0

\end{pmatrix}}
$  &

$\color{ao(english)}{
 \begin{pmatrix}
0 & 1 & 0 \\
0 & 0 & 1 \\
1  & 0 & 0

\end{pmatrix}}$
&
$\begin{pmatrix}
0 & 1 & 0 \\
0 & 0 & 1 \\
1  & 0 & 0

\end{pmatrix} $
 &

 $\begin{pmatrix}
0 & 1 & 0 \\
0 & 0 & 1 \\
1  & 0 & 0
\end{pmatrix}$
&
$
\begin{pmatrix}
0 & 0 & 1 \\
1 & 0 & 0 \\
0  & 1 & 0
\end{pmatrix}
$
&
$
\begin{pmatrix}
0 & 0 & 1 \\
1 & 0 & 0 \\
0  & 1 & 0
\end{pmatrix}
$
 \\
&&&&&\\
\hline
\end{tabular}}

{\footnotesize {TABLE 18.  ${\rm dim}(A^2)=3$; three non-zero entries.}} \label{tabla15}

\end{center}
\newpage

Therefore, there are only three orbits and, consequently, only three non-isomorphic evolution algebras $A$ in the case we are studying. Their structure matrices are:

\begin{equation}\label{treselem}
\begin{pmatrix}
1 & 0 & 0 \\
0 & 1 & 0 \\
0  & 0 & 1
\end{pmatrix}, \;
\begin{pmatrix}
0 & 1 & 0 \\
1 & 0 & 0 \\
0  & 0 & 1
\end{pmatrix} \; \rm{and} \;
\begin{pmatrix}
0 & 1 & 0 \\
0 & 0 & 1 \\
1  & 0  & 0
\end{pmatrix}.
\end{equation}

\medskip

\noindent
{\bf Case 2.} $M_B$ has four non-zero elements.

\noindent
Reasoning as in Case 1, we arrive at a natural basis $B'$ of the evolution algebra $A$ such that $M_{B'}=(\varpi_{ij})$, with $\varpi_{i\sigma(i)}=1$, $\varpi_{ij}\neq 0$ for some $j\neq \sigma(i)$ and  $\varpi_{ik}=0$ for every $k\neq \sigma(i), j$ for every permutation $\sigma \in S_3$.

In order to describe the matrices producing non-isomorphic evolution algebras, first, we notice the following. Given a matrix as explained below, no matter where we put the four non-zero elements (three 1 and one arbitrary parameter $\mu$ which has to be non-zero) that the resulting matrices correspond to isomorphic evolution algebras. This is because we will not be worried about where to place the parameter.
Then we  explain which are the possible cases.

We have to put five 0 into nine places (the nine entries of the matrix). This can be done in
$\binom{9}{5}= 126$ ways. But we must remove the cases in which $\vert M_{B'}\vert=0$. This happens:
\begin{enumerate}[\rm (a)]
\item When the entries of a row are zero.
\item When the entries of a column are zero but there is no a row which consists of zeros.
\item When the matrix has a $2 \times 2$ minor with every entry equals zero and it has not a row or a column of zeros.
\end{enumerate}
These three cases are mutually exclusive.

(a) The cases in which there is a column of zeros are $3 \binom{6}{2}= 45$ (3 corresponds to the three columns and $\binom{6}{2}$ corresponds to the different ways in which two zeros can be distributed in the six remaining places).

(b) For the rows the reasoning in similar: we have 45 cases. Now we have to take into account that there are cases which have been considered twice (just when there is a row and a column which are zero). This happens 9 times. Therefore, we have 45-9=36 options in this case.

(c) Once the matrix has a $2 \times 2$ minor with every entry equals zero, the fifth zero must be only in one place if we want to avoid the matrix having a row or column of zeros.  There are 9 options to put a zero in a matrix. Once this happens, we remove the corresponding row and the corresponding column and there are four places where to put four zeros. Hence, there are 9 possibilities in this case.

Taking into account (a), (b) and (c), there are $126-(45+36+9)= 36$ different matrices we can consider.

As in  Case 1, we list all the options in a table. The elements that appear in every row correspond to the action of every element of $S_3$ on the matrix placed first. There are six  mutually non-isomorphic parametric families of evolution algebras, which are listed below.

\begin{center}
\scalebox{0.7}{
\begin{tabular}{|c||c||c||c||c||c|}

\hline
&&&&&\\
&(1,2)&(1,3)&(2,3)&(1,2,3)&(1,3,2)\\
&&&&&\\[-0.2cm]
\hline
\hline
&&&&&\\[-0.2cm]
$\color{ao(english)}{
\begin{pmatrix}
1 & \mu & 0 \\
0 & 1 & 0 \\
0  & 0 & 1
\end{pmatrix}}
$  &

$\color{ao(english)}{
 \begin{pmatrix}
1 & 0 & 0 \\
\mu & 1 & 0 \\
0  & 0 & 1

\end{pmatrix}}$
&
$\color{ao(english)}{
\begin{pmatrix}
1 & 0 & 0 \\
0 & 1 & 0 \\
0  & \mu & 1
\end{pmatrix} }$
 &

$\color{ao(english)}{
\begin{pmatrix}
1 & 0 & \mu \\
0 & 1 & 0 \\
0  & 0 & 1
\end{pmatrix}}$
&
$\color{ao(english)}{
\begin{pmatrix}
1 & 0 & 0 \\
0 & 1 & \mu \\
0  & 0 & 1
\end{pmatrix}}$
&
$\color{ao(english)}{
\begin{pmatrix}
1 & 0 & 0 \\
0 & 1 & 0 \\
\mu  & 0 & 1
\end{pmatrix}}
$
\\
&&&&&\\[-0.2cm]
\hline
\hline
&&&&&\\[-0.2cm]
$\color{ao(english)}{
\begin{pmatrix}
\mu & 1 & 0 \\
1 & 0 & 0 \\
0  & 0 & 1

\end{pmatrix}}
$  &

$\color{ao(english)}{
 \begin{pmatrix}
0 & 1 & 0 \\
1 & \mu & 0 \\
0  & 0 & 1

\end{pmatrix}}$
&
$\color{ao(english)}{
\begin{pmatrix}
1 & 0 & 0 \\
0 & 0 & 1 \\
0  & 1 & \mu

\end{pmatrix} }$
 &

$\color{ao(english)}{
\begin{pmatrix}
\mu & 0 & 1 \\
0 & 1 & 0 \\
1  & 0 & 0
\end{pmatrix}}$
&
$\color{ao(english)}{
\begin{pmatrix}
1 & 0 & 0 \\
0 & \mu & 1 \\
0  & 1 & 0
\end{pmatrix}}$
&
$\color{ao(english)}{
\begin{pmatrix}
0 & 0 & 1 \\
0 & 1 & 0 \\
1  & 0 & \mu
\end{pmatrix}}
$ \\
&&&&&\\[-0.2cm]
\hline
\hline
&&&&&\\[-0.2cm]

$\color{ao(english)}{
\begin{pmatrix}
0 & 1 & \mu  \\
1 & 0 & 0 \\
0  & 0 & 1
\end{pmatrix}}
$  &

$\color{ao(english)}{
\begin{pmatrix}
0 & 1 & 0 \\
1 & 0 & \mu \\
0  & 0 & 1
\end{pmatrix}}
$

&

$\color{ao(english)}{
\begin{pmatrix}
1 & 0 & 0 \\
0 & 0 & 1 \\
\mu  & 1 & 0
\end{pmatrix} }$
 &

$\color{ao(english)}{
\begin{pmatrix}
0 & \mu & 1 \\
0 & 1 & 0 \\
1  & 0 & 0
\end{pmatrix}}$
&
$\color{ao(english)}{
\begin{pmatrix}
1 & 0 & 0 \\
\mu & 0 & 1 \\
0  & 1 & 0
\end{pmatrix}}
$
&
$\color{ao(english)}{
\begin{pmatrix}
0 & 0 & 1 \\
0 & 1 & 0 \\
1 & \mu  & 0
\end{pmatrix}}$ \\
&&&&&\\[-0.2cm]
\hline
\hline
&&&&&\\[-0.2cm]

$\color{ao(english)}{
\begin{pmatrix}
0 & 1 & 0 \\
1 & 0 & 0 \\
\mu  & 0 & 1
\end{pmatrix}}
$  &

$\color{ao(english)}{
\begin{pmatrix}
0 & 1 & 0 \\
1 & 0 & 0 \\
0  & \mu & 1
\end{pmatrix}}$
&
$\color{ao(english)}{
\begin{pmatrix}
1 & 0 & \mu \\
0 & 0 & 1 \\
0  & 1 & 0
\end{pmatrix}}$
 &

$\color{ao(english)}{
\begin{pmatrix}
0 & 0 & 1 \\
\mu & 1 & 0 \\
1  & 0 & 0
\end{pmatrix}}$
&
$\color{ao(english)}{
\begin{pmatrix}
1 & \mu & 0 \\
0 & 0 & 1 \\
0  & 1 & 0
\end{pmatrix}}
$
&
$\color{ao(english)}{
\begin{pmatrix}
0 & 0 & 1 \\
0 & 1 & \mu \\
1  & 0 & 0
\end{pmatrix}}
$

\\
&&&&&\\[-0.2cm]
\hline
\hline
&&&&&\\[-0.2cm]

$\color{ao(english)}{
\begin{pmatrix}
\mu & 0 & 1 \\
1 & 0 & 0 \\
0  & 1 & 0

\end{pmatrix}}
$  &

$\color{ao(english)}{
 \begin{pmatrix}
0 & 1 & 0 \\
0 & \mu & 1 \\
1  & 0 & 0

\end{pmatrix}}$
&
$\color{ao(english)}{
\begin{pmatrix}
0 & 1 & 0 \\
0 & 0 & 1 \\
1  & 0 & \mu

\end{pmatrix} }$
 &

$\color{ao(english)}{
\begin{pmatrix}
\mu & 1 & 0 \\
0 & 0 & 1 \\
1  & 0 & 0
\end{pmatrix}}$
&
$\color{ao(english)}{
\begin{pmatrix}
0 & 0 & 1 \\
1 & \mu & 0 \\
0  & 1 & 0
\end{pmatrix}}$
&
$\color{ao(english)}{
\begin{pmatrix}
0 & 0 & 1 \\
1 & 0 & 0 \\
0  & 1 & \mu
\end{pmatrix}}
$ \\
&&&&&\\[-0.2cm]
\hline
\hline
&&&&&\\[-0.2cm]

$\color{ao(english)}{
\begin{pmatrix}
0 & \mu & 1 \\
1 & 0 & 0 \\
0  & 1 & 0

\end{pmatrix}}
$  &

$\color{ao(english)}{
 \begin{pmatrix}
0 & 1 & 0 \\
\mu & 0 & 1 \\
1  & 0 & 0

\end{pmatrix}}$
&
$\color{ao(english)}{
\begin{pmatrix}
0 & 1 & 0 \\
0 & 0 & 1 \\
1  & \mu & 0
\end{pmatrix}}$
 &

$\color{ao(english)}{
\begin{pmatrix}
0 & 1 & \mu \\
0 & 0 & 1 \\
1  & 0 & 0
\end{pmatrix}}$
&
$\color{ao(english)}{
\begin{pmatrix}
0 & 0 & 1 \\
1 & 0 & \mu \\
0  & 1 & 0
\end{pmatrix}}
$
&
$\color{ao(english)}{
\begin{pmatrix}
0 & 0 & 1 \\
1 & 0 & 0 \\
\mu  & 1 & 0
\end{pmatrix}}
$
\\
&&&&&\\
\hline
\end{tabular}}

{\footnotesize {TABLE 19.  ${\rm dim}(A^2)=3$; four non-zero entries.}}

\end{center}

\begin{center}

\scalebox{0.7}{

\begin{tabular}{|c||c||c|}

\hline
&&\\
$M_{B}$& $P_{B'B}$ & $M_{B'}$\\
&&\\[-0.2cm]
\hline
\hline
&&\\[-0.2cm]

${
\begin{pmatrix}
\mu & 0 & 1 \\
1 & 0 & 0 \\
0 & 1 & 0
\end{pmatrix}}
$  &
${
 \begin{pmatrix}
-{\phi} & 0 & 0 \\
0 & {\phi}^2 & 0 \\
0 & 0 & {\phi}^4
\end{pmatrix}}$
&
${
\begin{pmatrix}
-{\phi} \mu & 0 & 1 \\
1 & 0 & 0 \\
0 & 1 & 0
\end{pmatrix}} $
\\

&&\\[-0.2cm]
\hline
\hline
&&\\[-0.2cm]
${
\begin{pmatrix}
\mu & 0 & 1 \\
1 & 0 & 0 \\
0 & 1 & 0
\end{pmatrix}}
$  &
${
 \begin{pmatrix}
{\phi}^2 & 0 & 0 \\
0 & {\phi}^4 & 0 \\
0 & 0 & -{\phi}
\end{pmatrix}}$
&
${
\begin{pmatrix}
{\phi}^2 \mu & 0 & 1 \\
1 & 0 & 0 \\
0 & 1 & 0
\end{pmatrix}} $
\\
&&\\[-0.2cm]
\hline
\hline
&&\\[-0.2cm]

${
\begin{pmatrix}
\mu & 0 & 1 \\
1 & 0 & 0 \\
0 & 1 & 0
\end{pmatrix}}
$  &
${
 \begin{pmatrix}
-{\phi}^3 & 0 & 0 \\
0 & {\phi}^6 & 0 \\
0 & 0 & -{\phi}^5
\end{pmatrix}}$
&
${
\begin{pmatrix}
-{\phi}^3 \mu & 0 & 1 \\
1 & 0 & 0 \\
0 & 1 & 0
\end{pmatrix}} $
\\
&&\\[-0.2cm]
\hline
\hline
&&\\[-0.2cm]

${
\begin{pmatrix}
\mu & 0 & 1 \\
1 & 0 & 0 \\
0 & 1 & 0
\end{pmatrix}}
$  &
${
 \begin{pmatrix}
{\phi}^4 & 0 & 0 \\
0 & -{\phi} & 0 \\
0 & 0 & {\phi}^2
\end{pmatrix}}$
&
${
\begin{pmatrix}
{\phi}^4 \mu & 0 & 1 \\
1 & 0 & 0 \\
0 & 1 & 0
\end{pmatrix}} $
\\

&&\\[-0.2cm]
\hline
\hline
&&\\[-0.2cm]

${
\begin{pmatrix}
\mu & 0 & 1 \\
1 & 0 & 0 \\
0 & 1 & 0
\end{pmatrix}}
$  &
${
 \begin{pmatrix}
-{\phi}^5 & 0 & 0 \\
0 & -{\phi}^3 & 0 \\
0 & 0 & {\phi}^6
\end{pmatrix}}$
&
${
\begin{pmatrix}
-{\phi}^5 \mu & 0 & 1 \\
1 & 0 & 0 \\
0 & 1 & 0
\end{pmatrix}} $
\\

&&\\[-0.2cm]
\hline
\hline
&&\\[-0.2cm]

${
\begin{pmatrix}
\mu & 0 & 1 \\
1 & 0 & 0 \\
0 & 1 & 0
\end{pmatrix}}
$  &
${
 \begin{pmatrix}
{\phi}^6 & 0 & 0 \\
0 & -{\phi}^5 & 0 \\
0 & 0 & -{\phi}^3
\end{pmatrix}}$
&
${
\begin{pmatrix}
{\phi}^6 \mu & 0 & 1 \\
1 & 0 & 0 \\
0 & 1 & 0
\end{pmatrix}} $
\\
&&\\[-0.2cm]
\hline
\hline
&&\\[-0.2cm]

${
\begin{pmatrix}
0 & \mu & 1 \\
1 & 0 & 0 \\
0 & 1 & 0
\end{pmatrix}}
$  &
${
 \begin{pmatrix}
-{\phi} & 0 & 0 \\
0 & {\phi}^2 & 0 \\
0 & 0 & {\phi}^4
\end{pmatrix}}$
&
${
\begin{pmatrix}
0 & -{\phi}^3 \mu & 1 \\
1 & 0 & 0 \\
0 & 1 & 0
\end{pmatrix}} $
\\

&&\\[-0.2cm]
\hline
\hline
&&\\[-0.2cm]
${
\begin{pmatrix}
0 & \mu & 1 \\
1 & 0 & 0 \\
0 & 1 & 0
\end{pmatrix}}
$  &
${
 \begin{pmatrix}
{\phi}^2 & 0 & 0 \\
0 & {\phi}^4 & 0 \\
0 & 0 & -{\phi}
\end{pmatrix}}$
&
${
\begin{pmatrix}
0 & {\phi}^6 \mu  & 1 \\
1 & 0 & 0 \\
0 & 1 & 0
\end{pmatrix}} $
\\
&&\\[-0.2cm]
\hline
\hline
&&\\[-0.2cm]

${
\begin{pmatrix}
0 & \mu  & 1 \\
1 & 0 & 0 \\
0 & 1 & 0
\end{pmatrix}}
$  &
${
 \begin{pmatrix}
-{\phi}^3 & 0 & 0 \\
0 & {\phi}^6 & 0 \\
0 & 0 & -{\phi}^5
\end{pmatrix}}$
&
${
\begin{pmatrix}
0 {\phi}^2 \mu & 1 \\
1 & 0 & 0 \\
0 & 1 & 0
\end{pmatrix}} $
\\
&&\\[-0.2cm]
\hline
\hline
&&\\[-0.2cm]

${
\begin{pmatrix}
0 & \mu  & 1 \\
1 & 0 & 0 \\
0 & 1 & 0
\end{pmatrix}}
$  &
${
 \begin{pmatrix}
{\phi}^4 & 0 & 0 \\
0 & -{\phi} & 0 \\
0 & 0 & {\phi}^2
\end{pmatrix}}$
&
${
\begin{pmatrix}
0 & -{\phi}^5 \mu  & 1 \\
1 & 0 & 0 \\
0 & 1 & 0
\end{pmatrix}} $
\\

&&\\[-0.2cm]
\hline
\hline
&&\\[-0.2cm]

${
\begin{pmatrix}
0 & \mu  & 1 \\
1 & 0 & 0 \\
0 & 1 & 0
\end{pmatrix}}
$  &
${
 \begin{pmatrix}
-{\phi}^5 & 0 & 0 \\
0 & -{\phi}^3 & 0 \\
0 & 0 & {\phi}^6
\end{pmatrix}}$
&
${
\begin{pmatrix}
0 & -{\phi} \mu & 1 \\
1 & 0 & 0 \\
0 & 1 & 0
\end{pmatrix}} $
\\

&&\\[-0.2cm]
\hline
\hline
&&\\[-0.2cm]

${
\begin{pmatrix}
0 & \mu & 1 \\
1 & 0 & 0 \\
0 & 1 & 0
\end{pmatrix}}
$  &
${
 \begin{pmatrix}
{\phi}^6 & 0 & 0 \\
0 & -{\phi}^5 & 0 \\
0 & 0 & -{\phi}^3
\end{pmatrix}}$
&
${
\begin{pmatrix}
0 & {\phi}^4 \mu & 1 \\
1 & 0 & 0 \\
0 & 1 & 0
\end{pmatrix}} $
\\
&&\\
\hline
\end{tabular}
}

\end{center}

\begin{center}

\scalebox{0.7}{

\begin{tabular}{|c||c||c|}

\hline
&&\\
$M_{B}$& $P_{B'B}$ & $M_{B'}$\\
&&\\[-0.2cm]
\hline
\hline
&&\\[-0.2cm]

${
\begin{pmatrix}
\mu & 1 & 0 \\
1 & 0 & 0 \\
0 & 0 & 1
\end{pmatrix}}
$  &
${
 \begin{pmatrix}
-\dfrac{1}{2}(1+\sqrt{-3}) & 0 & 0 \\
0 & \dfrac{1}{2}(-1+\sqrt{-3}) & 0 \\
0 & 0 & 1
\end{pmatrix}}$
&
${
\begin{pmatrix}
-\dfrac{1}{2}(1+\sqrt{-3}) \mu & 1 & 0 \\
1 & 0 & 0 \\
0 & 0 & 1
\end{pmatrix}} $
\\

&&\\[-0.2cm]
\hline
\hline
&&\\[-0.2cm]
${
\begin{pmatrix}
\mu & 1 & 0 \\
1 & 0 & 0 \\
0 & 0 & 1
\end{pmatrix}}
$  &
${
 \begin{pmatrix}
\dfrac{1}{2}(-1+\sqrt{-3}) & 0 & 0 \\
0 & -\dfrac{1}{2}(1+\sqrt{-3}) & 0 \\
0 & 0 & 1
\end{pmatrix}}$
&
${
\begin{pmatrix}
\dfrac{1}{2}(-1+\sqrt{-3}) \mu & 1 & 0 \\
1 & 0 & 0 \\
0 & 0 & 1
\end{pmatrix}} $
\\
&&\\[-0.2cm]
\hline
\hline
&&\\[-0.2cm]

${
\begin{pmatrix}
0 & 1 & \mu  \\
1 & 0 & 0 \\
0 & 0 & 1
\end{pmatrix}}
$  &
${
 \begin{pmatrix}
-\dfrac{1}{2}(1+\sqrt{-3}) & 0 & 0 \\
0 & \dfrac{1}{2}(-1+\sqrt{-3}) & 0 \\
0 & 0 & 1
\end{pmatrix}}$
&
${
\begin{pmatrix}
0 & 1 & \dfrac{1}{2}(-1+\sqrt{-3}) \mu  \\
1 & 0 & 0 \\
0 & 0 & 1
\end{pmatrix}} $
\\

&&\\[-0.2cm]
\hline
\hline
&&\\[-0.2cm]
${
\begin{pmatrix}
0 & 1 &\mu  \\
1 & 0 & 0 \\
0 & 0 & 1
\end{pmatrix}}
$  &
${
 \begin{pmatrix}
\dfrac{1}{2}(-1+\sqrt{-3}) & 0 & 0 \\
0 & -\dfrac{1}{2}(1+\sqrt{-3}) & 0 \\
0 & 0 & 1
\end{pmatrix}}$
&
${
\begin{pmatrix}
0 & 1 & -\dfrac{1}{2}(1+\sqrt{-3}) \mu  \\
1 & 0 & 0 \\
0 & 0 & 1
\end{pmatrix}} $
\\
&&\\[-0.2cm]
\hline
\hline
&&\\[-0.2cm]
${
\begin{pmatrix}
0 & 1 & 0 \\
1 & 0 & 0 \\
\mu & 0 & 1
\end{pmatrix}}
$  &
${
 \begin{pmatrix}
-\dfrac{1}{2}(1+\sqrt{-3}) & 0 & 0 \\
0 & \dfrac{1}{2}(-1+\sqrt{-3}) & 0 \\
0 & 0 & 1
\end{pmatrix}}$
&
${
\begin{pmatrix}
0 & 1 & 0  \\
1 & 0 & 0 \\
\dfrac{1}{2}(-1+\sqrt{-3}) \mu & 0 & 1
\end{pmatrix}} $
\\
&&\\[-0.2cm]
\hline
\hline
&&\\[-0.2cm]
${
\begin{pmatrix}
0 & 1 & 0 \\
1 & 0 & 0 \\
\mu  & 0 & 1
\end{pmatrix}}
$  &
${
 \begin{pmatrix}
\dfrac{1}{2}(-1+\sqrt{-3}) & 0 & 0 \\
0 & -\dfrac{1}{2}(1+\sqrt{-3}) & 0 \\
0 & 0 & 1
\end{pmatrix}}$
&
${
\begin{pmatrix}
0 & 1 & 0 \\
1 & 0 & 0 \\
-\dfrac{1}{2}(1+\sqrt{-3}) \mu  & 0 & 1
\end{pmatrix}} $
\\
&&\\
\hline
\end{tabular}
}

{\footnotesize {TABLE $19^\prime$.}}

\end{center}


 \medskip

\noindent
{\bf Case 3.} $M_B$ has five non-zero elements.
\medskip

\noindent
We proceed as in the cases above and obtain that in order to classify we need to consider only matrices with four zero entries and five non-zero entries. By changing the basis, we may assume that three of the elements are 1 and the other are arbitrary parameters $\lambda$ and $\mu$, with the only restriction of being non-zero and such that $\lambda\mu\neq 1$ (this condition is needed because the determinant must be non-zero).

The different matrices to be considered are those for which we place four zeros: $\binom{9}{4}= 126$. On the one hand, we must remove those for which there is a row or a column which are zero (because these matrices have zero determinant). If one row or column consists of zeros, then the fourth zero can be placed in six different positions. Since there are 3 rows and 3 columns, this happens 6 times. On the other hand, we must remove those for which there is a $2 \times 2$ minor with every entry equals zero. Consequently, we have  $126-6^2-9 = 81$ cases that we display in the table that follows. The number of mutually non-isomorphic parametric families of evolution algebras is sixteen. We show them in two tables.

\begin{center}
\scalebox{0.7}{
\begin{tabular}{|c||c||c||c||c||c|}
\hline
&&&&&\\
&(1,2)&(1,3)&(2,3)&(1,2,3)&(1,3,2)\\
&&&&&\\[-0.2cm]
\hline
\hline
&&&&&\\[-0.2cm]

$\color{ao(english)}{
\begin{pmatrix}
1 & \mu & \lambda \\
0 & 1 & 0 \\
0  & 0 & 1

\end{pmatrix}}
$  &

$\color{ao(english)}{
 \begin{pmatrix}
1 & 0 & 0 \\
\mu & 1 & \lambda \\
0  & 0 & 1

\end{pmatrix}}$
&
$\color{ao(english)}{
\begin{pmatrix}
1 & 0 & 0 \\
0 & 1 & 0 \\
\lambda  & \mu & 1

\end{pmatrix} }$
 &

 $
 {}{
 \begin{pmatrix}
1 & \lambda & \mu \\
0 & 1 & 0 \\
0  & 0 & 1
\end{pmatrix}}$
&
$
{}{
\begin{pmatrix}
1 & 0 & 0 \\
\lambda & 1 & \mu  \\
0  & 0 & 1
\end{pmatrix}}$
&
$
{}{
\begin{pmatrix}
1 & 0 & 0 \\
0 & 1 & 0 \\
\mu  & \lambda & 1
\end{pmatrix}}
$ \\
&&&&&\\[-0.2cm]
\hline
\hline
&&&&&\\[-0.2cm]

$\color{ao(english)}{
\begin{pmatrix}
1 & \mu & 0  \\
\lambda & 1 & 0 \\
0  & 0 & 1

\end{pmatrix}}
$  &

$
{}{
 \begin{pmatrix}
1 & \lambda & 0 \\
\mu & 1 & 0 \\
0  & 0 & 1

\end{pmatrix}}$
&

$\color{ao(english)}{
\begin{pmatrix}
1 & 0 & 0 \\
0 & 1 & \lambda \\
0  & \mu & 1
\end{pmatrix} }$
 &

 $\color{ao(english)}{
 \begin{pmatrix}
1 & 0 & \mu \\
0 & 1 & 0 \\
\lambda   & 0 & 1
\end{pmatrix}}$
&
$
{}{
\begin{pmatrix}
1 & 0 & 0 \\
0 & 1 & \mu \\
0  & \lambda & 1
\end{pmatrix}}
$
&
$
{}{
\begin{pmatrix}
1 & 0 & \lambda \\
0 & 1 & 0 \\
\mu & 0  & 1
\end{pmatrix}}
$ \\
\hline
\end{tabular}}

\end{center}

\medskip

\begin{center}

\scalebox{0.70}{

\begin{tabular}{|c||c||c||c||c||c|}

\hline
&&&&&\\
&(1,2)&(1,3)&(2,3)&(1,2,3)&(1,3,2)\\
&&&&&\\[-0.2cm]
\hline
\hline
&&&&&\\[-0.2cm]

$\color{ao(english)}{
\begin{pmatrix}
0 & 1 & 0 \\
1 & 0 & \mu\\
\lambda  & 0 & 1

\end{pmatrix}}
$  &

$\color{ao(english)}{
 \begin{pmatrix}
0 & 1 & \mu \\
1 & 0 & 0 \\
0  & \lambda & 1

\end{pmatrix}}$
&
$\color{ao(english)}{
\begin{pmatrix}
1 & 0 & \lambda \\
\mu & 0 & 1 \\
0  & 1 & 0

\end{pmatrix} }$
 &

 $\color{ao(english)}{
 \begin{pmatrix}
0 & 0 & 1 \\
\lambda & 1 & 0 \\
1  & \mu & 0
\end{pmatrix}}$
&
$\color{ao(english)}{
\begin{pmatrix}
1 & \lambda & 0 \\
0 & 0 & 1 \\
\mu  & 1 & 0
\end{pmatrix}}
$
&
$\color{ao(english)}{
\begin{pmatrix}
0 & \mu & 1 \\
0 & 1 & \lambda \\
1  & 0 & 0
\end{pmatrix}}
$ \\
&&&&&\\[-0.2cm]
\hline
\hline
&&&&&\\[-0.2cm]

$\color{ao(english)}{
\begin{pmatrix}
0 & 1 & 0 \\
1 & 0 & \mu \\
0  & \lambda & 1

\end{pmatrix}}
$  &

$\color{ao(english)}{
 \begin{pmatrix}
0 & 1 & \mu \\
1 & 0 & 0 \\
\lambda  & 0 & 1

\end{pmatrix}}$
&
$\color{ao(english)}{
\begin{pmatrix}
1 & \lambda & 0 \\
\mu & 0 & 1 \\
0  & 1 & 0

\end{pmatrix} }$
 &

 $\color{ao(english)}{
 \begin{pmatrix}
0 & 0& 1 \\
0 & 1 & \lambda \\
1  & \mu & 0
\end{pmatrix}}$
&
$\color{ao(english)}{
\begin{pmatrix}
1 & 0 & \lambda \\
0 & 0 & 1 \\
\mu  & 1 & 0
\end{pmatrix}}
$
&
$\color{ao(english)}{
\begin{pmatrix}
0 & \mu & 1 \\
\lambda & 1 & 0 \\
1  & 0 & 0
\end{pmatrix}}
$ \\
\hline
\end{tabular}}

{\footnotesize {TABLE 20. ${\rm dim}(A^2)=3$; five non-zero entries.}} \label{tabla18}

\end{center}
\medskip

\begin{center}

\scalebox{0.7}{

\begin{tabular}{|c||c||c|}

\hline
&&\\
$M_{B}$& $P_{B'B}$ & $M_{B'}$\\
&&\\[-0.2cm]
\hline
\hline
&&\\[-0.2cm]

${
\begin{pmatrix}
\mu & 1 & \lambda \\
1 & 0 & 0 \\
0 & 0 & 1
\end{pmatrix}}
$  &
${
 \begin{pmatrix}
-\dfrac{1}{2}(1+\sqrt{-3}) & 0 & 0 \\
0 & \dfrac{1}{2}(-1+\sqrt{-3}) & 0 \\
0 & 0 & 1
\end{pmatrix}}$
&
${
\begin{pmatrix}
-\dfrac{1}{2}(1+\sqrt{-3}) \mu & 1 & \dfrac{1}{2}(-1+\sqrt{-3})  \lambda \\
1 & 0 & 0 \\
0 & 0 & 1
\end{pmatrix}} $
\\

&&\\[-0.2cm]
\hline
\hline
&&\\[-0.2cm]
${
\begin{pmatrix}
\mu & 1 & \lambda \\
1 & 0 & 0 \\
0 & 0 & 1
\end{pmatrix}}
$  &
${
 \begin{pmatrix}
\dfrac{1}{2}(-1+\sqrt{-3}) & 0 & 0 \\
0 & -\dfrac{1}{2}(1+\sqrt{-3}) & 0 \\
0 & 0 & 1
\end{pmatrix}}$
&
${
\begin{pmatrix}
\dfrac{1}{2}(-1+\sqrt{-3}) \mu & 1 & -\dfrac{1}{2}(1+\sqrt{-3}) \lambda \\
1 & 0 & 0 \\
0 & 0 & 1
\end{pmatrix}} $
\\
\hline
\end{tabular}
}

\end{center}

\begin{center}

\scalebox{0.7}{

\begin{tabular}{|c||c||c|}

\hline
&&\\
$M_{B}$& $P_{B'B}$ & $M_{B'}$\\
&&\\[-0.2cm]
\hline
\hline
&&\\[-0.2cm]

${
\begin{pmatrix}
\mu & 1 & 0 \\
1 & 0 & 0 \\
0 & \lambda & 1
\end{pmatrix}}
$  &
${
 \begin{pmatrix}
-\dfrac{1}{2}(1+\sqrt{-3}) & 0 & 0 \\
0 & \dfrac{1}{2}(-1+\sqrt{-3}) & 0 \\
0 & 0 & 1
\end{pmatrix}}$
&
${
\begin{pmatrix}
-\dfrac{1}{2}(1+\sqrt{-3}) \mu & 1 & 0  \\
1 & 0 & 0 \\
0 &  -\dfrac{1}{2}(1+\sqrt{-3}) \lambda & 1
\end{pmatrix}} $
\\

&&\\[-0.2cm]
\hline
\hline
&&\\[-0.2cm]
${
\begin{pmatrix}
\mu & 1 & 0  \\
1 & 0 & 0 \\
0 &  \lambda & 1
\end{pmatrix}}
$  &
${
 \begin{pmatrix}
\dfrac{1}{2}(-1+\sqrt{-3}) & 0 & 0 \\
0 & -\dfrac{1}{2}(1+\sqrt{-3}) & 0 \\
0 & 0 & 1
\end{pmatrix}}$
&
${
\begin{pmatrix}
\dfrac{1}{2}(-1+\sqrt{-3}) \mu & 1 & -\dfrac{1}{2}(1+\sqrt{-3}) \lambda \\
1 & 0 & 0 \\
0 & 0 & 1
\end{pmatrix}} $
\\
\hline
\end{tabular}
}

\end{center}

\begin{center}

\scalebox{0.7}{

\begin{tabular}{|c||c||c|}

\hline
&&\\
$M_{B}$& $P_{B'B}$ & $M_{B'}$\\
&&\\[-0.2cm]
\hline
\hline
&&\\[-0.2cm]

${
\begin{pmatrix}
0 & 1 & \mu  \\
1 & 0 & \lambda \\
0 & 0 & 1
\end{pmatrix}}
$  &
${
 \begin{pmatrix}
0 & -\dfrac{1}{2}(1+\sqrt{-3}) &  0 \\
 \dfrac{1}{2}(-1+\sqrt{-3}) & 0 & 0  \\
0 & 0 & 1
\end{pmatrix}}$
&
${
\begin{pmatrix}
0 & 1 & -\dfrac{1}{2}(1+\sqrt{-3}) \lambda  \\
1 & 0 & \dfrac{1}{2}(-1+\sqrt{-3}) \mu \\
0 & 0  & 1
\end{pmatrix}} $
\\

&&\\[-0.2cm]
\hline
\hline
&&\\[-0.2cm]
${
\begin{pmatrix}
0 & 1 & \mu  \\
1 & 0 & \lambda \\
0 &  0 & 1
\end{pmatrix}}
$  &
${
 \begin{pmatrix}
0 & \dfrac{1}{2}(-1+\sqrt{-3})  & 0 \\
-\dfrac{1}{2}(1+\sqrt{-3}) & 0 & 0 \\
0 & 0 & 1
\end{pmatrix}}$
&
${
\begin{pmatrix}
0 & 1 & \dfrac{1}{2}(-1+\sqrt{-3}) \lambda  \\
1 & 0 &  -\dfrac{1}{2}(1+\sqrt{-3}) \mu \\
0 & 0 & 1
\end{pmatrix}} $
\\
\hline
\end{tabular}
}

\end{center}

\begin{center}

\scalebox{0.7}{

\begin{tabular}{|c||c||c|}

\hline
&&\\
$M_{B}$& $P_{B'B}$ & $M_{B'}$\\
&&\\[-0.2cm]
\hline
\hline
&&\\[-0.2cm]

${
\begin{pmatrix}
0 & 1 & 0  \\
1 & 0 & \mu \\
0 & \lambda & 1
\end{pmatrix}}
$  &
${
 \begin{pmatrix}
 -\dfrac{1}{2}(1+\sqrt{-3}) &  0  & 0 \\
 0 & \dfrac{1}{2}(-1+\sqrt{-3}) & 0 \\
0 & 0 & 1
\end{pmatrix}}$
&
${
\begin{pmatrix}
0 & 1 & 0 \\
1 & 0 & -\dfrac{1}{2}(1+\sqrt{-3}) \mu  \\
0 & -\dfrac{1}{2}(1+\sqrt{-3}) \lambda  & 1 \\
\end{pmatrix}} $
\\

&&\\[-0.2cm]
\hline
\hline
&&\\[-0.2cm]
${
\begin{pmatrix}
0 & 1 & 0  \\
1 & 0 & \mu \\
0 &  \lambda & 1
\end{pmatrix}}
$  &
${
 \begin{pmatrix}
 \dfrac{1}{2}(-1+\sqrt{-3})  & 0 & 0 \\
0 & -\dfrac{1}{2}(1+\sqrt{-3})  & 0 \\
0 & 0 & 1
\end{pmatrix}}$
&
${
\begin{pmatrix}
0 & 1 & 0 \\
1 & 0 & \dfrac{1}{2}(-1+\sqrt{-3}) \mu  \\
0 &  \dfrac{1}{2}(-1+\sqrt{-3}) \lambda & 1 \\
\end{pmatrix}} $
\\
\hline
\end{tabular}
}

\end{center}

\begin{center}

\scalebox{0.7}{

\begin{tabular}{|c||c||c|}

\hline
&&\\
$M_{B}$& $P_{B'B}$ & $M_{B'}$\\
&&\\[-0.2cm]
\hline
\hline
&&\\[-0.2cm]

${
\begin{pmatrix}
\mu & \lambda & 1 \\
1 & 0 & 0 \\
0 & 1 & 0
\end{pmatrix}}
$  &
${
 \begin{pmatrix}
-{\phi} & 0 & 0 \\
0 & {\phi}^2 & 0 \\
0 & 0 & {\phi}^4
\end{pmatrix}}$
&
${
\begin{pmatrix}
 -{\phi} \mu & -{\phi}^3 \lambda & 1  \\
1 & 0 & 0 \\
0 & 1 & 0
\end{pmatrix}} $
\\

&&\\[-0.2cm]
\hline
\hline
&&\\[-0.2cm]
${
\begin{pmatrix}
\mu & \lambda & 1 \\
1 & 0 & 0 \\
0 & 1 & 0
\end{pmatrix}}
$  &
${
 \begin{pmatrix}
{\phi}^2 & 0 & 0 \\
0 & {\phi}^4 & 0 \\
0 & 0 & -{\phi}
\end{pmatrix}}$
&
${
\begin{pmatrix}
 {\phi}^2 \mu  & {\phi}^6 \lambda & 1 \\
1 & 0 & 0 \\
0 & 1 & 0
\end{pmatrix}} $
\\
\hline
\end{tabular}
}

\end{center}

\begin{center}

\scalebox{0.7}{

\begin{tabular}{|c||c||c|}

\hline
&&\\
$M_{B}$& $P_{B'B}$ & $M_{B'}$\\
&&\\[-0.2cm]
\hline
\hline
&&\\[-0.2cm]

${
\begin{pmatrix}
\mu & 0 &  1 \\
1 & 0 & 0 \\
\lambda & 1 & 0
\end{pmatrix}}
$  &
${
 \begin{pmatrix}
-{\phi} & 0 & 0 \\
0 & {\phi}^2 & 0 \\
0 & 0 & {\phi}^4
\end{pmatrix}}$
&
${
\begin{pmatrix}
 -{\phi} \mu & 0 & 1   \\
1 & 0 & 0 \\
-{\phi}^5 \lambda  & 1 & 0
\end{pmatrix}} $
\\

&&\\[-0.2cm]
\hline
\hline
&&\\[-0.2cm]
${
\begin{pmatrix}
\mu & 0 & 1 \\
1 & 0 & 0 \\
\lambda & 1 & 0
\end{pmatrix}}
$  &
${
 \begin{pmatrix}
{\phi}^2 & 0 & 0 \\
0 & {\phi}^4 & 0 \\
0 & 0 & -{\phi}
\end{pmatrix}}$
&
${
\begin{pmatrix}
 {\phi}^2 \mu  & 0 & 1 \\
1 & 0 & 0 \\
-{\phi}^3 \lambda & 1 & 0
\end{pmatrix}} $
\\
\hline
\end{tabular}
}

{\footnotesize {TABLE $20^\prime$.}}
\end{center}
\newpage

\noindent
{\bf Case 4.} $M_B$ has six non-zero elements.

Once again we reason in the same way and we can fix our attention in those matrices with three zeros and six non-zero entries.

The different possibilities are: $\binom{9}{3}-6= 78$. Note that $\binom{9}{3}$ are the different ways of placing 3 zeros in a $3 \times 3$ matrix while $6$ corresponds to the cases in which there is a row or a column which is zero.

Making changes on the elements of the basis we may consider three entries equals 1. The only restrictions on the other three elements, say $\lambda, \mu$ and $\rho$, which must be non-zero, are the needed ones in order to not have zero determinant. This means $\mu\rho \neq 1$, $\lambda\rho \neq 1$, $\mu \lambda \neq 1$ and $\mu\rho\lambda \neq -1$.
\medskip

There are fifteen mutually non-isomorphic parametric families of  evolution algebras, which are listed in the table follows.
\bigskip

\begin{center}

\scalebox{0.70}{

\begin{tabular}{|c||c||c||c||c||c|}
\hline
&&&&&\\
&(1,2)&(1,3)&(2,3)&(1,2,3)&(1,3,2)\\
&&&&&\\[-0.2cm]
\hline
\hline
&&&&&\\[-0.2cm]

$\color{ao(english)}{
\begin{pmatrix}
1 & \mu & \lambda \\
\rho & 1 & 0 \\
0  & 0 & 1

\end{pmatrix}}
$  &

$\color{ao(english)}{
 \begin{pmatrix}
1 & \rho & 0 \\
\mu & 1 & \lambda \\
0  & 0 & 1

\end{pmatrix}}$
&
$\color{ao(english)}{
\begin{pmatrix}
1 & 0 & 0 \\
0 & 1 & \rho \\
\lambda  & \mu & 1

\end{pmatrix} }$
 &

$\color{ao(english)}{
\begin{pmatrix}
1 & \lambda & \mu \\
0 & 1 & 0 \\
\rho  & 0 & 1
\end{pmatrix}}$
&
$\color{ao(english)}{
\begin{pmatrix}
1 & 0 & 0 \\
\lambda & 1 & \mu  \\
0  & \rho & 1
\end{pmatrix}}$
&
$\color{ao(english)}{
\begin{pmatrix}
1 & 0 & \rho \\
0 & 1 & 0 \\
\mu  & \lambda & 1
\end{pmatrix}}
$ \\
&&&&&\\[-0.2cm]
\hline
\hline
&&&&&\\[-0.2cm]

$\color{ao(english)}{
\begin{pmatrix}
1 & \mu & \lambda  \\
0 & 1 & \rho \\
0  & 0 & 1

\end{pmatrix}}
$  &

$\color{ao(english)}{
 \begin{pmatrix}
1 & 0 & \rho  \\
\mu & 1 & \lambda \\
0  & 0 & 1

\end{pmatrix}}$
&
$\color{ao(english)}{
\begin{pmatrix}
1 & 0 & 0 \\
\rho & 1 & 0\\
\lambda  & \mu & 1

\end{pmatrix} }$
 &

$\color{ao(english)}{
\begin{pmatrix}
1 & \lambda & \mu \\
0 & 1 & 0 \\
0   & \rho & 1
\end{pmatrix}}$
&
$\color{ao(english)}{
\begin{pmatrix}
1 & 0 & 0 \\
\lambda & 1 & \mu \\
\rho  & 0 & 1
\end{pmatrix}}
$
&
$\color{ao(english)}{
\begin{pmatrix}
1 & \rho & 0 \\
0 & 1 & 0 \\
\mu & \lambda  & 1
\end{pmatrix}}
$ \\
\hline
\end{tabular}}

\end{center}
\newpage

\begin{center}

\scalebox{0.70}{

\begin{tabular}{|c||c||c||c||c||c|}

\hline
&&&&&\\
&(1,2)&(1,3)&(2,3)&(1,2,3)&(1,3,2)\\
&&&&&\\[-0.2cm]
\hline
\hline
&&&&&\\[-0.2cm]

$\color{ao(english)}{
\begin{pmatrix}
0 & 1 & \mu \\
1 & \lambda & 0\\
0 & \rho & 1

\end{pmatrix}}
$  &

$\color{ao(english)}{
 \begin{pmatrix}
\lambda & 1 & 0 \\
1 & 0 & \mu \\
\rho  & 0 & 1

\end{pmatrix}}$
&
$\color{ao(english)}{
\begin{pmatrix}
1 & \rho & 0 \\
0 & \lambda & 1 \\
\mu  & 1 & 0

\end{pmatrix} }$
 &

$\color{ao(english)}{
\begin{pmatrix}
0 & \mu & 1 \\
0 & 1 & \rho \\
1  & 0 & \lambda
\end{pmatrix}}$
&
$\color{ao(english)}{
\begin{pmatrix}
1 & 0 & \rho \\
\mu & 0 & 1 \\
0  & 1 & \lambda
\end{pmatrix}}
$
&
$\color{ao(english)}{
\begin{pmatrix}
\lambda & 0 & 1 \\
\rho & 1 & 0 \\
1  & \mu & 0
\end{pmatrix}}
$ \\
&&&&&\\[-0.2cm]
\hline
\hline
&&&&&\\[-0.2cm]

$\color{ao(english)}{
\begin{pmatrix}
0 & 1 & \mu \\
1 & 0 & \lambda \\
\rho  & 0 & 1

\end{pmatrix}}
$  &

$\color{ao(english)}{
 \begin{pmatrix}
0 & 1 & \lambda \\
1 & 0 & \mu \\
0  & \rho & 1

\end{pmatrix}}$
&
$\color{ao(english)}{
\begin{pmatrix}
1 & 0 & \rho \\
\lambda & 0 & 1 \\
\mu  & 1 & 0

\end{pmatrix} }$
 &

$\color{ao(english)}{
\begin{pmatrix}
0 &  \mu & 1 \\
\rho & 1 & 0 \\
1  & \lambda & 0
\end{pmatrix}}$
&
$\color{ao(english)}{
\begin{pmatrix}
1 & \rho & 0 \\
\mu & 0 & 1 \\
\lambda  & 1 & 0
\end{pmatrix}}
$
&
$\color{ao(english)}{
\begin{pmatrix}
0 & \lambda & 1 \\
0 & 1 & \rho \\
1  & \mu & 0
\end{pmatrix}}
$ \\
\hline

\end{tabular}}

{\footnotesize {TABLE 21. ${\rm dim}(A^2)=3$; six non-zero entries.}} \label{tabla20}
\end{center}
\medskip

\begin{center}

\scalebox{0.7}{

\begin{tabular}{|c||c||c|}

\hline
&&\\
$M_{B}$& $P_{B'B}$ & $M_{B'}$\\
&&\\[-0.2cm]
\hline
\hline
&&\\[-0.2cm]

${
\begin{pmatrix}
\mu & 1 & \lambda  \\
1 & 0 & \rho \\
0 & 0 & 1
\end{pmatrix}}
$  &
${
 \begin{pmatrix}
 -\dfrac{1}{2}(1+\sqrt{-3}) &  0  & 0 \\
 0 & \dfrac{1}{2}(-1+\sqrt{-3}) & 0 \\
0 & 0 & 1
\end{pmatrix}}$
&
${
\begin{pmatrix}
-\dfrac{1}{2}(1+\sqrt{-3}) \mu & 1 & \dfrac{1}{2}(-1+\sqrt{-3}) \lambda \\
1 & 0 & -\dfrac{1}{2}(1+\sqrt{-3}) \rho  \\
0 & 0 & 1 \\
\end{pmatrix}} $
\\

&&\\[-0.2cm]
\hline
\hline
&&\\[-0.2cm]
${
\begin{pmatrix}
\mu & 1 & \lambda  \\
1 & 0 & \rho \\
0 & 0 & 1
\end{pmatrix}}
$  &
${
 \begin{pmatrix}
 \dfrac{1}{2}(-1+\sqrt{-3})  & 0 & 0 \\
0 & -\dfrac{1}{2}(1+\sqrt{-3})  & 0 \\
0 & 0 & 1
\end{pmatrix}}$
&
${
\begin{pmatrix}
\dfrac{1}{2}(-1+\sqrt{-3}) \mu & 1 & -\dfrac{1}{2}(1+\sqrt{-3}) \lambda \\
1 & 0 &  \dfrac{1}{2}(-1+\sqrt{-3}) \rho \\
0 &  0  & 1 \\
\end{pmatrix}} $
\\
\hline
\end{tabular}
}

\end{center}

\begin{center}

\scalebox{0.7}{

\begin{tabular}{|c||c||c|}

\hline
&&\\
$M_{B}$& $P_{B'B}$ & $M_{B'}$\\
&&\\[-0.2cm]
\hline
\hline
&&\\[-0.2cm]

${
\begin{pmatrix}
0 & 1 & \mu  \\
1 & \lambda & 0 \\
\rho & 0 & 1
\end{pmatrix}}
$  &
${
 \begin{pmatrix}
 -\dfrac{1}{2}(1+\sqrt{-3}) &  0  & 0 \\
 0 & \dfrac{1}{2}(-1+\sqrt{-3}) & 0 \\
0 & 0 & 1
\end{pmatrix}}$
&
${
\begin{pmatrix}
0 & 1 & \dfrac{1}{2}(-1+\sqrt{-3}) \mu  \\
1 &  \dfrac{1}{2}(-1+\sqrt{-3}) \lambda & 0  \\
\dfrac{1}{2}(-1+\sqrt{-3}) \rho & 0 & 1 \\
\end{pmatrix}} $
\\

&&\\[-0.2cm]
\hline
\hline
&&\\[-0.2cm]
${
\begin{pmatrix}
0 & 1 & \mu  \\
1 & \lambda & 0 \\
\rho & 0 & 1
\end{pmatrix}}
$  &
${
 \begin{pmatrix}
 \dfrac{1}{2}(-1+\sqrt{-3})  & 0 & 0 \\
0 & -\dfrac{1}{2}(1+\sqrt{-3})  & 0 \\
0 & 0 & 1
\end{pmatrix}}$
&
${
\begin{pmatrix}
0 & 1 & -\dfrac{1}{2}(1+\sqrt{-3}) \mu  \\
1 &  -\dfrac{1}{2}(1+\sqrt{-3}) \lambda &  0 \\
-\dfrac{1}{2}(1+\sqrt{-3}) \rho &  0  & 1 \\
\end{pmatrix}} $
\\
\hline
\end{tabular}
}

\end{center}

\begin{center}

\scalebox{0.7}{

\begin{tabular}{|c||c||c|}

\hline
&&\\
$M_{B}$& $P_{B'B}$ & $M_{B'}$\\
&&\\[-0.2cm]
\hline
\hline
&&\\[-0.2cm]

${
\begin{pmatrix}
0 & 1 & \mu   \\
1 & 0 & \lambda \\
\rho & 0 & 1
\end{pmatrix}}
$  &
${
 \begin{pmatrix}
 -\dfrac{1}{2}(1+\sqrt{-3}) &  0  & 0 \\
 0 & \dfrac{1}{2}(-1+\sqrt{-3}) & 0 \\
0 & 0 & 1
\end{pmatrix}}$
&
${
\begin{pmatrix}
0 & 1 & \dfrac{1}{2}(-1+\sqrt{-3}) \mu  \\
1 & 0 &  -\dfrac{1}{2}(1+\sqrt{-3}) \lambda \\
\dfrac{1}{2}(-1+\sqrt{-3}) \rho & 0 & 1  \\
\end{pmatrix}} $
\\

&&\\[-0.2cm]
\hline
\hline
&&\\[-0.2cm]
${
\begin{pmatrix}
0 & 1 & \mu   \\
1 & 0 & \lambda \\
\rho & 0 & 1
\end{pmatrix}}
$  &
${
 \begin{pmatrix}
 \dfrac{1}{2}(-1+\sqrt{-3})  & 0 & 0 \\
0 & -\dfrac{1}{2}(1+\sqrt{-3})  & 0 \\
0 & 0 & 1
\end{pmatrix}}$
&
${
\begin{pmatrix}
0 & 1 & -\dfrac{1}{2}(1+\sqrt{-3}) \mu  \\
1 & 0 &  \dfrac{1}{2}(-1+\sqrt{-3}) \lambda \\
 -\dfrac{1}{2}(1+\sqrt{-3}) \rho &  0  & 1 \\
\end{pmatrix}} $
\\
\hline
\end{tabular}
}

\end{center}

\begin{center}

\scalebox{0.7}{

\begin{tabular}{|c||c||c|}

\hline
&&\\
$M_{B}$& $P_{B'B}$ & $M_{B'}$\\
&&\\[-0.2cm]
\hline
\hline
&&\\[-0.2cm]

${
\begin{pmatrix}
\mu & \lambda &  1 \\
1 & \rho & 0 \\
0 & 1 & 0
\end{pmatrix}}
$  &
${
 \begin{pmatrix}
-{\phi} & 0 & 0 \\
0 & {\phi}^2 & 0 \\
0 & 0 & {\phi}^4
\end{pmatrix}}$
&
${
\begin{pmatrix}
 -{\phi} \mu & -{\phi}^3 \lambda & 1   \\
1 & {\phi}^2 \rho & 0 \\
0  & 1 & 0
\end{pmatrix}} $
\\
&&\\[-0.2cm]
\hline
\hline
&&\\[-0.2cm]
${
\begin{pmatrix}
\mu & \lambda &  1 \\
1 & \rho & 0 \\
0 & 1 & 0
\end{pmatrix}}
$  &
${
 \begin{pmatrix}
{\phi}^2 & 0 & 0 \\
0 & {\phi}^4 & 0 \\
0 & 0 & -{\phi}
\end{pmatrix}}$
&
${
\begin{pmatrix}
 {\phi}^2 \mu  & {\phi}^6 \lambda & 1 \\
1 & {\phi}^4 \rho & 0 \\
0 & 1 & 0
\end{pmatrix}} $
\\
\hline
\end{tabular}
}

\end{center}

\begin{center}

\scalebox{0.7}{

\begin{tabular}{|c||c||c|}

\hline
&&\\
$M_{B}$& $P_{B'B}$ & $M_{B'}$\\
&&\\[-0.2cm]
\hline
\hline
&&\\[-0.2cm]

${
\begin{pmatrix}
0 & \mu &  1 \\
1 & 0 & \lambda \\
\rho & 1 & 0
\end{pmatrix}}
$  &
${\footnotesize{
 \begin{pmatrix}
 -{\phi} & 0 & 0  \\
 0 & {\phi}^2 & 0  \\
 0 & 0  & {\phi}^4
\end{pmatrix}}}$
&
${\footnotesize{
\begin{pmatrix}
 0 & -{\phi}^3 \mu   & 1   \\
1 & 0 &  {\phi}^6 \lambda \\
-{\phi}^5 \rho  & 1 & 0
\end{pmatrix}} }$
\\
&&\\[-0.2cm]
\hline
\hline
&&\\[-0.2cm]
${
\begin{pmatrix}
0 & \mu &  1 \\
1 & 0 & \lambda \\
\rho & 1 & 0
\end{pmatrix}}
$  &
${\footnotesize{
 \begin{pmatrix}
{\phi}^2 & 0 & 0  \\
0 & {\phi}^4 & 0  \\
 0 & 0 & -{\phi}
\end{pmatrix}}}$
&
${\footnotesize{
\begin{pmatrix}
 0 & {\phi}^6 \mu   & 1   \\
1 & 0 &  -{\phi}^5 \lambda \\
-{\phi}^3 \rho  & 1 & 0
\end{pmatrix}} }$
\\
\hline
\end{tabular}
}

\end{center}

\begin{center}

\scalebox{0.7}{

\begin{tabular}{|c||c||c|}

\hline
&&\\
$M_{B}$& $P_{B'B}$ & $M_{B'}$\\
&&\\[-0.2cm]
\hline
\hline
&&\\[-0.2cm]

${
\begin{pmatrix}
0 & \mu &  1 \\
1 & 0 & \lambda \\
\rho & 1 & 0
\end{pmatrix}}
$  &
${\footnotesize{
 \begin{pmatrix}
0 & 0 & \dfrac{1}{\sqrt[7]{\mu \lambda^2 \rho^4}} \\
0 &  \dfrac{1}{\sqrt[7]{\mu^4 \lambda \rho^2}} & 0 \\
\dfrac{1}{\sqrt[7]{\mu^2 \lambda^4 \rho}} & 0 & 0
\end{pmatrix}}}$
&
${\footnotesize{
\begin{pmatrix}
 0 & \sqrt[7]{\dfrac{\lambda^2}{\mu^6 \rho^3}}   & 1   \\
1 & 0 &  \sqrt[7]{\dfrac{\mu^2}{\lambda^3 \rho^6}} \\
\sqrt[7]{\dfrac{\rho^2}{\mu^3 \lambda^6}}  & 1 & 0
\end{pmatrix}} }$
\\

&&\\[-0.2cm]
\hline
\hline
&&\\[-0.2cm]
${
\begin{pmatrix}
0 & \mu &  1 \\
1 & 0 & \lambda \\
\rho & 1 & 0
\end{pmatrix}}
$  &
${\footnotesize{
 \begin{pmatrix}
0 & 0 & \dfrac{{\phi}^2}{\sqrt[7]{\mu \lambda^2 \rho^4}} \\
0 &  \dfrac{-{\phi}}{\sqrt[7]{\mu^4 \lambda \rho^2}} & 0 \\
\dfrac{{\phi}^4}{\sqrt[7]{\mu^2 \lambda^4 \rho}} & 0 & 0
\end{pmatrix}}}$
&
${\footnotesize{
\begin{pmatrix}
 0 & -\sqrt[7]{\dfrac{{\phi}^5\lambda^2}{\mu^6 \rho^3}}   & 1   \\
1 & 0 &  -\sqrt[7]{\dfrac{{\phi}^3\mu^2}{\lambda^3 \rho^6}} \\
\sqrt[7]{\dfrac{{\phi}^6\rho^2}{\mu^3 \lambda^6}}  & 1 & 0
\end{pmatrix}} }$
\\
\hline
\end{tabular}
}

\end{center}

\begin{center}

\scalebox{0.7}{

\begin{tabular}{|c||c||c|}

\hline
&&\\
$M_{B}$& $P_{B'B}$ & $M_{B'}$\\
&&\\[-0.2cm]
\hline
\hline
&&\\[-0.2cm]

${
\begin{pmatrix}
0 & \mu &  1 \\
1 & 0 & \lambda \\
\rho & 1 & 0
\end{pmatrix}}
$  &
${\footnotesize{
 \begin{pmatrix}
 0 & \dfrac{1}{\sqrt[7]{\mu \lambda^2 \rho^4}} & 0 \\
  \dfrac{1}{\sqrt[7]{\mu^4 \lambda \rho^2}} & 0 & 0 \\
 0 & 0 & \dfrac{1}{\sqrt[7]{\mu^2 \lambda^4 \rho}}
\end{pmatrix}}}$
&
${\footnotesize{
\begin{pmatrix}
 0 & \sqrt[7]{\dfrac{\mu^2}{\lambda^3 \rho^6}}   & 1   \\
1 & 0 &  \sqrt[7]{\dfrac{\rho^2}{\mu^3 \lambda^6}} \\
\sqrt[7]{\dfrac{\lambda^2}{\mu^6 \rho^3}}  & 1 & 0
\end{pmatrix}} }$
\\

&&\\[-0.2cm]
\hline
\hline
&&\\[-0.2cm]
${
\begin{pmatrix}
0 & \mu &  1 \\
1 & 0 & \lambda \\
\rho & 1 & 0
\end{pmatrix}}
$  &
${\footnotesize{
 \begin{pmatrix}
0 & \dfrac{{\phi}^2}{\sqrt[7]{\mu \lambda^2 \rho^4}} & 0  \\
\dfrac{-{\phi}}{\sqrt[7]{\mu^4 \lambda \rho^2}} & 0 & 0 \\
0 & 0 & \dfrac{{\phi}^4}{\sqrt[7]{\mu^2 \lambda^4 \rho}}
\end{pmatrix}}}$
&
${\footnotesize{
\begin{pmatrix}
 0 & -\sqrt[7]{\dfrac{{\phi}^3 \mu^2}{\lambda^3 \rho^6}}   & 1   \\
1 & 0 &  \sqrt[7]{\dfrac{{\phi} ^6\rho^2}{\mu^3 \lambda^6}} \\
-\sqrt[7]{\dfrac{{\phi}^5 \lambda^2}{\mu^6 \rho^3}}  & 1 & 0
\end{pmatrix}} }$
\\
\hline
\end{tabular}
}

\end{center}

\begin{center}

\scalebox{0.7}{

\begin{tabular}{|c||c||c|}

\hline
&&\\
$M_{B}$& $P_{B'B}$ & $M_{B'}$\\
&&\\[-0.2cm]
\hline
\hline
&&\\[-0.2cm]

${
\begin{pmatrix}
0 & \mu &  1 \\
1 & 0 & \lambda \\
\rho & 1 & 0
\end{pmatrix}}
$  &
${\footnotesize{
 \begin{pmatrix}
 \dfrac{1}{\sqrt[7]{\mu \lambda^2 \rho^4}} & 0 & 0 \\
0 & 0 &  \dfrac{1}{\sqrt[7]{\mu^4 \lambda \rho^2}}  \\
 0 & \dfrac{1}{\sqrt[7]{\mu^2 \lambda^4 \rho}} & 0
\end{pmatrix}}}$
&
${\footnotesize{
\begin{pmatrix}
 0 & \sqrt[7]{\dfrac{\rho^2}{\mu^3 \lambda^6}}   & 1   \\
1 & 0 &  \sqrt[7]{\dfrac{\lambda^2}{\mu^6 \rho^3}} \\
\sqrt[7]{\dfrac{\mu^2}{\lambda^3 \rho^6}}  & 1 & 0
\end{pmatrix}} }$
\\

&&\\[-0.2cm]
\hline
\hline
&&\\[-0.2cm]
${
\begin{pmatrix}
0 & \mu &  1 \\
1 & 0 & \lambda \\
\rho & 1 & 0
\end{pmatrix}}
$  &
${\footnotesize{
 \begin{pmatrix}
\dfrac{{\phi}^2}{\sqrt[7]{\mu \lambda^2 \rho^4}} & 0  & 0 \\
0 & 0 & \dfrac{-{\phi}}{\sqrt[7]{\mu^4 \lambda \rho^2}} \\
0 &  \dfrac{{\phi}^4}{\sqrt[7]{\mu^2 \lambda^4 \rho}}& 0
\end{pmatrix}}}$
&
${\footnotesize{
\begin{pmatrix}
 0 & \sqrt[7]{\dfrac{{\phi}^6 \rho^2}{\mu^3 \lambda^6}}   & 1   \\
1 & 0 &  -\sqrt[7]{\dfrac{{\phi}^5 \lambda^2}{\mu^6 \rho^3}} \\
-\sqrt[7]{\dfrac{{\phi}^3  \mu^2}{\lambda^3 \rho^6}}  & 1 & 0
\end{pmatrix}} }$
\\
\hline
\end{tabular}
}

\end{center}

\begin{center}

\scalebox{0.7}{

\begin{tabular}{|c||c||c|}

\hline
&&\\
$M_{B}$& $P_{B'B}$ & $M_{B'}$\\
&&\\[-0.2cm]
\hline
\hline
&&\\[-0.2cm]

${
\begin{pmatrix}
0 & \mu &  1 \\
1 & 0 & \lambda \\
\rho & 1 & 0
\end{pmatrix}}
$  &
${\footnotesize{
 \begin{pmatrix}
0 & 0 & -{\phi}  \\
 {\phi}^2 & 0 & 0  \\
0  & {\phi}^4 & 0
\end{pmatrix}}}$
&
${\footnotesize{
\begin{pmatrix}
 0 & {\phi}^6 \lambda   & 1   \\
1 & 0 &  -{\phi}^5 \rho \\
-{\phi}^3 \mu  & 1 & 0
\end{pmatrix}} }$
\\
&&\\[-0.2cm]
\hline
\hline
&&\\[-0.2cm]
${
\begin{pmatrix}
0 & \mu &  1 \\
1 & 0 & \lambda \\
\rho & 1 & 0
\end{pmatrix}}
$  &
${\footnotesize{
 \begin{pmatrix}
0 & 0 & {\phi}^2  \\
{\phi}^4 & 0 & 0 \\
 0 & -{\phi} & 0
\end{pmatrix}}}$
&
${\footnotesize{
\begin{pmatrix}
 0 & -{\phi}^5 \lambda   & 1   \\
1 & 0 &  -{\phi}^3 \rho \\
{\phi}^6 \mu  & 1 & 0
\end{pmatrix}} }$
\\
\hline
\end{tabular}
}

{\footnotesize {TABLE $21^\prime$.}}

\end{center}

\noindent
{\bf Case 5.} $M_B$ has seven non-zero elements.
\medskip

\noindent
The different cases that we must consider are $\binom{9}{2}= 36$. Every matrix has three entries which are 1 and four non-zero parameters $\delta, \lambda, \mu, \rho$, which must satisfy one of the following conditions, depending on the case we are considering, in order for the matrix to not have zero determinant: $\mu\rho \neq 1$; $\mu\rho+\delta\lambda \neq 1$; $\delta\mu \neq 1$; $\delta\mu + \lambda \rho \neq 1$; $\delta\lambda \neq 1$; $\delta\rho - \delta\lambda\mu \neq 1$; $\delta\rho \neq 1$; $\mu\rho-\delta\lambda\rho \neq 1$.

The number  of mutually non-isomorphic parametric families of evolution algebras is eight, which are listed below.
\bigskip

\newpage

\begin{center}

\scalebox{0.70}{

\begin{tabular}{|c||c||c||c||c||c|}
\hline
&&&&&\\
&(1,2)&(1,3)&(2,3)&(1,2,3)&(1,3,2)\\
&&&&&\\
\hline
\hline
&&&&&\\[-0.2cm]

$\color{ao(english)}{
\begin{pmatrix}
1 & \mu & \lambda \\
\rho & 1 & \delta \\
0  & 0 & 1

\end{pmatrix}}
$  &

$ \begin{pmatrix}
1 & \rho & \delta \\
\mu & 1 & \lambda \\
0  & 0 & 1

\end{pmatrix}$
&
$\color{ao(english)}{
\begin{pmatrix}
1 & 0 & 0 \\
\delta & 1 & \rho \\
\lambda  & \mu & 1

\end{pmatrix} }$
 &

$\color{ao(english)}{
\begin{pmatrix}
1 & \lambda & \mu \\
0 & 1 & 0 \\
\rho  & \delta & 1
\end{pmatrix}}$
&
$
\begin{pmatrix}
1 & 0 & 0 \\
\lambda & 1 & \mu  \\
\delta  & \rho & 1
\end{pmatrix}$
&
$
\begin{pmatrix}
1 & \delta & \rho \\
0 & 1 & 0 \\
\mu  & \lambda & 1
\end{pmatrix}
$ \\
&&&&&\\[-0.2cm]
\hline
\hline
&&&&&\\[-0.2cm]

$\color{ao(english)}{
\begin{pmatrix}
1 & \mu & \lambda  \\
\rho & 1 & 0\\
\delta  & 0 & 1

\end{pmatrix}}
$  &

$\color{ao(english)}{
 \begin{pmatrix}
1 & \rho & 0  \\
\mu & 1 & \lambda \\
0  & \delta & 1

\end{pmatrix}}$
&
$\color{ao(english)}{
\begin{pmatrix}
1 & 0 & \delta \\
0 & 1 & \rho\\
\lambda  & \mu & 1

\end{pmatrix} }$
 &

 $\begin{pmatrix}
1 & \lambda & \mu \\
\delta & 1 & 0 \\
\rho   & 0 & 1
\end{pmatrix}$
&
$
\begin{pmatrix}
1 & \delta & 0 \\
\lambda & 1 & \mu \\
0  & \rho & 1
\end{pmatrix}
$
&
$
\begin{pmatrix}
1 & 0 & \rho \\
0 & 1 & \delta \\
\mu & \lambda  & 1
\end{pmatrix}
$ \\
&&&&&\\[-0.2cm]
\hline
\hline
&&&&&\\[-0.2cm]

$\color{ao(english)}{
\begin{pmatrix}
1 & \mu & \lambda \\
\rho & 1 & 0 \\
0 & \delta & 1

\end{pmatrix}}
$  &

$\color{ao(english)}{
 \begin{pmatrix}
1 & \rho & 0 \\
\mu & 1 & \lambda \\
\delta  & 0 & 1

\end{pmatrix}}$
&
$\color{ao(english)}{
\begin{pmatrix}
1 & \delta & 0 \\
0 & 1 & \rho \\
\lambda  & \mu & 1

\end{pmatrix} }$
 &

$\color{ao(english)}{
\begin{pmatrix}
1 & \lambda & \mu \\
0 & 1 & \delta \\
\rho  & 0 & 1
\end{pmatrix}}$
&
$\color{ao(english)}{
\begin{pmatrix}
1 & 0 & \delta \\
\lambda & 1 & \mu \\
0  & \rho & 1
\end{pmatrix}}
$
&
$\color{ao(english)}{
\begin{pmatrix}
1 & 0 & \rho \\
\delta & 1 & 0 \\
\mu  & \lambda & 1
\end{pmatrix}}
$ \\
&&&&&\\[-0.2cm]
\hline
\hline
&&&&&\\[-0.2cm]

$\color{ao(english)}{
\begin{pmatrix}
1 & 0 & \mu \\
\lambda & 1 & \rho \\
\delta  & 0 & 1

\end{pmatrix}}
$  &

$\color{ao(english)}{
 \begin{pmatrix}
1 & \lambda & \rho \\
0 & 1 & \mu \\
0  & \delta & 1

\end{pmatrix}}$
&
$\begin{pmatrix}
1 & 0 & \delta \\
\rho & 1 & \lambda \\
\mu  & 0 & 1

\end{pmatrix} $
 &

 $\color{ao(english)}{
 \begin{pmatrix}
1 & \mu & 0 \\
\delta & 1 & 0 \\
\lambda  & \rho & 1
\end{pmatrix}}$
&
$
\begin{pmatrix}
1 & \delta & 0 \\
\mu & 1 & 0 \\
\rho  & \lambda & 1
\end{pmatrix}
$
&
$
\begin{pmatrix}
1 & \rho & \lambda \\
0 & 1 & \delta \\
0  & \mu & 1
\end{pmatrix}
$ \\

&&&&&\\[-0.2cm]
\hline
\hline
&&&&&\\[-0.2cm]

$\color{ao(english)}{
\begin{pmatrix}
 \mu & 1 & \lambda \\
1 & 0 & 0 \\
\rho  & \delta & 1

\end{pmatrix}}
$  &

$\color{ao(english)}{
 \begin{pmatrix}
0 & 1 & 0 \\
1 & \mu & \lambda \\
\delta  & \rho & 1

\end{pmatrix}}$
&
$\color{ao(english)}{
\begin{pmatrix}
1 & \delta & \rho \\
0 & 0 & 1 \\
\lambda  & 1 & \mu

\end{pmatrix} }$
 &

$\color{ao(english)}{
\begin{pmatrix}
\mu & \lambda & 1 \\
 \rho & 1 & \delta \\
1  & 0  & 0
\end{pmatrix}}$
&
$\color{ao(english)}{
\begin{pmatrix}
1 & \rho & \delta \\
\lambda & \mu & 1 \\
0  & 1 & 0
\end{pmatrix}}
$
&
$\color{ao(english)}{
\begin{pmatrix}
0 & 0 & 1 \\
\delta & 1 & \rho \\
1  & \lambda & \mu
\end{pmatrix}}
$ \\
&&&&&\\[-0.2cm]
\hline
\hline
&&&&&\\[-0.2cm]

$\color{ao(english)}{
\begin{pmatrix}
\mu & 1 &  \lambda \\
1 & 0 & \rho \\
\delta  & 0 & 1

\end{pmatrix}}
$  &

$\color{ao(english)}{
 \begin{pmatrix}
0 & 1 & \rho \\
1 & \mu & \lambda \\
0  & \delta & 1

\end{pmatrix}}$
&
$\color{ao(english)}{
\begin{pmatrix}
1 & 0 & \delta \\
\rho & 0 & 1 \\
\lambda  & 1 & \mu

\end{pmatrix} }$
 &

$\color{ao(english)}{
\begin{pmatrix}
\mu & \lambda & 1 \\
\delta & 1 & 0 \\
1  & \rho & 0
\end{pmatrix}}$
&
$\color{ao(english)}{
\begin{pmatrix}
1 & \delta & 0 \\
\lambda & \mu & 1 \\
\rho  & 1 & 0
\end{pmatrix}}
$
&
$\color{ao(english)}{
\begin{pmatrix}
0 & \rho & 1 \\
0 & 1 & \delta \\
1  & \lambda & \mu
\end{pmatrix}}
$ \\
&&&&&\\[-0.2cm]
\hline
\hline
&&&&&\\[-0.2cm]

$\color{ao(english)}{
\begin{pmatrix}
0 & 1 & \mu \\
1 & \lambda & \rho \\
\delta  & 0 & 1

\end{pmatrix}}
$  &

$\color{ao(english)}{
 \begin{pmatrix}
\lambda & 1 & \rho \\
1 & 0 & \mu \\
0  & \delta & 1

\end{pmatrix}}$
&
$\color{ao(english)}{
\begin{pmatrix}
1 & 0 & \delta \\
\rho & \lambda & 1 \\
\mu  & 1 & 0

\end{pmatrix} }$
 &

$\color{ao(english)}{
\begin{pmatrix}
0 & \mu & 1 \\
\delta & 1 & 0 \\
1  & \rho & \lambda
\end{pmatrix}}$
&
$\color{ao(english)}{
\begin{pmatrix}
1 & \delta & 0 \\
\mu & 0 & 1\\
\rho  & 1 & \lambda
\end{pmatrix}}
$
&
$\color{ao(english)}{
\begin{pmatrix}
\lambda & \rho & 1 \\
0 & 1 & \delta \\
1 & \mu  & 0
\end{pmatrix}}
$ \\
&&&&&\\[-0.2cm]
\hline
\hline
&&&&&\\[-0.2cm]

$\color{ao(english)}{
\begin{pmatrix}
0 & 1 & \mu \\
1 & 0 & \lambda \\
\rho  & \delta & 1

\end{pmatrix}}
$  &

$ \begin{pmatrix}
0 & 1 & \lambda \\
1 & 0 & \mu \\
\delta  & \rho & 1

\end{pmatrix}$
&
$\color{ao(english)}{
\begin{pmatrix}
1 & \delta & \rho \\
\lambda & 0 & 1 \\
\mu  & 1 & 0

\end{pmatrix} }$
 &

$\color{ao(english)}{
\begin{pmatrix}
0 & \mu & 1 \\
\rho & 1 & \delta \\
1  & \lambda & 0
\end{pmatrix}}$
&
$
\begin{pmatrix}
1 & \rho & \delta \\
\mu & 0 & 1\\
\lambda  & 1 & 0
\end{pmatrix}
$
&
$
\begin{pmatrix}
0 & \lambda & 1 \\
\delta & 1 & \rho \\
1 & \mu  & 0
\end{pmatrix}
$ \\

&&&&&\\
\hline
\end{tabular}}

{\footnotesize {TABLE 22. ${\rm dim}(A^2)=3$; seven non-zero entries.}} \label{tabla21}

\end{center}

\begin{center}

\scalebox{0.7}{

\begin{tabular}{|c||c||c|}

\hline
&&\\
$M_{B}$& $P_{B'B}$ & $M_{B'}$\\
&&\\[-0.2cm]
\hline
\hline
&&\\[-0.2cm]

${
\begin{pmatrix}
\mu & 1 & \lambda   \\
1 & 0 & 0 \\
\rho & \delta & 1
\end{pmatrix}}
$  &
${
 \begin{pmatrix}
 -\dfrac{1}{2}(1+\sqrt{-3}) &  0  & 0 \\
 0 & \dfrac{1}{2}(-1+\sqrt{-3}) & 0 \\
0 & 0 & 1
\end{pmatrix}}$
&
${
\begin{pmatrix}
-\dfrac{1}{2}(1+\sqrt{-3}) \mu & 1 & \dfrac{1}{2}(-1+\sqrt{-3}) \lambda   \\
1 & 0 & 0  \\
\dfrac{1}{2}(-1+\sqrt{-3}) \rho & -\dfrac{1}{2}(1+\sqrt{-3}) \delta & 1  \\
\end{pmatrix}} $
\\

&&\\[-0.2cm]
\hline
\hline
&&\\[-0.2cm]
${
\begin{pmatrix}
\mu & 1 & \lambda   \\
1 & 0 & 0 \\
\rho & \delta & 1
\end{pmatrix}}
$  &
${
 \begin{pmatrix}
 \dfrac{1}{2}(-1+\sqrt{-3})  & 0 & 0 \\
0 & -\dfrac{1}{2}(1+\sqrt{-3})  & 0 \\
0 & 0 & 1
\end{pmatrix}}$
&
${
\begin{pmatrix}
\dfrac{1}{2}(-1+\sqrt{-3}) \mu &  1 &  -\dfrac{1}{2}(1+\sqrt{-3}) \lambda   \\
1 & 0 & 0  \\
 -\dfrac{1}{2}(1+\sqrt{-3}) \rho &   \dfrac{1}{2}(-1+\sqrt{-3}) \delta  & 1 \\
\end{pmatrix}} $
\\
&&\\[-0.2cm]
\hline
\hline
&&\\[-0.2cm]

${
\begin{pmatrix}
\mu & 1 & \lambda   \\
1 & 0 & \rho \\
\delta & 0 & 1
\end{pmatrix}}
$  &
${
 \begin{pmatrix}
-\dfrac{1}{2}(1+\sqrt{-3}) & 0 & 0 \\
0 & \dfrac{1}{2}(-1+\sqrt{-3}) & 0 \\
0 & 0 & 1
\end{pmatrix}}$
&
${
\begin{pmatrix}
-\dfrac{1}{2}(1+\sqrt{-3}) \mu  & 1 &  \dfrac{1}{2}(-1+\sqrt{-3}) \lambda \\
1 & 0 &   -\dfrac{1}{2}(1+\sqrt{-3}) \rho  \\
\dfrac{1}{2}(-1+\sqrt{-3}) \delta & 0 & 1
\end{pmatrix}} $
\\

&&\\[-0.2cm]
\hline
\hline
&&\\[-0.2cm]
${
\begin{pmatrix}
\mu & 1 & \lambda   \\
1 & 0 & \rho \\
\delta & 0 & 1
\end{pmatrix}}
$  &
${
 \begin{pmatrix}
\dfrac{1}{2}(-1+\sqrt{-3}) & 0 & 0 \\
0 & -\dfrac{1}{2}(1+\sqrt{-3}) & 0 \\
0 & 0 & 1
\end{pmatrix}}$
&
${
\begin{pmatrix}
\dfrac{1}{2}(-1+\sqrt{-3}) \mu & 1 &  -\dfrac{1}{2}(1+\sqrt{-3}) \lambda  \\
1 & 0 &  \dfrac{1}{2}(-1+\sqrt{-3}) \rho \\
-\dfrac{1}{2}(1+\sqrt{-3}) \delta & 0 & 1
\end{pmatrix}} $
\\
&&\\[-0.2cm]
\hline
\hline
&&\\[-0.2cm]
${
\begin{pmatrix}
0 & 1 & \mu \\
1 & \lambda & \rho \\
\delta & 0 & 1
\end{pmatrix}}
$  &
${
 \begin{pmatrix}
0 & -\dfrac{1}{2}(1+\sqrt{-3})  & 0 \\
\dfrac{1}{2}(-1+\sqrt{-3})  & 0 & 0 \\
0 & 0 & 1
\end{pmatrix}}$
&
${
\begin{pmatrix}
0 & 1 & \dfrac{1}{2}(-1+\sqrt{-3}) \mu   \\
1 & \dfrac{1}{2}(-1+\sqrt{-3}) \lambda & -\dfrac{1}{2}(1+\sqrt{-3}) \rho \\
\dfrac{1}{2}(-1+\sqrt{-3}) \delta & 0 & 1
\end{pmatrix}} $
\\
&&\\[-0.2cm]
\hline
\hline
&&\\[-0.2cm]
${
\begin{pmatrix}
0 & 1 & \mu \\
1 & \lambda & \rho \\
\delta & 0 & 1
\end{pmatrix}}
$  &
${
 \begin{pmatrix}
0 & \dfrac{1}{2}(-1+\sqrt{-3})  & 0 \\
-\dfrac{1}{2}(1+\sqrt{-3})  & 0 & 0 \\
0 & 0 & 1
\end{pmatrix}}$
&
${
\begin{pmatrix}
0 & 1 & -\dfrac{1}{2}(1+\sqrt{-3}) \mu    \\
1 & -\dfrac{1}{2}(1+\sqrt{-3}) \lambda & \dfrac{1}{2}(-1+\sqrt{-3}) \rho \\
-\dfrac{1}{2}(1+\sqrt{-3}) \delta & 0 & 1
\end{pmatrix}} $
\\
&&\\
\hline
\end{tabular}
}

\end{center}

\begin{center}

\scalebox{0.7}{

\begin{tabular}{|c||c||c|}

\hline
&&\\
$M_{B}$& $P_{B'B}$ & $M_{B'}$\\
&&\\[-0.2cm]
\hline
\hline
&&\\[-0.2cm]

${
\begin{pmatrix}
0 & 1 & \mu   \\
1 & 0 & \lambda \\
\rho & \delta & 1
\end{pmatrix}}
$  &
${
 \begin{pmatrix}
 -\dfrac{1}{2}(1+\sqrt{-3}) &  0  & 0 \\
 0 & \dfrac{1}{2}(-1+\sqrt{-3}) & 0 \\
0 & 0 & 1
\end{pmatrix}}$
&
${
\begin{pmatrix}
0 & 1 & \dfrac{1}{2}(-1+\sqrt{-3}) \mu \\
1 & 0 & - \dfrac{1}{2}(1+\sqrt{-3}) \lambda   \\
\dfrac{1}{2}(-1+\sqrt{-3}) \rho & -\dfrac{1}{2}(1+\sqrt{-3}) \delta & 1  \\
\end{pmatrix}} $
\\

&&\\[-0.2cm]
\hline
\hline
&&\\[-0.2cm]
${
\begin{pmatrix}
0 & 1 & \mu   \\
1 & 0 & \lambda \\
\rho & \delta & 1
\end{pmatrix}}
$  &
${
 \begin{pmatrix}
 \dfrac{1}{2}(-1+\sqrt{-3})  & 0 & 0 \\
0 & -\dfrac{1}{2}(1+\sqrt{-3})  & 0 \\
0 & 0 & 1
\end{pmatrix}}$
&
${
\begin{pmatrix}
0 & 1 & -\dfrac{1}{2}(1+\sqrt{-3}) \mu \\
1 & 0 &   \dfrac{1}{2}(-1+\sqrt{-3}) \lambda   \\
 -\dfrac{1}{2}(1+\sqrt{-3}) \rho &   \dfrac{1}{2}(-1+\sqrt{-3}) \delta  & 1 \\
\end{pmatrix}} $
\\
&&\\[-0.2cm]
\hline
\hline
&&\\[-0.2cm]

${
\begin{pmatrix}
0 & 1 & \mu   \\
1 & 0 & \lambda \\
\rho & \delta & 1
\end{pmatrix}}
$  &
${
 \begin{pmatrix}
0 & -\dfrac{1}{2}(1+\sqrt{-3}) & 0 \\
\dfrac{1}{2}(-1+\sqrt{-3}) & 0 & 0 \\
0 & 0 & 1
\end{pmatrix}}$
&
${
\begin{pmatrix}
0 & 1 & -\dfrac{1}{2}(1+\sqrt{-3}) \mu  \\
1 & 0 &  \dfrac{1}{2}(-1+\sqrt{-3}) \lambda \\
 -\dfrac{1}{2}(1+\sqrt{-3}) \delta & \dfrac{1}{2}(-1+\sqrt{-3}) \rho &  1
\end{pmatrix}} $
\\

&&\\[-0.2cm]
\hline
\hline
&&\\[-0.2cm]
${
\begin{pmatrix}
0 & 1 & \mu   \\
1 & 0 & \lambda \\
\rho & \delta & 1
\end{pmatrix}}
$  &
${
 \begin{pmatrix}
0 & \dfrac{1}{2}(-1+\sqrt{-3}) & 0 \\
-\dfrac{1}{2}(1+\sqrt{-3}) & 0 & 0 \\
0 & 0 & 1
\end{pmatrix}}$
&
${
\begin{pmatrix}
0 & 1 & \dfrac{1}{2}(-1+\sqrt{-3}) \mu \\
1 & 0 &  -\dfrac{1}{2}(1+\sqrt{-3}) \lambda  \\
\dfrac{1}{2}(-1+\sqrt{-3}) \delta & -\dfrac{1}{2}(1+\sqrt{-3}) \rho & 1
\end{pmatrix}} $
\\
&&\\
\hline
\end{tabular}
}

{\footnotesize {TABLE $22^\prime$.}}
\end{center}

\medskip

\noindent
{\bf Case 6.} There are eight non-zero elements in the matrix.

\noindent
In this case there are only nine possibilities which appear in the table that follows. The condition that the entries of the matrix must satisfy is one of the following: $\eta\lambda+\mu\rho-\delta\eta\mu \neq 1$ or $\delta\mu+\eta\rho-\delta\eta\lambda \neq 1$, just to be sure that the determinant of the corresponding matrix is different from zero.
\medskip

\begin{center}

\scalebox{0.7}{

\begin{tabular}{|c||c||c||c||c||c|}
\hline
&&&&&\\
&(1,2)&(1,3)&(2,3)&(1,2,3)&(1,3,2)\\
&&&&&\\[-0.2cm]
\hline
\hline
&&&&&\\[-0.2cm]

$\color{ao(english)}{
\begin{pmatrix}
1 & \mu & \lambda \\
\rho & 1 & \delta \\
\eta  & 0 & 1

\end{pmatrix}}
$  &

$\color{ao(english)}{
 \begin{pmatrix}
1 & \rho & \delta \\
\mu & 1 & \lambda \\
0  & \eta & 1

\end{pmatrix}}$
&
$\color{ao(english)}{
\begin{pmatrix}
1 & 0 & \eta \\
\delta & 1 & \rho \\
\lambda  & \mu & 1

\end{pmatrix} }$
 &

$\color{ao(english)}{
\begin{pmatrix}
1 & \lambda & \mu \\
\eta & 1 & 0 \\
\rho  & \delta & 1
\end{pmatrix}}$
&
$\color{ao(english)}{
\begin{pmatrix}
1 & \eta & 0 \\
\lambda & 1 & \mu  \\
\delta  & \rho & 1
\end{pmatrix}}$
&
$\color{ao(english)}{
\begin{pmatrix}
1 & \delta & \rho \\
0 & 1 & \eta \\
\mu  & \lambda & 1
\end{pmatrix}}
$ \\
&&&&&\\[-0.2cm]
\hline
\hline
&&&&&\\[-0.2cm]

$\color{ao(english)}{
\begin{pmatrix}
\mu & \lambda & 1 \\
\rho & 1 & \delta \\
1 & \eta   & 0

\end{pmatrix}}
$  &

$ \begin{pmatrix}
1 & \rho  & \delta \\
\lambda & \mu & 1 \\
\eta & 1 & 0

\end{pmatrix}$
&
$\color{ao(english)}{
\begin{pmatrix}
0 &  \eta & 1 \\
\delta & 1 &  \rho \\
1  & \lambda & \mu

\end{pmatrix} }$
 &

$\color{ao(english)}{
\begin{pmatrix}
\mu & 1 & \lambda \\
1 & 0 & \eta  \\
\rho &  \delta & 1
\end{pmatrix}}$
&
$
\begin{pmatrix}
0 & 1 & \eta  \\
1 & \mu & \lambda \\
\delta  &  \rho & 1
 \end{pmatrix}$
&
$
\begin{pmatrix}
1 & \delta & \rho  \\
 \eta & 0 & 1 \\
\lambda  & 1 & \mu
\end{pmatrix}
$ \\
&&&&&\\
\hline
\end{tabular}}

{\footnotesize {TABLE 23. ${\rm dim}(A^2)=3$; eight non-zero entries.}} \label{tabla22}

\end{center}

\newpage

\begin{center}

\scalebox{0.7}{

\begin{tabular}{|c||c||c|}

\hline
&&\\
$M_{B}$& $P_{B'B}$ & $M_{B'}$\\
&&\\[-0.2cm]
\hline
\hline
&&\\[-0.2cm]

${
\begin{pmatrix}
\mu & \lambda & 1   \\
\rho & 1 & \delta \\
1 & \eta & 0
\end{pmatrix}}
$  &
${
 \begin{pmatrix}
 -\dfrac{1}{2}(1+\sqrt{-3}) &  0  & 0 \\
 0 & 1 & 0 \\
 0 & 0 & \dfrac{1}{2}(-1+\sqrt{-3}) & 0 \\
\end{pmatrix}}$
&
${
\begin{pmatrix}
- \dfrac{1}{2}(1+\sqrt{-3}) \mu  &  \dfrac{1}{2}(-1+\sqrt{-3}) \lambda & 1   \\
\dfrac{1}{2}(-1+\sqrt{-3}) \rho & 1 & -\dfrac{1}{2}(1+\sqrt{-3}) \delta  \\
1 & - \dfrac{1}{2}(1+\sqrt{-3}) \eta & 0
\end{pmatrix}} $
\\

&&\\[-0.2cm]
\hline
\hline
&&\\[-0.2cm]
${
\begin{pmatrix}
\mu & \lambda & 1   \\
\rho & 1 & \delta \\
1 & \eta & 0
\end{pmatrix}}
$  &
${
 \begin{pmatrix}
 \dfrac{1}{2}(-1+\sqrt{-3})  & 0 & 0 \\
0 & 0 & 1 \\
0 & 0 & -\dfrac{1}{2}(1+\sqrt{-3})  \\
\end{pmatrix}}$
&
${
\begin{pmatrix}
\dfrac{1}{2}(-1+\sqrt{-3}) \mu &  -\dfrac{1}{2}(1+\sqrt{-3}) \lambda & 1   \\
 -\dfrac{1}{2}(1+\sqrt{-3}) \rho &  1 &  \dfrac{1}{2}(-1+\sqrt{-3}) \delta   \\
 1 & \dfrac{1}{2}(-1+\sqrt{-3}) \eta & 0
\end{pmatrix}} $
\\
&&\\[-0.2cm]
\hline
\hline
&&\\[-0.2cm]

${
\begin{pmatrix}
\mu & \lambda & 1   \\
\rho & 1 & \delta \\
1 & \eta & 0
\end{pmatrix}}
$  &
${
 \begin{pmatrix}
0 & \dfrac{1}{\mu} & 0 \\
\dfrac{1}{\sqrt[3]{\delta \eta^2} } & 0 & 0 \\
0 & 0 & \dfrac{1}{\sqrt[3]{\delta^2 \eta}}
\end{pmatrix}}$
&
${
\begin{pmatrix}
\dfrac{1}{\sqrt[3]{\delta \eta ^2}} & \dfrac{\rho \sqrt[3]{\delta  \eta^2}}{\mu^2} & 1  \\
\dfrac{\mu \lambda}{\eta\sqrt[3]{\delta^2\eta}} & 1 & \dfrac{\mu}{\delta \sqrt[3]{\delta \eta ^2}} \\
 1 & \dfrac{\sqrt[3]{\delta^2 \eta}}{\mu^2} & 0
\end{pmatrix}} $
\\

&&\\[-0.2cm]
\hline
\hline
&&\\[-0.2cm]
${
\begin{pmatrix}
\mu & \lambda & 1   \\
\rho & 1 & \delta \\
1 & \eta & 0
\end{pmatrix}}
$  &
${
 \begin{pmatrix}
0 & \dfrac{1}{\mu} & 0 \\
-\dfrac{{\phi}}{\sqrt[3]{\delta \eta^2} } & 0 & 0 \\
0 & 0 & \dfrac{{\phi}^2}{\sqrt[3]{\delta^2 \eta}}
\end{pmatrix}}$
&
${
\begin{pmatrix}
-\dfrac{{\phi}}{\sqrt[3]{\delta \eta ^2}} & \dfrac{{\phi}^2\rho \sqrt[3]{\delta  \eta^2}}{\mu^2} & 1  \\
\dfrac{{\phi}^2 \mu \lambda}{\eta\sqrt[3]{\delta^2\eta}} & 1 & -\dfrac{{\phi}\mu}{\delta \sqrt[3]{\delta \eta ^2}} \\
 1 & -\dfrac{{\phi}\sqrt[3]{\delta^2 \eta}}{\mu^2} & 0
\end{pmatrix}} $
\\

&&\\[-0.2cm]
\hline
\hline
&&\\[-0.2cm]
${
\begin{pmatrix}
\mu & \lambda & 1   \\
\rho & 1 & \delta \\
1 & \eta & 0
\end{pmatrix}}
$  &
${
 \begin{pmatrix}
0 & \dfrac{1}{\mu} & 0 \\
\dfrac{{\phi}^2}{\sqrt[3]{\delta \eta^2} } & 0 & 0 \\
0 & 0 & -\dfrac{{\phi}}{\sqrt[3]{\delta^2 \eta}}
\end{pmatrix}}$
&
${
\begin{pmatrix}
\dfrac{{\phi}^2}{\sqrt[3]{\delta \eta ^2}} & -\dfrac{{\phi}\rho \sqrt[3]{\delta  \eta^2}}{\mu^2} & 1  \\
-\dfrac{{\phi} \mu \lambda}{\eta\sqrt[3]{\delta^2\eta}} & 1 & \dfrac{{\phi}^2\mu}{\delta \sqrt[3]{\delta \eta ^2}} \\
 1 & \dfrac{{\phi}^2\sqrt[3]{\delta^2 \eta}}{\mu^2} & 0
\end{pmatrix}} $
\\
&&\\
\hline
\end{tabular}
}

{\footnotesize {TABLE $23^\prime$.}}
\end{center}

\medskip

\noindent
{\bf Case 7} All the entries in the matrix are non-zero. In this case only one matrix appears.

\begin{center}

\scalebox{0.7}{

\begin{tabular}{|c||c||c||c||c||c|}

\hline
&&&&&\\
&(1,2)&(1,3)&(2,3)&(1,2,3)&(1,3,2)\\
&&&&&\\[-0.2cm]
\hline
\hline
&&&&&\\[-0.2cm]

$\color{ao(english)}{
\begin{pmatrix}
1 & \mu & \lambda \\
\rho & 1 & \delta \\
\eta  & \tau & 1

\end{pmatrix}}$  &

$
 \begin{pmatrix}
1 & \rho & \delta \\
\mu & 1 & \lambda \\
\tau  & \eta & 1

\end{pmatrix}$
&
$\begin{pmatrix}
1 &\tau & \eta \\
\delta & 1 &  \rho \\
\lambda  & \mu & 1

\end{pmatrix} $
 &

 $\begin{pmatrix}
1 & \lambda & \mu \\
\eta & 1 & \tau \\
\rho  & \delta & 1
\end{pmatrix}$
&
$
\begin{pmatrix}
1 & \eta & \tau \\
\lambda & 1 & \mu \\
\delta & \rho & 1
\end{pmatrix}$
&
$
\begin{pmatrix}
1 & \delta  & \rho \\
\tau & 1 & \eta \\
\mu & \lambda & 1
\end{pmatrix}
$ \\
&&&&&\\
\hline
\end{tabular}}

{\footnotesize {TABLE 24. ${\rm dim}(A^2)=3$; nine non-zero entries.}}\label{tabla23}
\end{center}
\medskip

\noindent
and the condition that the parameters must satisfy is $\eta\rho+\delta\lambda+\mu\tau-\eta\lambda\tau-\delta\mu\rho\neq 1$.
\medskip

\begin{appendix}
\section{}\label{apendice}

In this appendix we include the study of the different matrices that can appear as change of basis matrices for an evolution algebra $A$ such that ${\rm dim}(A)=3$ and ${\rm dim}(A^2)=1$.

We have separated this piece from the proof of Theorem \ref{thm:clasificacion} in order to not enlarge it. We think that it can be of interest as we did a similar study when ${\rm dim}(A^2)=2$  and when ${\rm dim}(A^2)=3$. The notation we use is as in Case ${\rm dim}(A^2)=1$.
\medskip

Let $B'$ be an arbitrary natural basis of $A$ and let $
P_{B'B}=
\begin{pmatrix}
p_{11} &  p_{12} & p_{13} \\
 p_{21} & p_{22} & p_{23} \\
 p_{31} & p_{32} & p_{33}
\end{pmatrix}
$
be the change of basis matrix.
By \eqref{ecuac3}:

\begin{eqnarray}\label{ecuacionesproducto1}
\left\lbrace
\begin{array}{lll}
p_{11}p_{12}a_{11}+p_{21}p_{22}c_1a_{11}+ p_{31}p_{32}c_2a_{11} &=0 \\
p_{11}p_{12}a_{21}+p_{21}p_{22}c_1a_{21}+ p_{31}p_{32}c_2a_{21} &=0 \\
p_{11}p_{12}a_{31}+p_{21}p_{22}c_1a_{31}+ p_{31}p_{32}c_2a_{31} &=0
\end{array}
\right.
\end{eqnarray}

\begin{eqnarray}\label{ecuacionesproducto2}
\left\lbrace
\begin{array}{lll}
p_{11}p_{13}a_{11}+p_{21}p_{23}c_1a_{11}+ p_{31}p_{33}c_2a_{11} &=0 \\
p_{11}p_{13}a_{21}+p_{21}p_{23}c_1a_{21}+ p_{31}p_{33}c_2a_{21} &=0 \\
p_{11}p_{13}a_{31}+p_{21}p_{23}c_1a_{31}+ p_{31}p_{33}c_2a_{31} &=0
\end{array}
\right.
\end{eqnarray}

\begin{eqnarray}\label{ecuacionesproducto3}
\left\lbrace
\begin{array}{lll}
p_{12}p_{13}a_{11}+p_{22}p_{23}c_1a_{11}+ p_{32}p_{33}c_2a_{11} &=0 \\
p_{12}p_{13}a_{21}+p_{22}p_{23}c_1a_{21}+ p_{32}p_{33}c_2a_{21} &=0 \\
p_{12}p_{13}a_{31}+p_{22}p_{23}c_1a_{31}+ p_{32}p_{33}c_2a_{31} &=0
\end{array}
\right.
\end{eqnarray}

\noindent
Since $e_1^2 \neq 0$, there exists $j \in \{1,2,3\}$ such that $a_{j1} \neq 0$. In any case, from \eqref{ecuacionesproducto1}, \eqref{ecuacionesproducto2} and \eqref{ecuacionesproducto3} we have:

\begin{eqnarray}
\label{ep4} p_{11}p_{12}& =-(p_{21}p_{22}c_1+ p_{31}p_{32}c_2);  \\
\label{ep5} p_{11}p_{13}& =-(p_{21}p_{23}c_1+ p_{31}p_{33}c_2);  \\
\label{ep6} p_{12}p_{13}& =-(p_{22}p_{23}c_1+ p_{32}p_{33}c_2).
\end{eqnarray}
\medskip

\noindent
{\bf Case 1} Assume $p_{11}p_{12}p_{13}\neq 0$.
\medskip

\noindent
This implies that $p_{21}p_{22}c_1+ p_{31}p_{32}c_2\neq 0$, $p_{21}p_{23}c_1+ p_{31}p_{33}c_2\neq 0$ and $p_{22}p_{23}c_1+ p_{32}p_{33}c_2 \neq 0$. So, $ p_{11}=\dfrac{-(p_{21}p_{22}c_1+ p_{31}p_{32}c_2)}{p_{12}}$. Replacing this value in \eqref{ep5}, we get $p_{13}=\dfrac{(p_{21}p_{23}c_1+p_{31}p_{33}c_2)p_{12}}{p_{21}p_{22}c_1+p_{31}p_{32}c_2}$. Finally, we replace $p_{13}$ in \eqref{ep6} and we have $p_{12}=\pm \sqrt {\dfrac{-(p_{21}p_{22}c_1+ p_{31}p_{32}c_2)(p_{22}p_{23}c_1+ p_{32}p_{33}c_2)}{p_{21}p_{23}c_1+p_{31}p_{33}c_2}}$.

\noindent
Therefore:

\begin{eqnarray*}
p_{11} &= & -\sqrt {-\dfrac{(p_{21}p_{22}c_1+ p_{31}p_{32}c_2)(p_{21}p_{23}c_1+p_{31}p_{33}c_2)}{p_{22}p_{23}c_1+ p_{32}p_{33}c_2}} \\
p_{12} & = &\sqrt{ {-\dfrac{(p_{21}p_{22}c_1+ p_{31}p_{32}c_2)(p_{22}p_{23}c_1+ p_{32}p_{33}c_2)}{p_{21}p_{23}c_1+p_{31}p_{33}c_2}}} \\
p_{13} & = &\sqrt{-\dfrac{(p_{21}p_{23}c_1+ p_{31}p_{33}c_2)(p_{22}p_{23}c_1+ p_{32}p_{33}c_2)}{p_{21}p_{22}c_1+p_{31}p_{32}c_2}}
\end{eqnarray*}

\noindent
or
\begin{eqnarray*}
p_{11} & = & \sqrt {-\dfrac{(p_{21}p_{22}c_1+ p_{31}p_{32}c_2)(p_{21}p_{23}c_1+p_{31}p_{33}c_2)}{p_{22}p_{23}c_1+ p_{32}p_{33}c_2}} \\
p_{12} & = &-\sqrt{ {-\dfrac{(p_{21}p_{22}c_1+ p_{31}p_{32}c_2)(p_{22}p_{23}c_1+ p_{32}p_{33}c_2)}{p_{21}p_{23}c_1+p_{31}p_{33}c_2}}} \\
p_{13} & = &-\sqrt{-\dfrac{(p_{21}p_{23}c_1+ p_{31}p_{33}c_2)(p_{22}p_{23}c_1+ p_{32}p_{33}c_2)}{p_{21}p_{22}c_1+p_{31}p_{32}c_2}}
\end{eqnarray*}

\noindent
 {\bf Case 2} Suppose $p_{13}=0$ and $p_{11}p_{12} \neq 0$.
\medskip

\noindent
We have the following equations:

\begin{eqnarray}
p_{11}p_{12} & = &- (p_{21}p_{22}c_1+p_{31}p_{32}c_2); \label{apendice1} \\
p_{21}p_{23}c_1 & = &- p_{31}p_{33}c_2;  \label{apendice2} \\
p_{22}p_{23}c_1& =&-p_{32}p_{33}c_2; \label{apendice3}
\end{eqnarray}

\noindent
 {\bf Case 2.1} $p_{31}c_2\neq 0$.
\medskip

\noindent
Necessarily $p_{23}\neq 0$, since otherwise $p_{33}=0$ contradicting the fact to $\vert P_{B'B} \vert \neq 0$. Moreover, $p_{32} \neq 0$. Indeed, if $p_{32}=0$ then or $p_{22}=0$ or $c_1=0$. But, on the other hand we have that $p_{21}p_{22}c_1 \neq 0$. Contradiction. Now, we distinguish between $c_1 =0$ or not.
\medskip

\noindent
 {\bf Case 2.1.1} $c_1=0$.
\medskip

\noindent
As $p_{31}p_{33}c_{2}=0$ necessarily $p_{33}=0$. Then, the change of basis matrix is

$$P_{B'B}=
\begin{pmatrix}
p_{11} &  -\dfrac{p_{31}p_{32}c_2}{p_{11}} & 0 \\
 p_{21} & p_{22} & p_{23} \\
 p_{31} & p_{32} & 0
\end{pmatrix}
$$

\noindent
with $p_{11}p_{12}p_{23}p_{31}p_{32}c_2 \neq 0$ and $p_{11}^2+c_2p_{31}^2 \neq 0$.
\medskip

\noindent
 {\bf Case 2.1.2} $c_1 \neq 0$.
\medskip

\noindent
 {\bf Case 2.1.2.1} $p_{21}=0$.
\medskip

\noindent
As $p_{31}p_{33}c_2 = 0$ necessarily $p_{33} = 0$. So $p_{22}p_{23}c_1 = 0$. Or equivalently $p_{22} =0$. Then
$$P_{B'B}=
\begin{pmatrix}
p_{11} &  -\dfrac{p_{31}p_{32}c_2}{p_{11}} & 0 \\
 0 & 0 & p_{23} \\
 p_{31} & p_{32} & 0
\end{pmatrix}
$$

\noindent
with $p_{11}p_{23}p_{31}p_{32}c_2 \neq 0$ and $p_{11}^2+c_2p_{31}^2 \neq 0$.
\medskip

\noindent
 {\bf Case 2.1.2.2} $p_{21} \neq 0$.
\medskip

\noindent
By \eqref{apendice2} $p_{33}=\dfrac{-p_{21}p_{23}c_1}{p_{31}c_2}$. If we remove $p_{33}$ in \eqref{apendice3} we get that $p_{22}=\dfrac{p_{32}p_{21}}{p_{31}}$. And finally if we replace $p_{22}$ and $p_{33}$ in \eqref{apendice1} we have   $p_{12}=-\dfrac{p_{32}(c_1p_{21}^2+c_2p_{31}^2)}{p_{11}p_{31}}$. Then
$$P_{B'B}=
\begin{pmatrix}
p_{11} &  -\dfrac{p_{32}(p_{21}^2c_1+p_{31}^2c_2)}{p_{11}p_{31}} & 0 \\
& & \\
 p_{21} & \dfrac{p_{32}p_{21}}{p_{31}} & p_{23} \\
 & & \\
p_{31} & p_{32} & \dfrac{-p_{21}p_{23}c_1}{p_{31}c_2}
\end{pmatrix},
$$

\noindent
where $p_{ij}\neq 0 \forall i,j \in {1,2,3} $ except $p_{13}$. Furthermore, $p_{21}^2c_1+p_{31}^2c_2 \neq 0$ and $p_{21}^2c_1+p_{31}^2c_2+p_{11}^2 \neq 0$.
\medskip

\noindent
 {\bf Case 2.2} $c_2 \neq 0$ and $p_{31}=0$.
\medskip

\noindent
The equations \eqref{apendice1}, \eqref{apendice2} and \eqref{apendice3} are as follows:

\begin{eqnarray*}
p_{11}p_{12} & = &- p_{21}p_{22}c_1;  \\
p_{21}p_{23}c_1 & = &0;   \\
p_{22}p_{23}c_1& =&-p_{32}p_{33}c_2.
\end{eqnarray*}

\noindent
As $p_{21}p_{22}c_1 \neq 0$ then necessarily $p_{23}=0$ and so $p_{32}=0$ ($p_{33} \neq 0$ because otherwise $\vert P \vert =0$). Therefore

$$P_{B'B}=
\begin{pmatrix}
p_{11} &  -\dfrac{p_{21}p_{22}c_1}{p_{12}} & 0 \\
 p_{21} & p_{22} & 0 \\
0 & 0 & p_{33}
\end{pmatrix},
$$

\noindent
with $p_{11}p_{12}p_{21}p_{22}p_{33}c_1c_2 \neq 0$ and $p_{11}^2+ p_{21}^2c_1 \neq 0$.
\medskip

\noindent
 {\bf Case 2.3} $c_2=0 $.
\medskip

\noindent
The equations \eqref{apendice1}, \eqref{apendice2} and \eqref{apendice3} turn out

\begin{eqnarray*}
p_{11}p_{12} & = &- p_{21}p_{22}c_1  \\
p_{21}p_{23}c_1 & = &0   \\
p_{22}p_{23}c_1& =&0
\end{eqnarray*}

\noindent
As $p_{21}p_{22}c_1 \neq 0 $ then $p_{23} =0$. Therefore,

$$P_{B'B}=
\begin{pmatrix}
p_{11} &  -\dfrac{p_{21}p_{22}c_1}{p_{11}} & 0 \\
 p_{21} & p_{22} & 0 \\
p_{31} & p_{32} & p_{33}
\end{pmatrix}
$$

\noindent
with $p_{11}p_{21}p_{22}p_{33}c_1 \neq 0$ and $p_{11}^2+ p_{21}^2c_1 \neq 0$.
\medskip

\noindent
  {\bf Case 3} $p_{12}=0$ and $p_{11}p_{13}=0$.
\medskip

\noindent
Reasoning in the same way as Case 1.2, we obtain the following results:
\medskip

\noindent
 {\bf Case 3.1} $p_{31}c_2 \neq 0$.
\medskip

\noindent
Necessarily   $p_{22}p_{33} \neq 0$.
\medskip

\noindent
{\bf Case 3.1.1} $ c_1 \neq 0$.
\medskip

\noindent
 {\bf Case 3.1.1.1} $ p_{21} \neq 0$.
\medskip

\noindent
The change of basis matrix is as follows:

$$P_{B'B}=
\begin{pmatrix}
p_{11} & 0 &   -\dfrac{p_{23}(p_{21}^2c_1+p_{31}^2c_2)}{p_{11}p_{21}}  \\
 p_{21} & \dfrac{-p_{31}p_{32}c_2}{p_{21}c_1} & p_{23} \\
p_{31} & p_{32} &\dfrac{p_{23}p_{31}}{p_{21}}
\end{pmatrix},
$$

\noindent
where $p_{ij}\neq 0 \forall i,j \in {1,2,3} $ except $p_{12}$. Furthermore, $p_{21}^2c_1+p_{31}^2c_2 \neq 0$ and $p_{21}^2c_1+p_{31}^2c_2+p_{11}^2 \neq 0$.
\medskip

\noindent
 {\bf Case 3.1.1.2} $ p_{21}= 0$.
\medskip

\noindent
So $p_{32}=0$ and $p_{23}=0$. And

$$P_{B'B}=
\begin{pmatrix}
p_{11} & 0 &   -\dfrac{p_{31}p_{33}c_2}{p_{11}}  \\
0 & p_{22} & 0 \\
p_{31} & 0 & p_{33}
\end{pmatrix}
$$

\noindent
with $p_{11}p_{22}p_{31}p_{33} c_2 \neq 0$ and $p_{11}^2+p_{31}^2c_2 \neq 0$.
\medskip

\noindent
 {\bf Case 3.1.2} $ c_1=0$.

$$P_{B'B}=
\begin{pmatrix}
p_{11} & 0 &   -\dfrac{p_{31}p_{33}c_2}{p_{11}}  \\
p_{21} & p_{22} & p_{23} \\
p_{31} & 0 & p_{33}
\end{pmatrix},
$$

\noindent
where $p_{11}p_{31}p_{33}c_2 \neq 0$ and $p_{11}^2+p_{31}^2c_2 \neq 0$.
\medskip

\noindent
{\bf Case 3.2} $ p_{31}=0$ and $c_2 \neq 0$.
\medskip

\noindent
Therefore $p_{22}p_{33}=0$ and

$$P_{B'B}=
\begin{pmatrix}
p_{11} & 0 &   -\dfrac{p_{21}p_{23}c_1}{p_{11}}  \\
p_{21} & 0 & p_{23} \\
0 & p_{32} &  0
\end{pmatrix}
$$

\noindent
with $p_{11}p_{13}p_{21}p_{23}p_{32}c_2 \neq 0$ and $p_{11}^2+p_{21}^2c_1 \neq 0$.
\medskip

\noindent
{\bf Case 3.3} $c_2 =0$.
\medskip

Necessarily $p_{22}=0$ and

$$P_{B'B}=
\begin{pmatrix}
p_{11} & 0 &   -\dfrac{p_{21}p_{23}c_1}{p_{11}}  \\
p_{21} & 0 & p_{23} \\
p_{31} & p_{32} &  p_{32}
\end{pmatrix},
$$

\noindent
where $p_{11}p_{21}p_{23}p_{32}c_1 \neq 0$ and $p_{11}^2+c_1p_{21}^2 \neq 0$.
\medskip

\noindent
 {\bf Case 4} $p_{11}=0$ and $p_{12}p_{13}=0$.
\medskip

\noindent
In the same way as in Case 1.2 and Case 1.3 we obtain the following change of basis matrices:
\medskip

\noindent
{\bf Case 4.1} $p_{32}c_2 \neq 0$.
\medskip

\noindent
Then  $p_{21}p_{33} \neq 0$.
\medskip

\noindent
 {\bf Case 4.1.1} $ c_1 \neq 0$.
\medskip

\noindent
  {\bf Case 4.1.1.1} $ p_{22} \neq 0$.
\medskip

\noindent
The change of basis matrix is as follows:

$$P_{B'B}=
\begin{pmatrix}
0 & p_{12} &   -\dfrac{p_{33}(p_{22}^2c_1+p_{32}^2c_2)}{p_{32}p_{12}}  \\
 p_{21} & p_{22}  & \dfrac{p_{22}p_{33}}{p_{32}}  \\
 -\dfrac{p_{21}p_{22}c_1}{p_{32}c_2} & p_{32} & p_{33}
\end{pmatrix},
$$

\noindent
where $p_{ij}\neq 0 \forall i,j \in {1,2,3} $ except $p_{11}$. Furthermore, $p_{22}^2c_1+p_{32}^2c_2 \neq 0$ and $p_{22}^2c_1+p_{32}^2c_2+p_{12}^2 \neq 0$.
\medskip

\noindent
 {\bf Case 4.1.1.2} $ p_{22}= 0$.
\medskip

\noindent
Then $p_{31}=p_{23}=0$. And

$$P_{B'B}=
\begin{pmatrix}
0 & p_{12} &   -\dfrac{p_{32}p_{33}c_2}{p_{12}}  \\
p_{21} & 0 & 0 \\
0 & p_{32} & p_{33}
\end{pmatrix},
$$

\noindent
where $p_{12}p_{21}p_{32}p_{33} c_2 \neq 0$ and $p_{12}^2+p_{32}^2c_2 \neq 0$.
\medskip

\noindent
 {\bf Case 4.1.2} $ c_1=0$.
\medskip

$$P_{B'B}=
\begin{pmatrix}
0 & p_{12} &   -\dfrac{p_{32}p_{33}c_2}{p_{12}}  \\
p_{21} & p_{22} & p_{23} \\
0 & p_{32} & p_{33}
\end{pmatrix}
$$

\noindent
with $(p_{12}p_{32}p_{33}c_2)(p_{12}^2+p_{32}^2c_2) \neq 0$.
\medskip

\noindent
 {\bf Case 4.2} $ p_{32}=0$ and $c_2 \neq 0$.
\medskip

\noindent
So $p_{21}p_{33}=0$ and

$$P_{B'B}=
\begin{pmatrix}
0 & p_{12} &   -\dfrac{p_{22}p_{23}c_1}{p_{12}}  \\
0 & p_{22} & p_{23} \\
p_{31} & 0 &  0
\end{pmatrix},
$$

\noindent
where $p_{12}p_{13}p_{22}p_{23}p_{31}c_2(p_{12}^2+p_{22}^2c_1) \neq 0$.
\medskip

\noindent
 {\bf Case 4.3} $c_2 =0$.
\medskip

\noindent
Then $p_{21}=0$ and

$$P_{B'B}=
\begin{pmatrix}
0 & p_{12} &   -\dfrac{p_{22}p_{23}c_1}{p_{12}}  \\
0 & p_{22} & p_{23} \\
p_{31} & p_{32} &  p_{33}
\end{pmatrix}
$$

\noindent
for $p_{12}p_{22}p_{23}p_{31}c_1 \neq 0$ and $p_{12}^2+c_1p_{22}^2 \neq 0$.
\medskip

\noindent
{\bf Case 5} $p_{1i}=p_{1j}=0$ for some $i,j \in {1,2,3} \, i\neq j$.
\medskip

\noindent
The equations \eqref{apendice1}, \eqref{apendice2} and \eqref{apendice3} are as follows:

\begin{eqnarray*}
p_{21}p_{22}c_1 & = &- p_{31}p_{32}c_2; \\
p_{21}p_{23}c_1 & = &- p_{31}p_{33}c_2;   \\
p_{22}p_{23}c_1& =&-p_{32}p_{33}c_2.
\end{eqnarray*}

\noindent
{\bf Case 5.1} $c_1=c_2=0$.
\medskip

\noindent
If, for example, $p_{12}=p_{13}=0$, the change of basis matrix is

$$P_{B'B}=
\begin{pmatrix}
p_{11} & 0 &  0  \\
p_{21} & p_{22} & p_{23} \\
p_{31} & p_{32} &  p_{33}
\end{pmatrix}
$$

\noindent
with $p_{11}(p_{22}p_{33}-p_{23}p_{32})\neq 0$.

\noindent
In case of $p_{11}=p_{12}=0$, the change of basis matrix is

$$P_{B'B}=
\begin{pmatrix}
0 & 0 &  p_{13} \\
p_{21} & p_{22} & p_{23} \\
p_{31} & p_{32} &  p_{33}
\end{pmatrix}
$$

\noindent
for $p_{13}(p_{21}p_{32}-p_{22}p_{31})\neq 0$.

\noindent
For $p_{11}=p_{13}=0$

$$P_{B'B}=
\begin{pmatrix}
0 & p_{12} &  0 \\
p_{21} & p_{22} & p_{23} \\
p_{31} & p_{32} &  p_{33}
\end{pmatrix},
$$

\noindent
where $p_{12}(p_{23}p_{31}-p_{21}p_{33})\neq 0$.
\medskip

\noindent
{\bf Case 5.2} $c_i=0$ and $c_j \neq 0$ for $i,j \in \{1,2\} \, i \neq j $.
\medskip

\noindent
Necessarily $p_{sk}=p_{sm}=0$ for some $k,m \in \{1,2,3\} \, k \neq m$ and  $s \in \{2,3\}$ depending on $c_2=0$ or $c_1=0$ respectively. We have to take in account that $\vert P \vert \neq 0$ and so there are possibilities that it can not be (those in which the structure matrix has a zero minor of order two). For example $p_{11}=p_{12}=p_{31}=p_{32}=0$ if $c_1=0$ is not available. Therefore there are nine possible change of basis matrices are the following:

\noindent
For $c_1=0$

$$P_{B'B}=
\begin{pmatrix}
0 & p_{12} &  0 \\
p_{21} & p_{22} & p_{23} \\
0 & 0 &  p_{33}
\end{pmatrix}
$$

\noindent
with $p_{12}p_{21}p_{33} \neq 0 $.

$$P_{B'B}=
\begin{pmatrix}
0 & p_{12} &  0 \\
p_{21} & p_{22} & p_{23} \\
p_{31} & 0 & 0
\end{pmatrix}
$$

\noindent
with $p_{12}p_{23}p_{31} \neq 0 $.

$$P_{B'B}=
\begin{pmatrix}
0 & 0 &  p_{13} \\
p_{21} & p_{22} & p_{23} \\
0 & p_{32} & 0
\end{pmatrix}
$$

\noindent
with $p_{13}p_{21}p_{32} \neq 0 $.

$$P_{B'B}=
\begin{pmatrix}
0 & 0 &  p_{13} \\
p_{21} & p_{22} & p_{23} \\
p_{31} & 0 & 0
\end{pmatrix}
$$

\noindent
with $p_{13}p_{22}p_{31} \neq 0 $.

$$P_{B'B}=
\begin{pmatrix}
p_{11} & 0 & 0 \\
p_{21} & p_{22} & p_{23} \\
0 & p_{32} & 0
\end{pmatrix}
$$

\noindent
with $p_{11}p_{23}p_{32} \neq 0 $.

$$P_{B'B}=
\begin{pmatrix}
p_{11} & 0 & 0 \\
p_{21} & p_{22} & p_{23} \\
0 & 0  & p_{33}
\end{pmatrix}
$$

\noindent
with $p_{11}p_{22}p_{33} \neq 0 $.

\noindent
If $c_2=0$

$$P_{B'B}=
\begin{pmatrix}
0 & 0 &  p_{13} \\
0 & p_{22} & 0 \\
p_{31} & p_{32} &  p_{33}
\end{pmatrix},
$$

\noindent
where $p_{13}p_{22}p_{31} \neq 0$.

$$P_{B'B}=
\begin{pmatrix}
0 & 0 &  p_{13} \\
p_{21} & 0 & 0 \\
p_{31} & p_{32} &  p_{33}
\end{pmatrix}
$$

\noindent
with $p_{13}p_{21}p_{32} \neq 0$.

$$P_{B'B}=
\begin{pmatrix}
0 &   p_{12} & 0 \\
0 & 0 & p_{23} \\
p_{31} & p_{32} &  p_{33}
\end{pmatrix}
$$

\noindent
with $p_{12}p_{23}p_{31} \neq 0$.

$$P_{B'B}=
\begin{pmatrix}
0 &   p_{12} & 0 \\
p_{21} & 0 & 0 \\
p_{31} & p_{32} &  p_{33}
\end{pmatrix}
$$

\noindent
for $p_{12}p_{21}p_{33} \neq 0$.

$$P_{B'B}=
\begin{pmatrix}
p_{11} &   0 & 0 \\
0 & p_{22} & 0 \\
p_{31} & p_{32} &  p_{33}
\end{pmatrix}
$$

\noindent
with $p_{11}p_{22}p_{33} \neq 0$.

$$P_{B'B}=
\begin{pmatrix}
p_{11} &   0 & 0 \\
0 & 0 & p_{23} \\
p_{31} & p_{32} &  p_{33}
\end{pmatrix}
$$

\noindent
with $p_{11}p_{23}p_{32} \neq 0$.
\medskip

\noindent
{\bf Case 5.3} $c_1c_2 \neq0$.
\medskip

\noindent
Fix  $i,j \in \{1,2,3\}$ with $i\neq j$ and such that $p_{1i}=p_{1j}=0$.
\medskip

\noindent
{\bf Case 5.3.1} $p_{2i}=0$ ($p_{2j}=0)$.
\medskip

\noindent
Then as $p_{3i}$ ($p_{3j})$ can not be zero, necessarily  $p_{3j}=p_{3k}=0$ ($p_{3i}=p_{3k}=0$) with $k \in \{1,2,3\}$ and $k \neq i,j$. Therefore we have that $p_{2k}=0$ because $p_{2j}$ ($p_{2i}$) is not possible to be zero.
Therefore $p_{1i}p_{2j}p_{3k}\neq 0$ with $i,j,k \in \{1,2,3\}$ and $i \neq j \neq k$. So, in this case the elements of $S_3 \rtimes (\K^\times)^3$ are the change of basis matrices.
\medskip

\noindent
{\bf Case 5.3.2} $p_{2i}p_{2j} \neq 0$.
\medskip

\noindent
We claim that $p_{2k}$ must be zero with $k \in \{1,2,3\}$ and $k \neq i \neq j$. Indeed, if $p_{2k}\neq 0$ then $p_{3s} \neq 0$ for every $s \in \{1,2,3\}$. Then as $p_{2i}=-\dfrac{p_{3i}p_{3j}c_2}{p_{2j}c_1}$ we have that $p_{3k}=\dfrac{p_{2k}p_{3j}}{p_{2j}}$. And finally we obtain that $p_{2j}^2c_1+ p_{3j}^2c_2 =0$, contradicting $\vert P \vert \neq 0$.

So $p_{2k}=0$. As   $p_{2i}p_{2j} \neq 0$ necessarily $p_{3k}=0$. There are three possible change of basis matrices.

$$P_{B'B}=
\begin{pmatrix}
0 &  0 & p_{13} \\
p_{21} & -\dfrac{p_{31}p_{32}c_2}{p_{21}c_1} & 0  \\
p_{31} & p_{32} &  0
\end{pmatrix}
$$

\noindent
with $p_{13}p_{21}p_{31}p_{32} \neq 0$ and $p_{21}^2 c_1 + p_{31}^2c_2 \neq 0$.

$$P_{B'B}=
\begin{pmatrix}
0 &  p_{12} & 0 \\
p_{21} & 0 & -\dfrac{p_{31}p_{33}c_2}{p_{21}c_1}   \\
p_{31} & 0 & p_{33}
\end{pmatrix}
$$

\noindent
for $p_{12}p_{21}p_{31}p_{33} \neq 0$ and $p_{21}^2 c_1 + p_{31}^2c_2 \neq 0$.

$$P_{B'B}=
\begin{pmatrix}
p_{11} &  0 & 0  \\
0 & p_{22} & -\dfrac{p_{32}p_{33}c_2}{p_{22}c_1}   \\
0 & p_{32} & p_{33}
\end{pmatrix},
$$

\noindent
where $p_{11}p_{22}p_{32}p_{33} \neq 0$ and $p_{22}^2 c_1 + p_{32}^2c_2 \neq 0$.

\section[B]{}\label{apendiceb}
\includegraphics[page=1, width=\textwidth, height=\textheight]{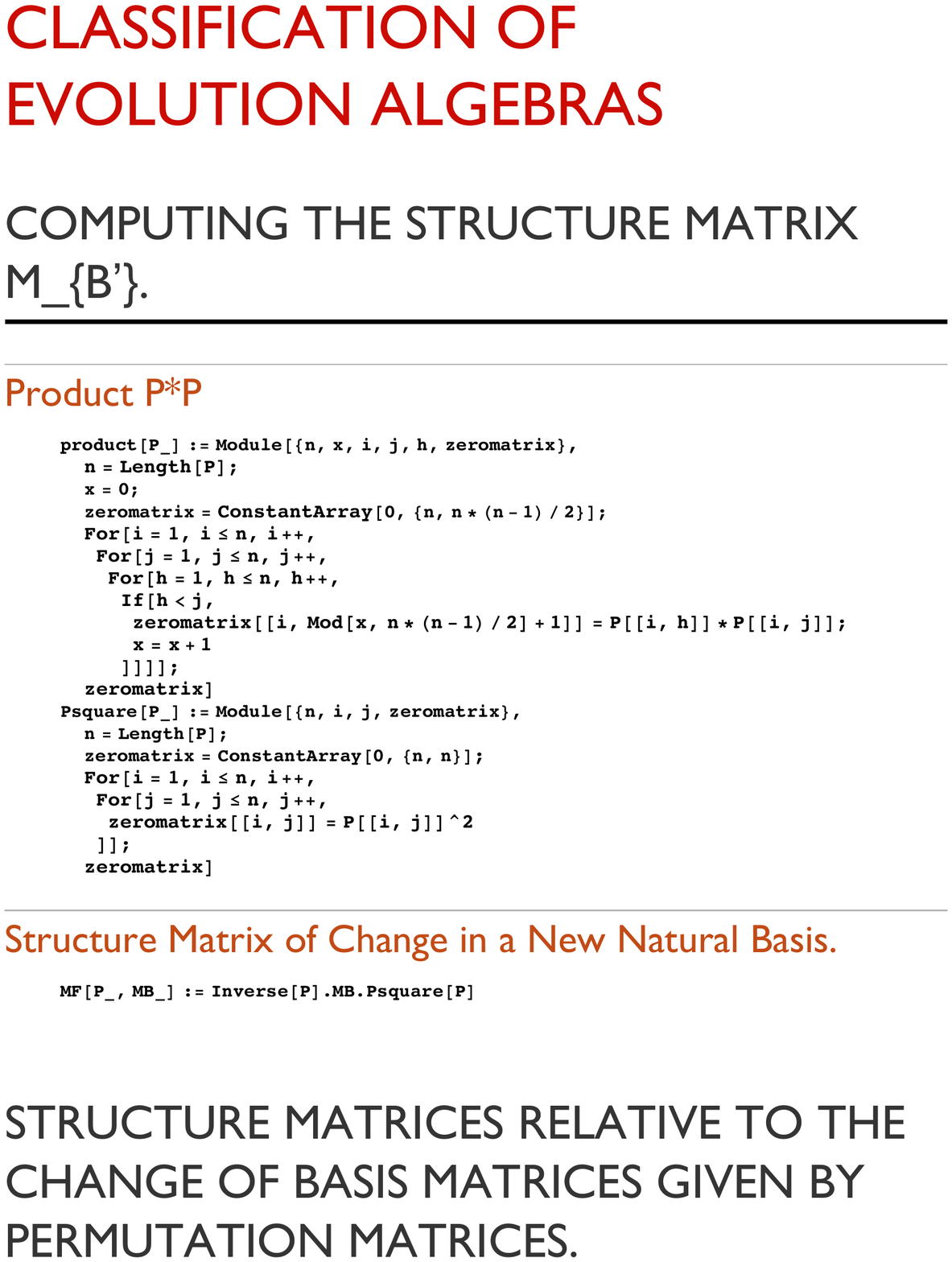}
\newpage

\includegraphics[page=2, width=\textwidth, height=\textheight]{isomorphismprogram.pdf}
\newpage

\includegraphics[page=3, width=\textwidth, height=\textheight]{isomorphismprogram.pdf}
\end{appendix}

\section*{Acknowledgments}
The three authors are supported by the Junta de Andaluc\'{\i}a and Fondos FEDER, jointly, through projects FQM-199 (the third author) and FQM-336 and FQM-7156 (the first and second authors). The two first authors are also supported by the Spanish Ministerio de Econom\'ia y Competitividad and Fondos FEDER, jointly, through project  MTM2013-41208-P.

The authors would like to thank all the participants in the M\'alaga permanent seminar ``Estructuras no asociativas y \'algebras de caminos de Leavitt" for useful discussions during the preparation of this paper, in particular to Prof. C\'andido Mart\'\i n Gonz\'alez for his helpful comments.

\end{document}